\documentclass{amsart}
\usepackage{verbatim}
\usepackage[textsize=scriptsize]{todonotes}

\usepackage{etoolbox}
\newtoggle{draft}

\usepackage{tikz}
\usetikzlibrary{matrix,arrows}
\usepackage{tikz-cd}

\usepackage[margin=1in]{geometry}
\usepackage{amsmath}
\usepackage{amssymb}
\usepackage{amsthm}
\usepackage{amscd}
\usepackage{enumerate}
\usepackage[pdfusetitle,unicode,hidelinks]{hyperref}
\usepackage{bbm}
\usepackage{etoolbox}

\usepackage[utf8]{inputenc}

\newcommand{\tomemail}{\href{mailto:tom.bachmann@zoho.com}{tom.bachmann@zoho.com}}

\newtheorem{proposition}{Proposition}
\newtheorem{corollary}[proposition]{Corollary}
\newtheorem{lemma}[proposition]{Lemma}
\newtheorem{theorem}[proposition]{Theorem}

\newtheorem*{conjecture*}{Conjecture}
\newtheorem*{theorem*}{Theorem}
\newtheorem*{corollary*}{Corollary}
\newtheorem*{proposition*}{Proposition}
\newtheorem*{lemma*}{Lemma}
\theoremstyle{definition}
\newtheorem{definition}[proposition]{Definition}
\newtheorem{construction}[proposition]{Construction}

\newtheorem*{definition*}{Definition}
\newtheorem*{construction*}{Construction}
\theoremstyle{remark}
\newtheorem{remark}[proposition]{Remark}
\newtheorem*{remark*}{Remark}

\newtheorem{example}[proposition]{Example}
\newtheorem*{example*}{Example}

\newcommand{\id}{\operatorname{id}}
\newcommand{\Z}{\mathbb{Z}}

\newcommand{\N}{\mathbb{N}}

\newcommand{\Spec}{\mathrm{Spec}}

\let\scr=\mathcal
\let\bb=\mathbb
\newcommand{\Gm}{{\mathbb{G}_m}}

\def\A{\bb A}
\def\P{\bb P}

\newcommand{\1}{\mathbbm{1}}

\newcommand{\veff}{{\text{veff}}}

\newcommand{\SH}{\mathcal{SH}}

\DeclareMathOperator*{\colim}{colim}

\let\lim=\relax
\DeclareMathOperator*{\lim}{lim}
\def\Map{\mathrm{Map}}

\def\CMon{\mathrm{CMon}}

\def\PSh{\mathcal{P}}

\def\Span{\mathrm{Span}}
\def\Cor{\mathrm{Corr}}

\def\Spc{\mathcal{S}\mathrm{pc}{}}
\def\Fin{\cat F\mathrm{in}}
\def\Fun{\mathrm{Fun}}
\def\End{\mathrm{End}}

\newcommand{\wequi}{\simeq}

\def\adj{\leftrightarrows}

\DeclareRobustCommand{\ul}{\underline}

\newcommand{\tr}{\mathrm{tr}}
\newcommand{\Hom}{\operatorname{Hom}}
\newcommand{\iHom}{\ul{\operatorname{Hom}}}
\def\op{\mathrm{op}}

\let\cat=\mathrm
\def\Sm{{\cat{S}\mathrm{m}}}

\def\Nis{\mathrm{Nis}}
\def\Zar{\mathrm{Zar}}

\def\mot{\mathrm{mot}}

\newcommand{\fr}{\mathrm{fr}}

\def\ph{\mathord-}

\numberwithin{proposition}{section}

\def\gc{\mathrm{gp}}
\newcommand{\Sp}{\mathcal{S}\mathrm{p}}
\newcommand{\stab}{\mathrm{st}}
\newcommand{\Shv}{\mathcal{S}\mathrm{hv}}
\newcommand{\efr}{\mathrm{efr}}
\newcommand{\Ab}{\mathrm{Ab}}

\newcommand{\SmAff}{\mathrm{SmAff}}

\iftoggle{draft} {
\newcommand{\future}[1]{\todo[color=blue!40]{#1}}
\newcommand{\tom}[1]{\todo[color=green!40]{#1}}
\newcommand{\elden}[1]{\todo[color=yellow!40]{#1}}
\newcommand{\TODO}[1]{\todo[color=red]{#1}}
\newcommand{\NB}[1]{\todo[color=gray!40]{#1}}
}{ 
\newcommand{\future}[1]{}
\newcommand{\tom}[1]{}
\newcommand{\elden}[1]{}
\newcommand{\TODO}[1]{}
\newcommand{\NB}[1]{}
\renewcommand{\todo}[1]{}
}

\title{Notes on motivic infinite loop space theory}
\date{\today}

\author{Tom Bachmann}
\address{Department of Mathematics, Massachusetts Institute of Technology, Cambridge, MA, USA}
\email{\tomemail}
\author{Elden Elmanto}
\address{Department of Mathematics, Harvard University, Cambridge, MA, USA}
\email{\href{mailto:elmanto@math.harvard.edu}{elmanto@math.harvard.edu}}

\begin{document}

\maketitle

\begin{abstract}
In fall of 2019, the Thursday Seminar at Harvard University studied motivic infinite loop space theory.
As part of this, the authors gave a series of talks outlining the main theorems of the theory, together with their proofs, in the case of infinite perfect fields.
In winter of 2021/2, the first author taught a topics course at LMU Munich on strict $\A^1$-invariance of framed presheaves (which is one of the main theorems, but was not covered in detail during the Thursday Seminar).
These are our extended notes on these topics.
\end{abstract}

\tableofcontents

\section{Introduction}
We shall assume knowledge of the basic notions of unstable motivic homotopy theory; see e.g. \cite[\S2.2]{bachmann-norms} for a review and \cite{antieau2017primer} for an introduction.
We shall also use freely the language of $\infty$-categories as set out in \cite{HTT,HA}.

Given a base scheme $S$, we thus have the presentably symmetric monoidal $\infty$-category $\Spc(S)$ of \emph{motivic spaces}, and a functor $\Sm_S \to \Spc(S)$ which preserves finite products (and finite coproducts).
We write $\Spc(S)_* = \Spc(S)_{*/}$ for the presentably symmetric monoidal $\infty$-category of pointed motivic spaces; we use the smash product symmetric monoidal structure.
Let $\P^1_S$ be pointed at $1$; this defines an object of $\Spc(S)_*$.
We write $\Sigma^\infty: \Spc(S)_* \to \SH(S)$ for the universal presentably symmetric monoidal $\infty$-category under $\Spc(S)_*$ in which $\P^1$ becomes $\otimes$-invertible.
Denote by $\SH(S)^\veff \subset \SH(S)$ the closure under colimits of the essential image of the functor $\Sigma^\infty$.

The aim of motivic infinite loop space theory is to describe the category $\SH(S)^\veff$.
It turns out that there is a good answer to this problem if $S=\Spec(k)$, where $k$ is a perfect field.
This uses the notion of \emph{framed transfers}, first discovered by Voevodsky \cite{voevodsky-framed}.
The theory was taken up, and many important results proved, by Garkusha--Panin \cite{garkusha2014framed} and their numerous collaborators; see e.g. \cite{agp,hitr,DruzhininPanin,druzhinin2018framed,gnp}.

Their results were not as complete as one might hope for.
The main reason for this is a deficiency in the interaction between Voevodsky's framed correspondences and products of varieties.
This problem was overcome by Hoyois--Khan--Sosnilo--Yakerson and the second author in \cite{EHKSY1}; their main contribution is the invention of the notion of \emph{tangentially framed correspondences} and an accompanying symmetric monoidal $\infty$-category $\Cor^\fr(k)$.

Using this category, motivic infinite loop space theory can be stated as follows.
\begin{theorem} \label{thm:main}
For a perfect field $k$, there exists a canonical, symmetric monoidal equivalence of $\infty$-categories \[ \Spc^\fr(k)^\gc \wequi \SH(k)^\veff. \]
\end{theorem}
Here $\Spc^\fr(k)$ is a category obtained from $\Cor^\fr(k)$ by the usual procedure (consisting of sifted-free cocompletion and motivic localization); it is semiadditive and $\Spc^\fr(k)^\gc$ denotes its subcategory of grouplike objects.

The principal aim of these notes is to explain how to prove this theorem, assuming that $k$ is infinite.
Our secondary aim is to reformulate some of the technical results of \cite{agp,hitr,DruzhininPanin,druzhinin2018framed,gnp} (those that we need in order to prove Theorem \ref{thm:main}) in the language of $\infty$-categories.
As it turns out, this simplifies many of the statements and also many of the proofs.
Given this focus, we do not treat here the construction of the category $\Cor^\fr(k)$ and we refer freely to \cite{EHKSY1} for this and many basic results about framed motivic spaces.

\subsection*{Organization}
The proof of Theorem \ref{thm:main} consists mainly in two steps.
Firstly we show that there is an equivalence $\SH^\fr(k) \wequi \SH(k)$; here $\SH^\fr(k)$ is obtained from $\Spc^\fr(k)$ by inverting the framed motivic space corresponding to $\P^1$.
This is known as the \emph{reconstruction theorem}.
Then we show that the canonical functor $\Spc^\fr(k) \to \SH^\fr(k) \wequi \SH(k)$ is fully faithful.
This is called the \emph{cancellation theorem}.

In \S\ref{sec:reconstruction} we prove the reconstruction theorem modulo a technical result, known as the \emph{cone theorem}.
We then spend all of \S\ref{sec:cone} on proving the cone theorem.
In \S\ref{sec:cancellation} we prove the cancellation theorem, modulo strict $\A^1$-invariance of framed presheaves.
In \S\ref{sec:strict-A1} we prove strict $\A^1$-invariance (of $\A^1$-invariant framed presheaves over infinite perfect fields).

\subsection*{Acknowledgments} We would like to thank the participants of the Thursday seminar who made the experience educational, enjoyable and lively, especially those who gave talks --- Dexter Chua, Jeremy Hahn, Peter Haine, Mike Hopkins, Dylan Wilson, and Lucy Yang. We would additionally like to thank Andrei Druzhinin for useful discussions around the cone theorem and Håkon Kolderup for discussions about the cancellation theorem.

\section{The reconstruction theorem}
\label{sec:reconstruction}
Primary sources: \cite{EHKSY1,garkusha2014framed}.

\subsection{Setup}
Let $S$ be a scheme.
Recall from \cite[\S4]{EHKSY1} that there is a symmetric monoidal, semiadditive $\infty$-category $\Cor^\fr(S)$ and a symmetric monoidal functor $\gamma: \Sm_{S+} \to \Cor^\fr(S)$.\footnote{Recall that for a category with finite coproducts and a final object $*$, $\scr C_+ \subset \scr C_{*/}$ denotes the subcategory on objects of the form $c \coprod *$. We mainly use this in conjunction with the equivalence $\PSh_\Sigma(\scr C_+) \wequi \PSh_\Sigma(\scr C)_*$ \cite[Lemma 2.1]{bachmann-norms}.}
It preserves finite coproducts \cite[Lemma 3.2.6]{EHKSY1} and is essentially surjective (by construction); we refer the reader to \cite[3.2.2]{EHKSY1} for the most important properties.
We denote by $\gamma^*: \PSh_\Sigma(\Sm_{S+}) \to \PSh_\Sigma(\Cor^\fr(S))$ its sifted cocontinuous extension.\footnote{We denote by $\PSh_\Sigma(\scr C) = \Fun^\times(\scr C^\op, \Spc)$ the non-abelian derived category of $\scr C$.}
Write $\Spc(S)_*$ for the localization of $\PSh_\Sigma(\Sm_{S+})$ at the generating motivic equivalences, that is, (generating) Nisnevich equivalences and $\A^1$-homotopy equivalences, and $\Spc^\fr(S)$ for the localization of $\PSh_\Sigma(\Cor^\fr(S))$ at the images of the generating motivic equivalences under $\gamma^*$.
Let $\P^1 \in \Spc(S)_*$ be pointed at $1$.
Recall that for any presentably symmetric monoidal $\infty$-category $\scr C$ and any object $P \in \scr C$ there is a universal presentably symmetric monoidal $\infty$-category under $\scr C$ in which $P$ becomes $\otimes$-invertible \cite[\S2.1]{robalo}; we denote it by $\scr C[P^{-1}]$.

The following is the main result of this section, which we will give a proof of assuming the comparison results explained in~\ref{sec:compare}.
\begin{theorem}[reconstruction] \label{thm:reconstruction}
The induced functor \[ \gamma^*: \Spc(S)_*[(\P^1)^{-1}] \to \Spc^\fr(S)[\gamma^*(\P^1)^{-1}] \] is an equivalence.
\end{theorem}

We write $\SH(S) = \Spc(S)_*[(\P^1)^{-1}]$ and $\SH^\fr(S) = \Spc^\fr(S)[\gamma^*(\P^1)^{-1}]$.
We shall prove the result when $S=Spec(k)$ is the spectrum of an infinite field.
The result for general $S$ is reduced to this case in \cite{framed-loc} (using \cite[\S B]{EHKSY1}).

\begin{remark}
While $\SH^\fr(S)$ appears to be a more complicated $\infty$-category than $\SH(S)$ (the $\infty$-category of motivic spectra), the point of motivic infinite loop space theory (and the rest of this note) is to give explicit formulas for mapping spaces in $\SH^\fr(S)$, at least when $S$ is the spectrum of a perfect field.
More precisely: the functor $\Spc(S)_* \rightarrow \SH(S)$ is far from being fully faithful, while the cancellation theorem Theorem~\ref{thm:cancel-main} asserts that the functor $\Spc^\fr(S)_* \rightarrow \SH^\fr(S)$ is fully faithful on grouplike objects. 
\end{remark}

\subsection{Preliminary reductions}
The functor $\gamma^*$ preserves colimits by construction, so has a right adjoint $\gamma_*$.
The stable presentable $\infty$-category $\SH(S)$ is compactly generated by objects of the form $\Sigma^\infty_+ X \wedge (\P^1)^{\wedge n}$, for $X \in \Sm_S$ and $n \in \Z$.
Similarly $\SH^\fr(S)$ is compactly generated by $\gamma^*(\Sigma^\infty_+ X \wedge (\P^1)^{\wedge n})$.
It follows that $\gamma_*: \SH^\fr(S) \to \SH(S)$ is conservative and preserves colimits.

Conservativity of $\gamma_*$ implies that in order to prove that $\gamma^*$ is an equivalence, it suffices to show that it is fully faithful, or equivalently that the unit of adjunction $u: \id \to \gamma_*\gamma^*$ is an equivalence.
Indeed the composite \[ \gamma_* \xrightarrow{u \gamma_*} \gamma_*\gamma^*\gamma_* \xrightarrow{\gamma_* c}\gamma_* \] is the identity ($\gamma^*$ and $\gamma_*$ being adjoints), the first transformation is an equivalence by assumption, hence so is the second one, and finally so is the counit $c$ since $\gamma_*$ is conservative.

Since $\gamma_*$ preserves colimits, the class of objects on which $u$ is an equivalence is closed under colimits.
Hence it suffices to show that $u$ is an equivalence on the generators.

Given any adjunction $F: \scr C \adj \scr D: U$ with $F$ symmetric monoidal, the right adjoint $U$ satisfies a projection formula for strongly dualizable objects: if $P \in \scr C$ is strongly dualizable, then there is an equivalence of functors $\gamma_*(\ph \otimes \gamma^* P) \wequi \gamma_*(\ph) \otimes P$.
Indeed we have a sequence of binatural equivalences \begin{gather*} \Map(\ph, \gamma_*(\ph \otimes \gamma^* P)) \wequi \Map(\gamma^*(\ph), \ph \otimes \gamma^* P) \\ \wequi \Map(\gamma^*(\ph \otimes P^\vee), \ph) \wequi \Map(\ph \otimes P^\vee, \gamma_*(\ph)) \wequi \Map(\ph, \gamma_*(\ph) \otimes P), \end{gather*} and hence the result follows by the Yoneda lemma.

Since $\Sigma^\infty \P^1 \in \SH(S)$ is invertible and hence strongly dualizable, in order to prove Theorem \ref{thm:reconstruction} it is thus enough to show that for every $X \in \Sm_S$, the unit map \[ \Sigma^\infty_+ X \to \gamma_*\gamma^* \Sigma^\infty_+ X \in \SH(S) \] is an equivalence.
Using Zariski descent, we may further assume that $X$ is affine.

\subsection{Recollections on prespectra}
Let $\scr C$ be a presentably symmetric monoidal $\infty$-category\NB{Since tensoring with a fixed object is cocontinuous it has a right adjoint, so the monoidal structure is closed.}, and $P \in \scr C$.
We denote by $\Sp^\N(\scr C, P)$ the $\infty$-category whose objects are sequences $(X_1, X_2, \dots)$ with $X_i \in \scr C$, together with ``bonding maps'' $P \otimes X_i \to X_{i+1}$. The objects are called \emph{prespectra}.
The morphisms are the evident commutative diagrams.
We call $X = (X_n)_n \in \Sp^\N(\scr C, P)$ an \emph{$\Omega$-spectrum} if the adjoints of the bonding maps, $X_i \to \Omega_P X_{i+1}$, are all equivalences.
Here $\Omega_P: \scr C \to \scr C$ denotes the right adjoint of the functor $\Sigma_P := P \otimes (\ph)$.
We denote by $L_\stab \Sp^\N(\scr C, P) \subset \Sp^\N(\scr C, P)$ the subcategory of $\Omega$-spectra.
The inclusion has a left adjoint which we denote by $L_\stab:\Sp^\N(\scr C, P) \to L_\stab \Sp^\N(\scr C, P)$; the maps inverted by $L_\stab$ are called stable equivalences.

\begin{remark} \label{rmk:symmetric-stabilization}
If $P$ is a symmetric object, i.e. for some $n \ge 2$ the cyclic permutation on $P^{\otimes n}$ is homotopic to the identity, then $L_\stab \Sp^\N(\scr C, P) \wequi \scr C[P^{-1}]$.
This is proved in \cite[Corollary 2.22]{robalo}.
\end{remark}

\subsubsection{Spectrification}
There is a natural transformation \[ \Sigma_P \Omega_P \xrightarrow{c} \id \xrightarrow{u} \Omega_P \Sigma_P. \]
Using this we can build a functor $Q_1: \Sp^\N(\scr C, P) \to \Sp^\N(\scr C, P)$ with the property that for $X = (X_n)_n \in \Sp^\N(\scr C, P)$ we have $Q_1(X)_n = \Omega_P X_{n+1}$.
Moreover there is a natural transformation $\id \to Q_1$.
Iterating this construction and taking the colimit we obtain \[ \id \to Q := \colim_n Q_1^{\circ n}. \]

The following is well-known.
\begin{lemma} \label{lemm:spectrification}
Let $X \in \Sp^\N(\scr C, P)$.
\begin{enumerate}
\item The map $X \to QX$ is a stable equivalence.
\item If $\Omega_P$ preserves filtered colimits (i.e. $P \in \scr C$ is compact), then $QX$ is an $\Omega$-spectrum.
\end{enumerate}
\end{lemma}\todo{proof?}

\subsubsection{Prolongation}
Let $F: \scr C \to \scr C$ be an endofunctor.
Following Hovey \cite[Lemma 5.2]{hovey2001spectra}, we call $F$ \emph{prolongable} if we are provided with a natural transformation $\tau: \Sigma_P F \to F \Sigma_P$; we denote the data of a prolongable functor as a pair $(F, \tau)$.
Equivalently, we should provide a natural transformation $F \to \Omega_P F \Sigma_P$.
In any case, there is an obvious category of prolongable endofunctors (having objects the pairs $(F, \tau)$ as above).
Any prolongable functor $(F, \tau)$ induces an endofunctor \[ F: Sp^\N(\scr C, P) \to Sp^\N(\scr C, P), (X_n)_n \mapsto (FX_n)_n. \]
The bonding maps of $FX$ are given by \[ \Sigma_P F(X_n) \xrightarrow{\tau_{X_n}} F(\Sigma_P X_n) \xrightarrow{F b_n} F(X_{n+1}), \] where $b_n: \Sigma_P X_n \to X_{n+1}$ is the original bonding map.

\begin{example} \label{ex:OS-prolongable}
The functor $F_n = \Omega_P^n \Sigma_P^n$ is prolongable by $\Omega_P^n u \Sigma_P^n: F_n \to \Omega_P F_n \Sigma_P$, where $u: \id \to \Omega_P \Sigma_P$ is the unit transformation.
One checks easily\todo{ahem...} that \[ F_n \Sigma^\infty X \wequi Q_1^{\circ n} \Sigma^\infty X. \]
The transformation $\Omega_P^nu\Sigma_P^n$ defines a morphism $F_n \to F_{n+1}$ of prolongable functors; let $F_\infty$ be its colimit.
Then one checks that \[ F_\infty \Sigma^\infty X \wequi Q \Sigma^\infty X. \]
\end{example}

\begin{example}
The functor $F = \Sigma_P$ can be prolonged a priori in (at least) two ways, via the canonical isomorphism $\tau_1: \Sigma_P F = \Sigma_P \Sigma_P = F \Sigma_P$ and via the switch map $\tau_2: \Sigma_{P\otimes P} \to \Sigma_{P \otimes P}$.
Then $F_1 \wequi F_2$ as prolongable functors if and only if the switch map on $P \otimes P$ is the identity.
\end{example}

\begin{example} \label{ex:cyclic-prolongation-problem}
Let $F: \scr C \to \scr C$ be a lax $\scr C$-module functor, so that in particular for each $A \in \scr C$ we are given a transformation $\Sigma_A F \to F \Sigma_A$.
Specializing to $A = P$ we obtain a prolongable functor $\tilde F$, natural in the lax $\scr C$-module functor $F$.
The functor $F_n$ (from Example \ref{ex:OS-prolongable}) is a lax $\scr C$-module functor, via \[ A \otimes \iHom(P^{\otimes n}, P^{\otimes n} \otimes X) \to \iHom(P^{\otimes n}, P^{\otimes n} \otimes A \otimes X), \text{``}(a \otimes f) \mapsto c_a \otimes f\text{''}, \] where $c_a$ denotes the ``constant map at $a$''.

Suppose that $P$ is strongly dualizable with dual $P^{\vee}$. Then $F_n$ and $\tilde F_n$ have equivalent underlying functors. However, their prolongation are described in different ways. 
The functor $F_n$ can be written as \[ P^{\vee \otimes n} \otimes P^{\otimes n} \xrightarrow{u} P^{\vee \otimes n} \otimes P^{\otimes n} \otimes P^\vee \otimes P \xrightarrow{\sigma_{324}} P^{\vee \otimes n} \otimes P^\vee \otimes P \otimes P^{\otimes n} \wequi P^{\vee \otimes n+1} \otimes P^{\otimes n+1}. \]
On the other hand the prolongation of $\tilde F_n$ can be written as \[ P^{\vee \otimes n} \otimes P^{\otimes n} \xrightarrow{u} P^{\vee \otimes n} \otimes P^{\otimes n} \otimes P^\vee \otimes P \xrightarrow{\sigma_{123}} P^\vee \otimes P^{\vee \otimes n} \otimes P^{\otimes n} \otimes P \wequi P^{\vee \otimes n+1} \otimes P^{\otimes n+1}. \]
They are isomorphic if and only if the $(n+1)$-fold cyclic permutation acts trivially on $P^{\otimes n}$.\todo{I think.}
\end{example}

\begin{example}
Let $\id \xrightarrow{u} F_1 \xrightarrow{\rho} \id$ be a retraction of prolongable functors.
Since $\rho: F_1 \to \id$ is a morphism of prolongable functors, the following square commutes
\begin{equation*}
\begin{CD}
\Omega_P \Sigma_P = F_1 @>{\Omega_P u \Sigma_P}>> \Omega_P F_1 \Sigma_P = \Omega_P^2 \Sigma_P^2 \\
@V{\rho}VV                             @V{\Omega_P \rho \Sigma_P}VV \\
\id               @>u>>                      F_1 = \Omega_P \Sigma_P.
\end{CD}
\end{equation*}
Hence $u \rho \wequi \Omega_P \rho u \Sigma_P \wequi \id_{\Omega_P \Sigma_P}$ and so $u$ and $\rho$ are inverse equivalences.
\end{example}

\begin{remark} \tom{think this through again}
The prolongations using lax module structures interacts more reasonably with categorical constructions than the one via units of adjunction.
For this reason it is more natural to have a retraction $\id \xrightarrow{u'} \tilde F_1 \xrightarrow{\rho} \id$.
If $P$ is $2$-symmetric (i.e. the switch on $P^{\otimes 2}$ is the identity), then $F_1 \wequi \tilde F_1$ and $u' \wequi u$, under this equivalence.
Hence $u', \rho$ are inverse equivalences.
This holds more generally if $P$ is $n$-symmetric for any $n \ge 2$; this is the content of Voevodsky's cancellation theorem. See Theorem \ref{thm:abstract-cancellation} in \S\ref{sec:cancellation}.
\end{remark}

\subsection{Equationally framed correspondences}
\subsubsection{Framed correspondences}
We have the lax $\Sm_{S+}$-module functor \[ h^\fr: \Sm_{S+} \to \PSh_\Sigma(\Sm_S)_*, X_+ \mapsto \gamma_* \gamma^* X_+. \]
We extend this to a sifted cocontinuous functor \[ h^\fr: \PSh_\Sigma(\Sm_S)_* \wequi \PSh_\Sigma(\Sm_{S+}) \to \PSh_\Sigma(\Sm_S)_*. \]
Of course $\gamma_* \gamma^*$ is already sifted cocontinuous, so $h^\fr \wequi \gamma_* \gamma^*$ and this is just a notational change.

\subsubsection{Equationally framed correspondences}
There are explicitly defined lax $\Sm_{S+}$-module functors \cite[\S2.1]{EHKSY1} \[ h^{\efr,n}: \Sm_{S+} \to \PSh_\Sigma(\Sm_S)_* \] and natural transformations $\sigma: h^{\efr,n} \to h^{\efr,n+1}$.
We denote by \[ h^{\efr,n}: \PSh_\Sigma(\Sm_S)_* \to \PSh_\Sigma(\Sm_S)_* \] the sifted cocontinuous extensions, and by \[ h^\efr: \PSh_\Sigma(\Sm_S)_* \to \PSh_\Sigma(\Sm_S)_* \] the colimit along $\sigma$. We will elaborate on this in \S\ref{sec:quots}.

\subsubsection{Relative equationally framed correspondences} \label{subsec:releq}
Let $U \subset X \in \Sm_S$ be an open immersion.
There are explicitly defined presheaves \[ h^{\efr,n}(X, U) \in \PSh_\Sigma(\Sm_{S+}); \] they depend functorially on the pair $(X,U)$ and are lax modules, in a way which we will not elaborate on.
For us the most important case is where $X = X' \times \A^m$ and $U = X' \times \A^m \setminus X' \times \{0\}$; we put \[ h^{\efr,n}(X', \scr O^n) = h^{\efr,n}(X' \times \A^m, X' \times \A^m \setminus X' \times \{0\}. \]
These assemble into lax $\Sm_{S+}$-module functors $\Sm_{S+} \to \PSh_\Sigma(\Sm_{S+})$. We will elaborate on this in \S\ref{sec:rel}.

\subsection{Comparison results}\label{sec:compare} We now explain the comparison results which go into the proof of the reconstruction theorem. 
\subsubsection{Equationally framed versus tangentially framed}
There is a canonical transformation \[ h^\efr \to h^\fr \in \Fun(\Sm_{S+}, \PSh_\Sigma(\Sm_{S+})) \] which is a motivic equivalence (objectwise) \cite[Corollaries 2.2.20 and 2.3.25]{EHKSY1}.
Since motivic equivalences are stable under (sifted) colimits, the sifted cocontinuous extension of the natural transformation is still a motivic equivalence objectwise.
The transformations are compatible with the lax module structures.

\subsubsection{The cone theorem}
There is a canonical transformation \[ h^{\efr,n}(X/U) \to h^{\efr,n}(X, U); \] here the left hand side is obtained by sifted cocontinuous extension.
This is a motivic equivalence for $X$ affine, provided the base is an infinite field.
This is known as the cone theorem, and will be treated in \S\ref{sec:cone}.

The natural transformation \[ h^{\efr,n}(X \times \A^m/X \times \A^m \setminus X \times 0) \to h^{\efr,n}(X, \scr O^m) \] can be promoted to a lax module transformation.

\subsubsection{Voevodsky's lemma}
We denote by $T \in \PSh_\Sigma(\Sm_{S+})$ the presheaf quotient $\A^1/\Gm$. It comes equipped with a canonical map of presheaves $a:\P^1 \rightarrow L_\Zar T$ by presenting the domain as a (Zariski-local) pushout $\A^1 \cup_{\Gm} \A^1 \simeq \P^1$.

There is a canonical equivalence of lax module functors \[ h^{\efr,n}(X, \scr O^m) \to \Omega_{\P^1}^n L_\Nis \Sigma_T^{n+m}X_+. \]
This is known as Voevodsky's Lemma, see \cite[Appendix A]{EHKSY1} for a proof.
The equivalence is compatible with the natural stabilization maps (increasing $n$) on both sides.

\subsection{Proof of reconstruction}
Write $\Shv_\Nis(S) = L_\Nis \PSh_\Sigma(\Sm_S)$ and similarly $\Shv_\Nis^\fr(S) = L_\Nis \PSh_\Sigma(\Cor^\fr(S))$.

\begin{lemma} \label{lemm:forget-mot-equiv}
The forgetful functor $\Shv_\Nis^\fr(S) \to \Shv_\Nis(S)$ preserves and detects motivic equivalences.
\end{lemma}
\begin{proof}
Immediate from \cite[Proposition 3.2.14]{EHKSY1}.
\end{proof}

Since $\gamma^*: \Shv_\Nis(S)_* \to \Shv_\Nis^\fr(S)$ is symmetric monoidal, it induces a functor $\gamma^*_\N$ upon passage to prespectra.
We obtain an adjunction \[ \gamma^*_\N: \Sp^\N(\Shv_\Nis(S)_*, \P^1) \adj \Sp^\N(\Shv_\Nis^\fr(S), \gamma^* \P^1): \gamma_*^\N; \] the right adjoint $\gamma_*^\N$ is given by the formula $\gamma_*^\N(X)_n \wequi \gamma_*(X_n)$.
We call a map $X \to Y \in \Sp^\N(\Shv_\Nis(S)_*, \P^1)$ a \emph{level motivic equivalence} if each map $X_n \to Y_n$ is a motivic equivalence, and similarly for framed prespectra.
The saturated class generated by level motivic equivalences and stable equivalences is called stable motivic equivalences.
Local objects for this class of maps are called motivic $\Omega$-spectra; these are the prespectra $X = (X_n)_n$ such that $X$ is an $\Omega$-spectrum and each $X_n$ is motivically local.

\begin{corollary} \label{cor:cancellation1}
The functor $\gamma_*^\N$ preserves and detects stable motivic equivalences.
\end{corollary}
\begin{proof}
Since $\gamma_*^\N$ preserves motivic $\Omega$-spectra (from its formula above) it is enough to show that it commutes with spectrification.
Let $X = (X_n)_n$ be a prespectrum. By Lemma \ref{lemm:spectrification}(2), its spectrification is given by \[ (QL_\mot X)_n = \colim_i \Omega_{\P^1}^i L_\mot X_{n+i}. \]
Since $\gamma_*: \Shv_\Nis^\fr(S) \to \Shv_\Nis(S)_*$ preserves motivic equivalences, filtered colimits (both by Lemma \ref{lemm:forget-mot-equiv}), and $\P^1$-loops, the result follows.
\end{proof}

We also note the following.
\begin{lemma} \label{lemm:cancellation2}
There are canonical equivalences \[ L_{\stab,\mot} \Sp^\N(\Shv_\Nis(S)_*, \P^1) \wequi \SH(S) \] and \[ L_{\stab,\mot} \Sp^\N(\Shv_\Nis^\fr(S), \gamma^* \P^1) \wequi \SH^\fr(S). \]
\end{lemma}
\begin{proof}
We prove the result for unframed spectra; the other case is similar.
It is easy to see that $L_\mot \Sp^\N(\Shv_\Nis(S)_*, \P^1) \wequi \Sp^\N(\Spc(S)_*, \P^1)$ (see e.g. \cite[Lemma 26]{bachmann-real-etale}).
But $\P^1$ is symmetric in $\Spc(S)_*$ \cite[Lemma 6.3]{hoyois-equivariant} and hence the result follows from Remark \ref{rmk:symmetric-stabilization}.
\end{proof}

Let $G: \Shv_\Nis(S)_* \to \Shv_\Nis(S)_*$ be an endofunctor.
We say that $G$ is \emph{mixed prolongable} if we are given a natural transformation $\Sigma_{\P^1} G \to G \Sigma_T$.
Then $G$ naturally induces a functor \[ G: \Sp^\N(\Shv_\Nis(S)_*, T) \to \Sp^\N(\Shv_\Nis(S)_*, \P^1). \]
Let $G_n = \Omega_{\P^1}^n \Sigma_T^n$.
This is mixed prolongable via \[ \Omega_{\P^1}^n \Sigma_T^n \xrightarrow{\Omega_{\P^1}^n u \Sigma_T^n} \Omega_{\P^1}^n \Omega_T \Sigma_T^{n+1} \xrightarrow{a^*} \Omega_{\P^1}^{n+1} \Sigma_T^{n+1}; \] here $a: \P^1 \to T$ is the canonical map.
For $X \in \Sm_S$, let $\Sigma^\infty_T X$ denote the associated $T$-suspension prespectrum.
Then \[ G_0 \Sigma^\infty_T X = (X, T \wedge X, T^2 \wedge X, \dots) \in \Sp^\N(\Shv_\Nis(S)_*, \P^1) \] is a spectrum motivically equivalent to $\Sigma^\infty_{\P^1} X$.
By Corollary \ref{cor:cancellation1} and Lemma \ref{lemm:cancellation2} it is hence enough to show that \[ G_0 \Sigma^\infty_T X \to \gamma^\N_* \gamma_\N^* G_0 \Sigma^\infty_T X \] is a stable motivic equivalence.
There are canonical maps of mixed prolongable functors $G_0 \to G_1 \to \dots$, and one checks that \[ Q G_0 \Sigma^\infty_T X \wequi \colim_i G_i \Sigma^\infty_T X. \]
In particular the map \[ G_0 \Sigma^\infty_T X \to \colim_i G_{2i+1} \Sigma^\infty_T X \] is a stable equivalence.

The functor $\Omega_{\P^1}^n \Sigma_T^n$ is mixed prolongable in another way, using the lax module structure.
Denote the mixed prolongable functor obtained in this way by $\tilde G_n$.
Arguing as in Example \ref{ex:cyclic-prolongation-problem}, $G_n$ and $\tilde G_n$ differ by cyclic permutations of $\P^1, T$ of order $n+1$.
Note that the functor $\iHom(\ph, \ph)$ preserves $\A^1$-homotopy equivalences in both variables.
Since the cyclic permutation on $(\P^1)^{\wedge 2n+1}$ is $\A^1$-homotopic to the identity\footnote{Observe the equality of permutations $(1,2,3)(3,4,5) \cdots (2n-1,2n,2n+1) = (1,2, \dots, 2n, 2n+1)$.}, and the same holds for $T$\NB{reference}, we deduce that $G_{2i+1} \stackrel{\A^1}{\wequi} \tilde{G}_{2i+1}$ as prolongable functors.\todo{really??}
We learn that the canonical map \[ G_0 \Sigma^\infty_T X \to \colim_i G_{2i+1} \Sigma^\infty_T X \stackrel{\A^1}{\wequi} \colim_i \tilde G_{2i+1} \Sigma^\infty_T X \wequi \colim_i \tilde G_i \Sigma^\infty_T X \] is an $\A^1$-equivalence.

Let $E_i$ denote the sifted cocontinuous approximation\footnote{In the sense that we restrict $G_i$ to $\Sm_{S+}$ and then take the sifted cocontinuous extension.} of $\tilde G_i$, so that there is a map $E_i \to \tilde G_i$ of mixed prolongable functors\todo{details}.
We can view $h^\efr$ (and $h^\fr$) as mixed prolongable functors (note that they preserve Nisnevich equivalences in $\PSh_\Sigma(\Sm_{S+})$ by \cite[Propositions 2.3.7(ii) and 2.1.5(iii)]{EHKSY1} and so descend to Nisnevich sheaves) by using their lax module structures.
By Voevodsky's Lemma and the fact that both functors are sifted cocontinuous extensions from smooth schemes, $E_i \wequi h^{\efr,i}$ as lax modules and hence as mixed prolongable functors.
Thus by the cone theorem (see Theorem \ref{thm:cone} and Remark \ref{rmk:cone-approx} in \S\ref{sec:cone} for more details), the map \[ E_i \Sigma^\infty_T X \to \tilde G_i \Sigma^\infty_T X \] is a level motivic equivalence (here we use that the base is an infinite field).
We obtain the following commutative diagram
\begin{equation*}
\begin{tikzcd}
G_0 \Sigma^\infty_T X \ar[r, "L_{\stab}"] \ar[dr] & G_\infty \Sigma^\infty_T X \ar[r, "L_{\A^1}", leftrightarrow] & \tilde G_\infty \Sigma^\infty_T X \\
         & E_\infty \Sigma^\infty_T \ar[ur, "L_\mot"] \ar[r, "\wequi"] X & h^\efr \Sigma^\infty_T X \ar[r, "L_\mot"] & h^\fr \Sigma^\infty_T X.
\end{tikzcd}
\end{equation*}
All maps are the canonical ones; labels on the arrows denote the type of equivalence.
The composite $G_0 \Sigma^\infty_T X \to h^\fr \Sigma^\infty_T X \wequi \gamma^\N_* \gamma_\N^* G_0 \Sigma^\infty_T X$ is the unit of adjunction.
The diagram proves this unit is a stable motivic equivalence.
This concludes the proof.

\section{The cone theorem}
\label{sec:cone}

Primary sources: \cite{gnp,druzhinin2018framed}.

\subsection{Introduction}

The cone theorem is the determination of the motivic homotopy type of $h^{\efr}(X/U)$, i.e., the ``framed cone'' of an open immersion $U \hookrightarrow X$ where $X$ is smooth. In the proof of the reconstruction theorem, coupled with Voevodsky's lemma (Lemma~\ref{lem:v-lem}), it relates the endofunctor on pointed Nisnevich sheaves given by $\Omega_{\P^1}\Sigma_T$ and the sifted cocontinous extension of a framed model of this functor. 

\newcommand{\qf}{\mathrm{qf}}

\begin{theorem} \label{thm:cone}
Let $k$ be an infinite field, $X$ a smooth affine $k$-scheme, and $U \subset X$ open.
Then there is a canonical motivic equivalence \[ h^{\efr,n}(X/U) \to h^{\efr,n}(X, U). \]
\end{theorem}

For now we work over an arbitrary base scheme $S$. We have already discussed Voevodsky's lemma that describes $h^{\efr,n}(X)$ in terms of maps of pointed sheaves (see \S\ref{subsub:voevodsky-lemma}). In general we can describe the sections of the (pointed) sheaf
\[
L_{\Nis}(X/U),
\]
as follows. Define
\[
Q(X, U)(T) = \{ (Z, \phi) \mid Z \subset T \text{ closed, } \phi: T^h_Z \rightarrow X, \phi^{-1}(X \setminus U) = Z \}
\]
which is pointed at $(\emptyset, \mathrm{can})$. Here $T_Z^h$ denotes the henselization of $T$ in $Z$.\NB{reference?} There is canonical map
\[
Q(X,U) \rightarrow L_{\Nis}(X/U),
\]
which sends a section $(Z, \phi)$ over $T$ to the map
\[
T \simeq L_{\Nis}(T^h_Z \amalg_{T^h_Z \setminus Z} T \setminus Z) \xrightarrow{\phi} X/U.
\]

\begin{lemma}[\cite{EHKSY1}, Proposition A.1.4] \label{lem:v-lem} The map $Q(X, U) \rightarrow L_{\Nis}(X/U)$ is an isomorphism.
\end{lemma}

The presheaf of equationally framed correspondences of level $n$ can be phrased in these terms. Let us elaborate on how this is done. Recall that we have $n$ closed immersions $(\P^1)^{n-1} \hookrightarrow (\P^1)^n$ as the components of the ``divisor at $\infty$" (so that $\bigcup (\P^1)^{\times n-1}$ is the divisor $\partial \P$). We then have the fiber sequence (in sets)
\[
h^{\efr,n}(X)(T) \rightarrow Q(\A^n \times X, \A^n_X \setminus 0_X )((\P^1)^{\times n} \times T) \rightarrow \prod_{1 \leq i \leq n}  Q(\A^n \times X, \A^n_X \setminus 0_X )((\P^1)^{\times n-1} \times T).
\]
Via Lemma~\ref{lem:v-lem}, $h^{\efr,n}(X)$ is isomorphic to
\[
\underline{\mathrm{Hom}}_{\PSh_\Sigma(\Sm_{S+})}( (\P^1)^{\wedge n} \wedge(-)_+, L_\Nis (T^{\wedge n} \wedge X_+)). 
\]

\subsection{Relative equationally framed correspondences} \label{sec:rel}
We elaborate on the discussion in~\S\ref{subsec:releq}. Throughout $X$ is a smooth affine $S$-scheme and we have a cospan of $S$-schemes
\[
Y \stackrel{i}{\hookrightarrow} X \stackrel{j}{\hookleftarrow} X \setminus Y (=: U),
\]
where $i$ is a closed immersion and $j$ is its open complement. The presheaf of {\em relative equationally framed correspondences} $h^{\efr,n}(X, U)$ is then defined via a similar formula:
\[
h^{\efr,n}(X, U)(T) \rightarrow Q(\A^n \times X, \A^n_X \setminus (0 \times Y))((\P^1)^{\times n} \times T) \rightarrow \prod_{1 \leq i \leq n}  Q(\A^n \times X, \A^n_X \setminus (0 \times Y))((\P^1)^{\times n-1} \times T).
\]

The next lemma follows from the above discussion.

\begin{lemma} There is a canonical isomorphism of sheaves of sets 
\[
h^{\efr,n}(X, U) \wequi \underline{\mathrm{Hom}}_{\PSh_\Sigma(\Sm_{S+})}( (\P^1)^{\wedge n} \wedge(-)_+, L_\Nis(T^{\wedge n} \wedge (X/U))).
\]
\end{lemma}

\begin{remark} \label{rmk:cone-approx}
Consider the functor $G: \PSh_\Sigma(\Sm_{S+}) \to \PSh_\Sigma(\Sm_{S+})$ given by \[ G(P) =  \underline{\mathrm{Hom}}_{\PSh_\Sigma(\Sm_{S+})}( (\P^1)^{\wedge n} \wedge(-)_+, L_\Nis(T^{\wedge n} \wedge P)). \]
Write $c: E \to G$ for the sifted-cocontinuous approximation of $G$ (i.e. the left Kan extension of $E|_{\Sm_{S+}}$).
Then by Voevodsky's lemma we have $E \wequi h^{\efr,n}$.
Consequently we obtain a natural map \[ c_{X/U}: h^{\efr,n}(X/U) \wequi E(X/U) \to G(X/U) \wequi h^{\efr,n}(X,U). \]
This is the map of Theorem \ref{thm:cone}.
\end{remark}

Explicitly, elements of $h^{\efr,n}(X, U)(T)$ are described as (equivalence classes of) tuples
\[
(Z, (\phi, g), W),
\]
where
\begin{enumerate}
\item $Z \hookrightarrow \A^n_T$ is a closed subscheme, finite over $T$,
\item $W$ is an \'etale neighborhood of $Z$ in $\A^n_T$,
\item $(\phi, g): W \rightarrow \A^n \times X$ is a morphism such that 
\[
Z = (\phi,g)^{-1}(0 \times Y) = \phi^{-1}(0) \cap g^{-1}(Y).
\]
\end{enumerate}

For example, suppose that $X = \A^1$ and $U = \Gm$. Then $h^{\mathrm{efr}}_n(\A^1, \Gm)$ is isomorphic to
\[
\underline{\mathrm{Hom}}_{\PSh_\Sigma(\Sm_{S+})}( (\P^1)^{\wedge n} \wedge(-)_+, L_\Nis T^{\wedge n+1}). 
\]

\begin{remark}
The subscheme $Z$ in the definition of $Q(X,U)((\P^1)^{\wedge n})$ is not required to be finite.
However, in the definition of $h^{\efr,n}(X, U)$, the $Z$ appearing is a closed subset of both $(\P^1)^{\times n}$ and $\A^n$, so both proper and affine, hence finite.
\end{remark}

We will also need the next presheaf.
\begin{definition} \label{def:qf} Let $h^{\efr,n}_{\qf}(X, U) \subset h^{\efr}(X, U)$ be the subpresheaf consisting of those $(Z, (\phi, g), W)$ where $\phi^{-1}(0) \to T$ is quasi-finite.
\end{definition}
\begin{remark}
Recall that the scheme $W$ in an equationally framed correspondence is well-defined only up to refinement.
If $p: W' \to W$ is such a refinement and $\phi^{-1}(0)$ is quasi-finite, then so is $(\phi \circ p)^{-1}(0)$, $p$ being quasi-finite.\NB{to be pedantic, $p$ is only locally quasi-finite, but can be refined by a quasi-finite map}
The converse need not hold.
\end{remark}

\begin{example} In $h^{\efr,1}(\A^1, \Gm)(k)$, we have the cycle $c = (Z = 0_k, (0, x), \A^1)$, where $0$ indicates the constant function at zero, so we are considering the zero locus of the map
\[
(0, x): \A^1 \rightarrow \A^1 \times \A^1.
\]
In this situation, $0^{-1}(0) = \A^1$ and hence is \emph{not} quasi-finite over the base field, so $c \not\in h^{\efr,1}_{\mathrm{qf}}(\A^1, \Gm)(k)$.
On the other hand $0^{-1}(0) \cap x^{-1}(0) = 0$, which restores the finiteness of $Z$, as needed.
Generically, we should expect quasi-finiteness of $\phi^{-1}(0)$ --- the only function we need to avoid in the above example is literally the constant function at zero.
\end{example}

The relevance of the quasi-finite version is the following.
\begin{construction} \label{eq:construct-can} We have a map
\[
h^{\efr,n}(X) \rightarrow h^{\efr,n}(X, U) \qquad (W, (\phi, g), Z) \mapsto (W, (\phi, g), Z_Y = \phi^{-1}(0) \cap g^{-1}(Y)),
\]
which factors as
\begin{equation} \label{eq:factors}
h^{\efr,n}(X) \rightarrow h^{\efr,n}_{\qf}(X, U) \subset h^{\efr,n}(X, U).
\end{equation}
(Since $\phi^{-1}(0)$ is, in fact, finite.)
Now, consider the diagram
\[
h^{\efr,n}(X \amalg U) \rightrightarrows h^{\efr,n}(X),
\]
\[
(Z, (\phi, g), W) \mapsto ((Z, (\phi, \nabla \circ g), W), (Z_X, (\phi_X, g_X), W_X)),
\]
where $(Z_X, (\phi_X, g_X), W_X)$ is the component of $(Z, (\phi, g), W)$ over $X$, and $\nabla: X \amalg U \to X$ is the fold map. Denote the set-theoretic coequalizer of this diagram (taken sectionwise) by $\tau_{\leq 0}h^{\efr,n}(X/U)$. (This notation is justified in \S\ref{sec:quots}.) The map~\eqref{eq:factors} then further factors as
\[
\begin{tikzcd} 
 & h^{\efr}(X) \ar{dl} \ar{dr} & \\
\tau_{\leq 0}h^{\efr,n}(X/U) \ar[dashed]{rr} & & h^{\efr,n}_\qf(X,U).
\end{tikzcd}
\]
\end{construction}

We can explicitly describe the sections of the presheaf $\tau_{\leq 0}h^{\efr,n}(X/U)$:  if $T \in \Sm_S$, then $\tau_{\leq 0}h^{\efr,n}(X/U) (T)$ is the quotient of $h^{\efr,n}(X)$ modulo the equivalence relation generated by
\[
(W, (\phi, g), Z) \sim (W', (\phi', g'), Z'),
\]
whenever there exists $(W'', (\phi'',g''), Z'')$ such that $g'': W'' \rightarrow U \subset X$, $W = W' \amalg W''$ up to refining the étale neighbourhoods, and $(g,\phi) = (g',\phi') \amalg (g'',\phi'')$.

%
%
\begin{remark} \label{rem:add} We warn the reader that the canonical map $h^{\efr,n}(X \amalg Y) \rightarrow h^{\efr,n}(X) \times h^{\efr}(Y)$ is not an equivalence (unless $X = \emptyset$ or $Y = \emptyset$). It becomes so after applying $L_{\A^1}$ and letting $n \rightarrow \infty$ \cite[Remark 2.19]{EHKSY1}, \cite[Theorem 6.4]{garkusha2014framed}.
\end{remark}

\begin{lemma} \label{lem:l-nis} Let $S$ be any scheme.
The map $\tau_{\leq 0}h^{\efr,n}(X/U)  \rightarrow h^{\efr,n}_\qf(X,U)$ is an $L_{\Nis}$-equivalence.
\end{lemma}

\begin{proof}
Let $T$ be the henselization of a smooth $S$-scheme in a point.
It suffices to show that the map on sections over $T$ is both surjective and injective.

\emph{Surjectivity:} Take $(Z, (\phi, g), W) \in h^{\efr,n}_\qf(X,U)(T)$ and put $V =  \phi^{-1}(0)$, so that $Z = V \cap g^{-1}(Y)$.
We may assume that $W$ is affine (see \cite[Lemma A.1.2(ii)]{EHKSY1}), and hence so is $V$.
Since $V$ is quasi-finite, we may write \[ V = V_1 \amalg \dots \amalg V_{n+1}, \] where $V_i$ is local and finite over $T$ for $i \le n$, and $V_{n+1} \to T$ misses the closed point \cite[Tag 04GJ]{stacks-project}.
Similarly $Z = Z_1 \amalg \dots \amalg Z_d$.
We may assume that $Z_i \subset V_i$ and $d \le n$ (note that $Z_i \to T$ hits the closed point by finiteness and hence $Z_i \not\subset V_{n+1}$).
Removing $V_{d+1} \cup \dots \cup V_{n+1}$ from $W$, we may also assume that $n=d$ and $Z_{n+1} = \emptyset$.
In particular $V$ is finite over $T$.
It remains to prove that $V \to \A^n_T$ is a closed immersion.\NB{Rest of the argument feels unnecessarily convoluted.}
Denote by $\bar V \subset \A^n_T$ the image of $V$, which is a closed subscheme finite over $T$.\NB{$V \to \P^n_T$ is proper, so $\bar V \subset \P^n_T$ is closed, whence $\bar V$ proper over $T$}
We can write $W_{\bar V} = W_1 \amalg W_2$, where $W_1$ is finite over $\bar V$ and $W_2$ misses the closed points.
Then $W_2 \subset W$ is closed and misses all closed points of $V$, so $V \subset W \setminus W_2 =: W'$.
Now $W'_{\bar V} = W_1$ and so $W'_V \to V$ is finite étale; also $W'_Z \to Z$ is an isomorphism, whence so is $W'_V \to V$ \cite[Tag 04GK]{stacks-project}.
It follows that $W'_V \wequi V \to W'$ is a closed immersion and hence $V \to \A^n_T$ is a locally closed immersion (using fpqc descent \cite[Tag 02L6]{stacks-project}).
Since $V$ is finite, this is a closed immersion.


\emph{Injectivity:} Consider two cycles \[ c=(Z, (\phi, g), W), c'=(Z', (\phi', g'), W'), \] with the same image in $h^{\efr,n}_\qf(X,U)$.
Put $Z_1 = Z \cap g^{-1}(Y)$ and $Z_1' = Z' \cap g^{-1}(Y)$.
In other words $Z_1 = Z_1'$ and there exists an étale neighborhood $W''$ refining $W$ and $W'$ such that $(\phi, g)|_{W''} = (\phi', g')|_{W''}$.
We may write $Z = C \amalg D$, where $D \cap Z_1 = \emptyset$ and every component of $C$ meets $Z_1$ (using again \cite[Tag 04GJ]{stacks-project}).
Shrinking $W$ to remove $D$ replaces $c$ by a cycle with the same image in $\tau_{\leq 0}h^{\efr,n}(X/U)$; we may thus assume that $D = \emptyset$.
Now $\sigma: W''_Z \to Z$ is open and its image contains all closed points, so $\sigma$ is surjective.
Since every closed point of $Z$ lifts along $\sigma$ and $\sigma$ is étale, it follows that $\sigma$ admits a section \cite[Tags 04GJ and 04GK]{stacks-project}.
Thus, shrinking $W''$ if necessary, we may assume that it is an étale neighborhood of $Z$.
Arguing the same way for $Z'$ concludes the proof.

\end{proof}

\subsection{Quotients versus homotopy quotients} \label{sec:quots}

The quotient $X/U$ is given by the geometric realization of the following diagram in presheaves (also called a ``bar construction'')
\begin{equation} \label{eq:bar}
\begin{tikzcd}
X_+ 
& (X \amalg U)_+ \arrow[l, shift left]
\arrow[l, shift right]
& (X \amalg U \amalg U)_+ \arrow[l]
\arrow[l, shift left=
2
]
\arrow[l, shift right=
2
]
& \cdots.
\arrow[l, shift left]
\arrow[l, shift right]
\arrow[l, shift left=3]
\arrow[l, shift right=3]
\end{tikzcd}
\end{equation}
By definition (as sifted-colimit preserving extensions) we get that $h^{\efr}(X/U)$ is the colimit of the simplicial diagram
\begin{equation} \label{eq:bar-efr}
\begin{tikzcd}
h^{\efr}(X) 
& h^{\efr}(X\amalg U) \arrow[l, shift left]
\arrow[l, shift right]
& h^{\efr}(X\amalg U\amalg U) \arrow[l]
\arrow[l, shift left=
2
]
\arrow[l, shift right=
2
]
& \cdots.
\arrow[l, shift left]
\arrow[l, shift right]
\arrow[l, shift left=3]
\arrow[l, shift right=3].
\end{tikzcd}
\end{equation}
We remark that the first two maps coincide with those from Construction~\ref{eq:construct-can}. There is thus a canonical map
\[
h^{\efr,n}(X/U) \rightarrow \tau_{\leq 0}h^{\efr,n}(X/U),
\]
which witnesses $0$-truncation of the resulting geometric realization.

%

\begin{construction} \label{eq:construct-can-colim} Composing with the map from Construction~\ref{eq:construct-can}, we get maps
\[
h^{\efr,n}(X/U) \rightarrow \tau_{\leq 0}h^{\efr,n}(X/U) \rightarrow h^{\efr,n}_{\qf}(X,U) \hookrightarrow h^{\efr,n}(X,U),
\]
The composite is the map in question in the cone theorem. 
\end{construction}

We now claim that the first map is an equivalence, i.e., $h^{\efr,n}(X/U)$ is $0$-truncated.

\begin{construction} Let $\mathbf{efr}(X,U)(T)$ denote the following (1-)category (in fact, a poset):
\begin{itemize}
\item the objects are elements of $h^{\efr}(X)(T)$.
\item there is a morphism
\[
(Z, (\phi,g),W) \rightarrow (Z', (\phi',g'),W'),
\]
if and only if there exists a decomposition $Z \amalg Z'' = Z'$, $g'|_{Z''}$ factors through $U \subset X$, and $(\phi',g')|_{W'^h_Z} = (\phi,g)|_{W_Z^h}$.
\end{itemize}
\end{construction}

\begin{lemma} \label{lem:nerve} There is canonical equivalence 
\[
|N_{\bullet}\mathbf{efr}(X,U)(T)| \simeq h^{\efr}(X/U)
\]
\end{lemma}

\begin{proof} For this proof we will abbreviate $(W, (\phi,g), Z)$ as $(Z, \Phi)$; as we manipulate these cycles what happens on the data of the \'etale neighborhood and defining functions will be clear. For each $n$, we have a map
\[
N_n\mathbf{efr}(X,U)(T) \rightarrow h^{\efr}(X \amalg U^{\amalg n})(T),
\]
given by
\[
(Z_0, \Phi_0) \rightarrow \cdots \rightarrow (Z_n, \Phi_n) \mapsto (Z_0 \amalg (Z_1 \setminus Z_0) \amalg (Z_2 \setminus Z_1) \amalg (Z_n \setminus Z_{n-1}), \Phi_n).
\]

%

On the other hand, if $(Z, \Phi) \in h^{\efr}(X \amalg U^{\amalg n})(T)$ we get cycles $\{ Z'_i, \Phi'_i \}_{i \geq 1}$ by pulling back along the various inclusions $\{ \iota_i: U \hookrightarrow X \amalg U^{\amalg n} \}$ and also a cycle $(Z_0, \Phi_0)$ by pulling back along $X \hookrightarrow X \amalg U^{\amalg n}$. This defines an element $N_n\mathbf{efr}(X,U)(T)$ by setting $(Z_i, \Phi)= (Z_0 \amalg Z'_1 \amalg \cdots Z'_i, \Phi_i)$, with the maps determined. These maps induce mutual inverses of simplicial sets.
\end{proof}

\begin{lemma} \label{lem:discrete} The space $|N_{\bullet}\mathbf{efr}(X,U)(T)|$ is $0$-truncated.
\end{lemma}

\begin{proof} Consider the subcategory 
\[
\mathbf{efr}(X,U)(T)^0 \subset \mathbf{efr}(X,U)(T),
\]
consisting of those cycles $(Z, \Phi)$ such that no (nonempty) connected component of $Z$ factors through $U$. Then $\mathbf{efr}(X,U)(T)^0$ is a category with no non-identity arrows, whence $|N_{\bullet}\mathbf{efr}(X,U)(T)^0|$ is $0$-truncated. The inclusion $\mathbf{efr}(X,U)(T)^0 \to \mathbf{efr}(X,U)(T)$ admits a right adjoint (given by discarding all components of $Z$ that factor through $U$), and hence induces an equivalence on classifying spaces.
The result follows.
\end{proof}

It follows that the canonical map
\[
h^{\efr,n}(X/U) \rightarrow \tau_{\leq 0}h^{\efr,n}(X/U),
\]
is a sectionwise equivalence of spaces.
Combining Lemmas \ref{lem:l-nis}, \ref{lem:nerve} and \ref{lem:discrete}, we have proved the following result.
\begin{theorem} \label{thm:lnis} Let $S$ be a scheme. The map
\[
h^{\efr,n}(X/U) \rightarrow h^{\efr,n}_{\qf}(X,U),
\]
is an $L_{\Nis}$-equivalence.
\end{theorem}

\subsection{Moving into quasi-finite correspondences}

In order to complete the proof of the cone theorem, we will need the following result.

\begin{theorem} \label{thm:laone} Let $S=Spec(k)$, where $k$ is an infinite field. The inclusion of presheaves
\[
h^{\efr,n}_{\qf}(X,U) \hookrightarrow h^{\efr,n}(X,U),
\]
is an $L_{\A^1}$-equivalence.
\end{theorem}

This is a moving lemma in motivic homotopy theory. 

\begin{remark} In \cite{gnp}, this moving lemma was discovered for $X = \A^n$ and $U = \A^n \setminus 0$ which suffices for the purposes of computing the framed motives of algebraic varieties. We will follow the treatment \cite{druzhinin2018framed} which performs the moving lemma for more general pairs.
\end{remark}

For the rest of this section, we work over an infinite field.
We fix the smooth affine scheme $X$, its open subscheme $U$ and closed complement $Y$.
Write $\overline{X}$ for a projective closure of $X$, and $\overline{Y} := \overline{X} \setminus U$.
By considering the Segre embedding, we find a very ample line bundle $\scr O(1)$ on $\P^n \times \overline{X}$ with a section $x_0$ such that $x_0|_{\A^n \times X}$ is non-vanishing.
We also have sections $x_1, \dots, x_n \in H^0(\P^n \times \overline{X}, \scr O(1))$ such that $x_i/x_0|_{\A^n \times X}$ are the usual coordinates on $\A^n$.
Denote by $\scr N \subset \P^n \times \overline{X}$ the closed subscheme which is the first order thickening of $0 \times \overline{Y}$.
Pick $d > 0$.
Set \[ \vec{x}=(x_1x^{d-1}_{0}, x_2x^{d-1}_{0}, \dots, x_n x_0^{d-1}) \in H^0(\P^n \times \overline{X}, \scr O(d)^{\oplus n}) \] and \[ H^0(\P^n \times \overline{X}, \scr O(d)^{\oplus n}) \supset \Gamma_d := \{ \vec{s} \mid \vec{s}|_\scr{N} = \vec{x}|_{\scr N} \}. \]
Note that $\Gamma_d$ is a finite dimensional\footnote{This is the reason for compactifying $X$.} affine $k$-space, which we will view as an affine scheme.

Suppose that $\vec{s}=(s_1, \cdots, s_n) \in \Gamma_{d}(k)$.
Then $\vec{s}|_{\A^n \times X}/x_0^d$ defines a regular map $\A^n \times X \to \A^n$.
Combining with the projection $\A^n \times X \to X$ we obtain \[ f_{\vec{s}}: \A^n \times X \rightarrow \A^n \times X. \]
By construction, $f_{\vec{s}}$ is the identity in the first order neighborhood of $0 \times Y \subset \A^n \times X$.
This has the following significance.

\begin{lemma} \label{lemm:track-zeros}
Let $\varphi: W \to \A^n_X$ be arbitrary.
Set $Z = \varphi^{-1}(0 \times Y) \subset W$.
For $\vec s \in \Gamma(k)$ we have \[ (f_{\vec s} \circ \varphi)^{-1}(0 \times Y) = Z \amalg Z' \] (for some $Z'$ depending on $\vec s$).
\end{lemma}
\begin{proof}
Let $Z_1 = f_{\vec s}^{-1}(0 \times Y)$.
It suffices to prove that $0 \times Y \to Z_1$ is an open (whence clopen) immersion.
Since $f_{\vec s}|_{\scr N} = \id$, we get $Z_1 \cap \scr N = Z$.
In other words, if $I$ is the sheaf of ideals defining $0 \times Y$, then $I|_{Z_1} = I^2|_{Z_1}$.
The result follows by \cite[Tag 00EH]{stacks-project}.
\end{proof}

\begin{construction} \label{cons:moving}
If $(Z, (\phi, g), W) \in  h^{\efr}(X, U)(T)$, then we define \[ \vec{s} \cdot (Z, (\phi, g), W) = ( Z, (f_{\vec{s}}\circ (g, \phi)), W \setminus Z').  \]
This makes sense by Lemma \ref{lemm:track-zeros} and yields in fact an action\todo{details?} \[ \Gamma_d \times h^{\efr,n}(X, U) \to h^{\efr,n}(X, U),\, (\vec{s}, \Phi) \mapsto \vec{s} \cdot \Phi. \]
\end{construction}

Multiplication by $x_0$ induces an injection $\Gamma_d \to \Gamma_{d+1}$.
Write $\Gamma_\infty = \bigcup_d \Gamma_d$.
Note that the action of $\Gamma_d$ on $h^{\efr,n}(X,U)$ factors through multiplication by $d$ and hence induces an action by $\Gamma_\infty$.

We need to be able to draw paths in $\Gamma_{d}$ with controlled properties.
This is made precise by the next result, whose proof will be discussed in \S\ref{sec:proofs-main}.
\begin{lemma} \label{lemm:key}
Let $T_1, \dots, T_n \in \Sm_k$, $c_i \in h^{\efr,n}(X, U)(T_i)$, $V_i \subset \Gamma_\infty$ finite dimensional.

Then there exists $\vec \gamma \in \Gamma_\infty \setminus \cup_i V_i$ such that, for all $i$, if $V'_i \subset \Gamma_\infty$ is the cone on $V_i$ with tip $\vec \gamma$, then for all $\vec v \in V_i' \setminus V_i$ we have $\vec v \cdot c_i \in h^{\efr,n}_{\qf}(X,U)(T_i)$.
Moreover we can arrange that if $\vec x \in V_i'$ then already $\vec x \in V_i$.
\end{lemma}
\begin{remark}
Taking $V_i=\{\vec x\}$, the lemma in particular asserts that we can use paths in $\Gamma_\infty$ to make correspondences quasi-finite.
The more general case $V_i \ne \{\vec x\}$ is used to show that these paths are essentially unique.
\end{remark}

\subsection{Filtration and finishing the proof}

Granting ourselves the above lemma, we finish the proof of the cone theorem. 

We begin with some preparations.
Let $A$ be a category and $D: A \to \PSh(\Sm_k)$ be an $A$-indexed diagram.
We construct a simplicial object \[ Tel_A(D)_\bullet \in \Fun(\Delta^\op, \PSh(\Sm_k)) \] by setting \[ Tel_A(D)_n = \coprod_{i_0 \to i_1 \to \dots \to i_n \in \scr C} D(i_0). \]
The simplicial structure maps involve the cosimplicial structure maps in the standard cosimplicial category $[\bullet]$ and the functoriality of $D$.
This is a standard construction, for which see e.g. \cite[\S4]{dugger2008primer}.
The standard cosimplicial affine scheme $\A^\bullet$ yields a functor $S_{\A}^\bullet: \PSh(\Sm_k) \to \Fun(\Delta^\op, \PSh(\Sm_k))$ which has a left adjoint $|\ph|_{\A^1}$.

\begin{lemma} \label{lemm:A1-hocolim}
The geometric realization $|Tel_A(D)_\bullet|_{\A^1}$ is $\A^1$-equivalent to $\colim_A D$.
\end{lemma}
\begin{proof}
For $F \in \PSh(\Sm_k)$ $\A^1$-invariant we have \[ \Map(|X_\bullet|_{\A^1},F) \wequi \Map(X_\bullet, S_{\A}^\bullet F) \wequi \Map(X_\bullet, cF) \wequi \Map(|X_\bullet|, F), \] where $cF$ denotes the constant simplicial presheaf.
The result follows since the usual geometric realization of $Tel_A(D)_\bullet$ is a standard model for the sectionwise homotopy colimit of $D$ \cite[\S4]{dugger2008primer}.
\end{proof}

\begin{proof} [Proof of Theorem~\ref{thm:laone}]
We shall supply a filtered poset $A$ as well as systems of subpresheaves \[ \{h^{\efr,n}(X,U)^\alpha\}_{\alpha \in A} \subset h^{\efr,n}(X,U), \quad \{h^{\efr,n}_{\qf}(X,U)^\alpha\}_{\alpha \in A} \subset h^{\efr,n}_{\qf}(X,U) \] such that \[ h^{\efr,n}(X,U) = \bigcup_{\alpha \in A} h^{\efr,n}(X,U)^\alpha \text{ and } h^{\efr,n}_{\qf}(X,U) = \bigcup_{\alpha \in A} h^{\efr,n}_{\qf}(X,U)^\alpha. \]
Next we construct for $\alpha = (\alpha_0 \le \dots \le \alpha_n), \alpha_i \in A$ maps \[ r_\alpha: h^{\efr,n}(X,U)^{\alpha_0} \times \A^n \to h^{\efr,n}_{\qf}(X,U) \] and \[ H_\alpha: \A^1 \times h^{\efr,n}(X,U)^{\alpha_0} \times {\A^n} \to h^{\efr,n}(X,U), \] \[ K_\alpha: \A^1 \times h^{\efr,n}_{\qf}(X,U)^{\alpha_0} \times \A^n \to h^{\efr,n}_{\qf}(X,U), \] all compatible with the (co)simplicial structure maps.
Applying $|\ph|_{\A^1}$, we obtain via Lemma \ref{lemm:A1-hocolim} a map \[ |r_\bullet|_{\A^1}: |Tel_A(h^{\efr,n}(X,U)^{(\ph)}|_{\A^1} \wequi \colim_A h^{\efr,n}(X,U)^{(\ph)} \wequi h^{\efr,n}(X,U) \to h^{\efr,n}_{\qf}(X,U). \]
The construction is arranged in such a way that $|H_\bullet|_{\A^1}$ and $|K_\bullet|_{\A^1}$ exhibit homotopies making the following triangles commute
\begin{equation*}
\begin{tikzcd} 
h^{\efr,n}_{\qf}(X,U) \ar[r,"\id"] \ar[d] & h^{\efr,n}_{\qf}(X,U)  &\quad& h^{\efr,n}(X,U) \ar[r,"\id"] \ar[d, "|r_\bullet|_{\A^1}"] & h^{\efr,n}(X,U)\\
h^{\efr,n}(X,U) \ar[ur, "|r_\bullet|_{\A^1}" swap] & && h^{\efr,n}_{\qf}(X,U) \ar[ur].
\end{tikzcd}
\end{equation*}

Set \[ \tilde A = \{\vec s,  \{(T_1, c_1, V_1), \dots, (T_n, c_n, V_n)\} \mid \vec s \in \Gamma_\infty, V_i \subset \Gamma_\infty, T_i \in \Sm_k, c_i \in h^{\efr,n}(X,U)(T_i) \}. \]
Here $V_i$ is a finite-dimensional, affine subspace.
Let $A \subset \tilde A$ be the subset of elements having the following properties:
\begin{itemize}
\item $\vec s \in V_i$, $\vec x \not\in V_i$,
\item for all $i$ and all $\vec {s'} \in V_i^{\vec x} \setminus \{\vec x\}$ we have $\vec {s'} \cdot c_i \in h^{\efr,n}_{\qf}(X,U)(T_i)$.
\end{itemize}
Here $V_i^{\vec x}$ denotes the affine subspace generated by $V_i$ and $\vec x$.

For $\alpha = (\vec s, M) \in A$ we denote by $h^{\efr,n}(X,U)^\alpha \subset h^{\efr,n}(X,U)$ the subpresheaf generated by the sections $c$ for $(T,c,V) \in M$ (ignoring the $V$ component), and similarly $h^{\efr,n}_{\qf}(X,U)^\alpha$ is the subpresheaf generated by those $c$ which happen to be quasi-finite.
We put an ordering on $A$ by declaring that $(\vec s, M) \le (\vec t, N)$ if for all $(T,c,V) \in M$ we have $(T,c,V^{\vec t}) \in N$.
It is immediate from Lemma \ref{lemm:key} that this makes $A$ into a filtered poset and that the filtrations of $h^{\efr,n}(X,U)$ and $h^{\efr,n}_{\qf}(X,U)$ are exhaustive.

With this preparation out of the way, let $\alpha = (\alpha_0 \le \alpha_1 \le \dots \le \alpha_n) \in A$, with $\alpha_i = (\vec s_i, M_i)$.
We set
\begin{gather*}
  r_\alpha(c, \lambda) = \vec s(\lambda) \cdot c \\
  H_\alpha(t, c, \lambda) = (t\vec s(\lambda) + (1-t)\vec x) \cdot c \\
  K_\alpha(t, c, \lambda) = (t\vec s(\lambda) + (1-t)\vec x) \cdot c.
\end{gather*}
Here $\lambda = (\lambda_1, \dots, \lambda_n) \in \A^n$ and \[ \vec s(\lambda) = (1-\sum_i \lambda_i) \vec s_0 + \sum_i \lambda_i \vec s_i. \]
The cosimplicial structure on $\A^\bullet$ comes from viewing $\A^n$ as the subspace of $\A^{n+1}$ where the sum of the coordinates is $1$.
With this interpretation, it is clear that this construction is compatible with the simplicial structure.
It remains to show that the maps $r_\alpha$ and $K_\alpha$ land in $h^{\efr,n}_{\qf}(X,U)$.
Let $(T,c,V) \in M_0$.
Let $V'$ be the affine subspace generated by $V$ and all the $s_i$.
One checks by induction that $(T,c,V') \in M_n$.
The required quasi-finiteness follows (recall that by assumption, $\vec v \cdot c$ is quasi-finite for $\vec v \in (V')^{\vec x} \setminus \vec x \supset V'$).
\end{proof}

\subsection{Proof of Lemma~\ref{lemm:key}} \label{sec:proofs-main}
We now prove the key moving lemma, following arguments of Druzhinin.
We will in fact establish the following stronger result.
\begin{theorem}\label{thm:moving}
Let $T \in \Sm_k$ $c \in h^{\efr,n}(X, U)(T)$, $V \subset \Gamma_{d'}$.
There exists $d''>d'$ such that for all $d \ge d''$, there is an open, non-empty subset $U_d \subset \Gamma_d$ of ``allowable cone points''.
(That is, any $\vec \gamma \in U_d$ has the required properties for the single correspondence $c$.)
\end{theorem}
Lemma~\ref{lemm:key} follows from this by applying the Theorem to each $(T_i,c_i,V_i)$ and picking a rational point in the intersection of the sets $U_d$ obtained (which is possible because this intersection is a non-empty, open subset of an affine space and $k$ is infinite).

We spend the rest of the section proving this result.
Fix \[ c = (W, (\phi, g), Z) \in h^{\efr}(X,U)(T). \]
The canonical map (induced by $W \to T$ and $(\phi, g): W \to \A^n \times X$) \[ \psi: W \to T \times \A^n \times X \] is finite over $T \times 0 \times Y$.
Since the quasi-finite locus is open \cite[Tag 01TI]{stacks-project}, there exists an open neighborhood $W' \subset W$ of $Z$ such that $\psi|_{W'}$ is quasi-finite.
Replacing $W$ by $W'$, we may assume that $\psi$ is quasi-finite.
Let $m>0$ and consider the map \[ \psi^m: W^{\times_T m} \to T \times (\A^n \times X)^m. \]
It is still quasi-finite.
Define \[ T \times (\A^n \times X)^m \supset \scr E_m := \{ (t, p_1, \dots, p_m) \mid t \in T, p_i \in \A^n \times X, p_i \ne p_j, p_i \not\in 0 \times Y \}. \]
Consider further \begin{gather*} W^{\times_T m} \times \Gamma_d \supset B_{m,d} := \{(w_1, \dots, w_m, \vec{s}) \mid \psi(w_1, \dots, w_m) \in \scr E_m, (f_{\vec s} \circ (\phi, g))(w_i) \in 0 \times (X \setminus Y), \\ (f_{\vec s} \circ (\phi, g))^{-1}(0 \times (X \setminus Y)) \text{ not quasi-finite at $w_i$} \} \end{gather*} and \[ T \times \Gamma_d \supset B_d := \{(t,\vec{s}) \mid (f_{\vec s} \circ (\phi, g))^{-1}(0 \times (X \setminus Y)) \text{ not quasi-finite over $t$} \}. \]
There is an evident map $B_{m,d} \to B_d$.
We shall prove the following:
\begin{enumerate}
\item For any $m, d$, the map $B_{m,d} \to B_d$ is surjective with fibers of dimension $\ge m$.
\item For fixed $m$, and $d=d(m)$ sufficiently large, we have $\dim B_{m,d} \le \dim T + \dim \Gamma_d$.
\end{enumerate}
We deduce that \[ \dim B_d \stackrel{(1)}{\le} \dim B_{m,d} - m \stackrel{(2)}{\le} \dim \Gamma_d + \dim T - m. \]
Choosing $m \ge \dim T + \dim V + 2$, we can ensure that \[ \dim B_d \le \dim \Gamma_d - \dim V - 2. \]
Write $p: B_d \amalg \{\vec x\} \to \Gamma_d$ for the projection and inclusion.
Write \[ q: (B_d \amalg \{\vec x\}) \times V \times \A^1 \to \Gamma_d, (b, v, t) \mapsto tp(b) + (1-t)v. \]
Then the image of $q$ has dimension $<\dim \Gamma_d$, and so the complement of the closure of the image of $q$ is a non-empty open $U_d \subset \Gamma_d$.

\begin{proof}[Proof of Theorem \ref{thm:moving}]
Let $\vec s \in \Gamma_d \setminus p(B_d)$.
We claim that $\vec s \cdot c \in h^{\efr,n}_\qf(X,U)$.
Indeed if $\varphi = f_{\vec s} \circ (\phi, g)$ then we know that $\varphi^{-1}(0 \times Y) = Z \amalg Z'$.
We also know that $\varphi^{-1}(0 \times (X \setminus Y))$ is quasi-finite over $T$.
Replacing $W$ by $W \setminus Z'$ we arrange that $\varphi^{-1}(0 \times Y) = Z$ is finite over $T$.
The claim follows.

Now let $\vec \gamma \in U_d$, $v \in V$, $t \in \A^1$.
If $t\vec \gamma + (1-t)v = b \in p(B_d) \cup \{\vec x\}$ with $t \ne 0$ then \[ \vec \gamma = 1/t\, b + (t-1)/t\, v, \] which contradicts the construction of $U_d$.
In other words, if $V'$ is the cone on $V$ with tip $\vec \gamma$ and $b \in V' \setminus V$, then $b \not\in B$ and so $b \cdot c$ is quasi-finite, as needed.
Similarly $\vec x \not\in V'$ unless $\vec x \in V$.
\end{proof}

The main idea for proving (1) is that if a morphism (of finite type, say) is not quasi-finite over some point, then the fiber must have dimension $\ge 1$.
Taking $m$-fold products, we obtain something of dimension $\ge m$.
\begin{proof}[Proof of (1).]
We may base change to an algebraically closed field, and it suffices to treat fibers over closed (hence rational) points.\NB{details}
Thus let $t \in T, \vec s \in \Gamma_d$ be closed points with $(t,\vec s) \in B_d$.
Set $\varphi = f_{\vec s} \circ (\phi, g): W \to \A^n \times X$, so that $A := \varphi^{-1}(0 \times (X \setminus Y))$ is not quasi-finite over $t$.
Let $A_1 \subset A$ be a positive dimensional component of the fiber over $t$ (which exists because $A$ is not quasi-finite over $t$).
Since $\psi$ is quasi-finite, \[ B := \psi(A_1) \subset \{t\} \times (\A^n_X \setminus 0_Y) \subset T \times \A^n \times X \] is infinite.
By Chevalley's theorem \cite[Tag 054K]{stacks-project}, $B$ is a finite disjoint union of locally closed subsets, and hence contains an infinite subset $B_0 \subset B$ which is a scheme.
Being of finite type over a field, $B_0$ has positive dimension.
Let $C \subset B_0^m$ be the subscheme of distinct points.
Since $\dim B_0 \ge 1$ we have $\dim C \ge m$.
By construction the image of $(B_{m,d})_{t,\vec s} \to \scr E_m$ contains $C$.
It follows that \[ \dim (B_{m,d})_{t, \vec s} \ge \dim C \ge m, \] as needed.
\end{proof}

For proving (2), we may (and will) ignore the quasi-finiteness condition in the definition of $B_{m,d}$.
The main idea is that the condition $f_{\vec s}((\phi,g)(w)) \in 0 \times X$ is equivalent to the vanishing of $n$ sections at $w$, and hence $m$ such conditions should have codimension $mn = \dim W^{\times_T m} - \dim T$.
\begin{proof}[Proof of (2).]
We may base change to an algebraically closed field.
Let \[ W^{\times_T m} \supset \scr W_m := \psi^{-1}(\scr E_m) \] so that we have a map $q: B_{m,d} \to \scr W_m$.
Since $\dim \scr W_m \le \dim T + mn$, it will suffice to show that the fibers of $q$ (over closed points) have dimension $\le \dim \Gamma - mn$.
Let $(w_1, \dots, w_m) \in \scr W_m$ have image $(p_1, \dots, p_m) \in \scr E_m$.
Put \[ \Gamma_{d,(p_1, \dots, p_m)} = \{\vec s \in \Gamma_d \mid f_{\vec s}(p_i) \in 0 \times X\}. \]
Then $q^{-1}(w_1, \dots, w_m) \subset \Gamma_{d,(p_1,\dots,p_m)}$ and hence it suffices to show that $\dim \Gamma_{d,(p_1, \dots, p_m)} \le \dim \Gamma_d - mn$.
We have an exact sequence \[ 0 \to \Gamma_{d,(p_1,\dots,p_m)} \to \Gamma_d \xrightarrow{ev} \bigoplus_i H^0(p_i, \scr O(d)^{\oplus n}). \]
Since the right hand term has dimension $mn$, it suffices to prove that the evaluation map $ev$ is surjective.
Set $\scr K = \ker(\scr O(d)^{\oplus n} \to \scr O(d)^{\oplus n}|_{\scr N})$, so that $\Gamma_d = \{\vec x\} + H^0(\P^n \times \overline{X}, \scr K)$.
By construction $p_i \not\in \scr N$, and hence $\scr K|_{p_i} = \scr O(d)^{\oplus n}|_{p_i}$.
The result thus follows from Lemma \ref{lem:standard} below (applied with $\scr F = \scr K$).
\end{proof}

We used the following well-known result.\todo{reference??}
\begin{lemma} \label{lem:standard}
Let $X$ be a projective scheme over a field, $\scr F$ a coherent sheaf on $X$ and $m \ge 0$.
There exists $N$ such that for all $d \ge N$ and distinct rational points $p_1, \dots, p_m \in X$, the map \[ H^0(X, \scr F(d)) \to \bigoplus_i H^0(p_i, \scr F(d)) \] is surjective.
\end{lemma}
\begin{proof}
Replacing $\scr F$ by the pushforward along an embedding of $X$ into projective space, we may assume that $X = \P^n$.
Given a surjection $\scr F' \to \scr F$, the result for $\scr F'$ implies the one for $\scr F$.
The result for $\scr F_1, \scr F_2$ implies it for $\scr F_1 \oplus \scr F_2$.
Hence it suffices to prove the result for $\scr F = \scr O$ (use \cite[Tag 01YS]{stacks-project}).
We can find $L_{ij} \in H^0(\P^n, \scr O(1))$ such that $L_{ij}(p_i) = 0$ but $L_{ij}(p_j) \ne 0$.
Then for fixed $j$, the section \[ s_j = \prod_{i \ne j} L_{ij} \in H^0(\P^n, \scr O(m-1)) \] has $s_j(p_j) \ne 0$ but $s_j(p_i) = 0$ for all $i \ne j$.
This shows that $N=m-1$ works (in this case).
\end{proof}

\section{The cancellation theorem}
\label{sec:cancellation}
Primary sources: \cite{EHKSY1,voevodsky2002cancellation,agp}.

After these lecture notes were written, some of the ideas from this section were used in \cite{bachmann-cancel}; that work may also serve as a somewhat more formal exposition of some of the ideas presented here.

\subsection{Group-complete framed spaces}
\begin{lemma}[\cite{EHKSY1}, Proposition 3.2.10(iii)]
The category $\Spc^\fr(S)$ is semiadditive.
\end{lemma}

It follows that, for every $\scr X \in \Spc^\fr(S)$ and $X \in \Sm_S$, $\pi_0 \scr X(X)$ is an abelian monoid.
\begin{definition}
We call $\scr X$ group-complete (or grouplike) if $\pi_0 \scr X(X)$ is, for every $X \in \Sm_S$.
We denote by $\Spc^\fr(S)^\gc \subset \Spc^\fr(S)$ the subcategory of group-complete spaces.
\end{definition}

The group-complete spaces are closed under limits and filtered colimits (in fact all colimits), and hence the inclusion $\Spc^\fr(S)^\gc \subset \Spc^\fr(S)$ admits a left adjoint $\scr X \mapsto \scr X^\gc$ which is easily seen to be symmetric monoidal.
The functor $\Omega^\infty: \SH(S) \wequi \SH^\fr(S) \to \Spc^\fr(S)$ has image contained in $\Spc^\fr(S)^\gc$\NB{justification?}.
It follows that $\Sigma^\infty: \Spc^\fr(S) \to \SH^\fr(S) \wequi \SH(S)$ inverts group completions and so factors through a symmetric monoidal, cocontinuous functor \[ \Sigma^\infty: \Spc^\fr(S)^\gc \to \SH(S). \]

The following is the main result.
\begin{theorem}[$\P^1$-cancellation] \label{thm:cancel-main}
If $k$ is a perfect field, then \[ \Sigma^\infty: \Spc^\fr(k)^\gc \to \SH(k) \] is fully faithful.
\end{theorem}

\begin{remark}
The essential image of $\Sigma^\infty$ is closed under colimits and known as the subcategory of \emph{very effective spectra}.
\end{remark}

\begin{remark}
The theorem is equivalent to showing that for $\scr X, \scr Y \in \Spc^\fr(k)^\gc$ we have $\Map(\scr X, \scr Y) \wequi \Map(\Sigma_{\P^1} \scr X, \Sigma_{\P^1} \scr Y)$, and this is further equivalent to showing that \[ \scr Y \to \Omega_{\P^1} \Sigma_{\P^1} \scr Y \] is an equivalence.
Here $\Sigma_\P^1: \Spc^\fr(k)^\gc \to \Spc^\fr(k)^\gc$ is the functor of tensor product with the image of $\P^1$ in $\Spc^\fr(k)^\gc$.
\end{remark}

Since $\P^1 \wequi S^1 \wedge \Gm$, it suffices to prove separate statements for these two suspensions.
This is how we shall establish Theorem \ref{thm:cancel-main}.

\subsection{$S^1$-cancellation}
\begin{proposition} \label{prop:S1-cancellation}
For $\scr X \in \Spc^\fr(k)^\gc$, the canonical map \[ \scr X \to \Omega_{S^1}\Sigma_{S^1}  \scr X \] is an equivalence.
\end{proposition}
\begin{proof}
Let $\scr Y \in \PSh_\Sigma(\Cor^\fr(S))^\gc$.
We shall first determine $\Sigma_{S^1} \scr Y$.
Let $X \in \Sm_S$.
There is a finite coproduct preserving functor $c_X: \Span(\Fin) \to \Cor^\fr(S)$ sending $*$ to $X$.
Its sifted-cocontinuous extension admits a right adjoint $c_{X*}: \PSh_\Sigma(\Cor^\fr(S)) \to \PSh_\Sigma(\Span(\Fin)) \wequi \CMon(\Spc)$ \cite[Proposition C.1]{bachmann-norms} which preserves limits and sifted colimits, and hence all colimits by semiadditivity and \cite[Lemma 2.8]{bachmann-norms}.
We deduce that \begin{equation}\label{eq:S1-suspension} (\Sigma_{S^1} \scr Y)(X) \wequi \Sigma_{S^1}(\scr Y(X)) \in \CMon(\Spc). \end{equation}
This implies both that $\Sigma_{S^1} \scr Y$ is group-complete and, using that $\CMon(\Spc)^\gc \wequi \SH_{\ge 0}$ \cite[Remark 5.2.6.26]{HA}, that \[ \scr Y \to \Omega_{S^1} \Sigma_{S^1} \scr Y \in \PSh_\Sigma(\Cor^\fr(S))^\gc \] is an equivalence.
To promote this to the same statement for $\scr X \in \Spc^\fr(S)^\gc$, it is enough to show that whenever $\scr Y$ is motivically local, the same holds for $L_\Nis \Sigma_{S^1} \scr Y$; indeed $\Omega_{S^1}$ is computed sectionwise and hence preserves Nisnevich equivalences.
Equation \eqref{eq:S1-suspension} shows that $\Sigma_{S^1} \scr Y$ is $\A^1$-invariant; the result thus follows from Corollary \ref{cor:strictly-invariant} in \S\ref{sec:strict-A1}.
\end{proof}

\subsection{Abstract cancellation}
The following is extracted from \cite[\S4]{voevodsky2002cancellation}.

\begin{theorem} \label{thm:abstract-cancellation}
Let $\scr C$ be a symmetric monoidal $1$-category and $G \in \scr C$ a symmetric object.
Suppose that the functor $\Sigma_G := G \otimes \ph$ admits a right adjoint $\Omega_G$.
Note that $\Omega_G$ is canonically a lax $\scr C$-module functor.
Suppose that the unit transformation \[ u: \id_{\scr C} \to \Omega_G \Sigma_G \] admits a retraction $\rho$ in the category of lax $\scr C$-module functors.
Then $u, \rho$ are inverse isomorphisms.
\end{theorem}

\begin{remark} \label{rmk:infinity-cancellation}
If $\scr C$ is an $\infty$-category and $\rho$ is a lax $\scr C$-module retraction of $u: \id_{\scr C} \to \Omega_G \Sigma_G$, then the same conclusion holds (apply the theorem to $h\scr C$).
\end{remark}

\begin{remark}
Since $\scr C$ is a $1$-category, a lax $\scr C$-module structure on an endofunctor $F: \scr C \to \scr C$ just consists of compatible morphisms $X \otimes F(Y) \to F(X \otimes Y)$ for all $X, Y \in \scr C$.
Moreover a transformation $\alpha: F \to G$ being a lax $\scr C$-module transformation is a property: it is the requirement that for $X, Y \in \scr C$, the following square commutes
\begin{equation*}
\begin{CD}
X \otimes F(Y) @>{\id_X \otimes \alpha_Y}>> X \otimes G(Y) \\
@VVV                                           @VVV        \\
F(X \otimes Y) @>{\alpha_{X \otimes Y}}>>   G(X \otimes Y).
\end{CD}
\end{equation*}
\end{remark}

\begin{example} \label{ex:id-endo}
A lax $\scr C$-module transformation $\alpha: \id \to \id$ (of $\id_{\scr C}$ with its canonical $\scr C$-module structure) is completely determined by $\alpha_\1: \1 \to \1$.
In particular $\rho$ being a retraction of $u$ is equivalent to the composite \[ \1 \xrightarrow{u_\1} \Omega_G G \xrightarrow{\rho_\1} \1 \] being the identity.
\end{example}

To simplify notation, from now on we will write $\iHom(G, \ph)$ for $\Omega_G$, and also use suggestive notation like $\otimes \id_Y: \iHom(A, B) \to \iHom(A \otimes Y, B \otimes Y)$, when convenient.
\begin{lemma} \label{lemm:cancellation-technical}
For $X, Y \in \scr C$, the following diagram commutes
\begin{equation*}
\begin{CD}
\iHom(G, G \otimes X) @>{\rho_X}>> \iHom(\1, X) \\
@V{\otimes \id_Y}VV               @V{\otimes \id_Y}VV \\
\iHom(G \otimes Y, G \otimes X \otimes Y) @>{\Omega_Y\rho_{X \otimes Y}}>> \iHom(Y, X \otimes Y).
\end{CD}
\end{equation*}
\end{lemma}
\begin{proof}
Decompose the diagram as
\begin{equation*}
\begin{CD}
\iHom(G, G \otimes X) @>{\rho_X}>> \iHom(\1, X) \\
@VuVV                              @VuVV        \\
\iHom(Y, Y \otimes \iHom(G, G \otimes X)) @>{\iHom(Y, Y \otimes \rho_X)}>> \iHom(Y, Y \otimes \iHom(\1, X)) \\
@VVV                                                                       @VVV \\
\iHom(Y, \iHom(G, G \otimes X \otimes Y)) @>{\iHom(Y, \rho_{X \otimes Y})}>> \iHom(Y, \iHom(\1, X \otimes Y)) \\
@V{\wequi}VV                                                               @V{\wequi}VV \\
\iHom(G \otimes Y, G \otimes X \otimes Y) @>{\Omega_Y\rho_{X \otimes Y}}>> \iHom(Y, X \otimes Y).
\end{CD}
\end{equation*}
Here the middle vertical transformations are the lax module structure maps, and the bottom vertical isomorphisms hold in any symmetric monoidal category.
The upper and lower squares commute by naturality, and the middle one by assumption of $\rho$ being a lax module transformation.
The vertical composites are given by $\otimes \id_Y$.
This concludes the proof.
\end{proof}

\begin{proof}[Proof of Theorem \ref{thm:abstract-cancellation}.]
Let $X \in \scr C$.
It suffices to show that the composite $\Omega_G \Sigma_G X \xrightarrow{\rho_X} X \xrightarrow{u_X} \Omega_G \Sigma_G X$ is the identity.
Let $n \ge 2$ and $\alpha: G^{\otimes n} \to G^{\otimes n}$ be an automorphism.
Consider the composite \[ p(\alpha): \iHom(G, G \otimes X) \xrightarrow{\id_{G^{\otimes n-1}} \otimes} \iHom(G^{\otimes n}, G^{\otimes n} \otimes X) \xrightarrow{c_\alpha} \iHom(G^{\otimes n}, G^{\otimes n} \otimes X) \xrightarrow{\rho^{n-1}} \iHom(G, G \otimes X), \] where $c_\alpha$ denotes the conjugation by $\alpha$.

Note that the map ``$\id_{G^{\otimes n-1}} \otimes$'' is a composite of units $u$ and hence by assumption of $\rho$ being a retraction, we get $p(\id) = \id$.

On the other hand let $\alpha = \sigma$ be the cyclic permutation.
Then the first $n-2$ applications of $\rho$ are again ``cancelling out identities'', so that $p(\sigma)$ is the same as the composite \[ \iHom(G, G \otimes X) \xrightarrow{f} \iHom(G^{\otimes 2}, G^{\otimes 2} \otimes X) \xrightarrow{f_2} \iHom(G, G \otimes X), \] where $f_1$ ``inserts $\id_G$ in the middle'', and ``$f_2$ applies $\rho$ at the front''. \todo{I still wish this could be explained better.}
Lemma \ref{lemm:cancellation-technical} implies that this is the same as $u_X \rho_X$.

Hence if $G$ is $n$-symmetric, then since $\sigma = \id$ we find that \[ u_X \rho_X = p(\sigma) = p(\id) = \id. \]
This concludes the proof.
\end{proof}

\subsection{Twisted framed correspondences}
Using \cite[\S B]{EHKSY3} it is possible to construct a symmetric monoidal $\infty$-category $\Cor^\fr_L(S)$ with the following properties:
\begin{itemize}
\item Its objects are pairs $(X, \xi)$ with $X \in \Sm_S$ and $\xi \in K(X)$.
\item The morphisms from $(X, \xi)$ to $(Y, \zeta)$ are given by spans \[ X \xleftarrow{f} Z \xrightarrow{g} Y, \] where $Z$ is a \emph{derived} scheme and $f$ is a quasi-smooth morphism, together with a trivialization \[ f^*(\xi) + L_f \wequi g^*(\eta) \in K(Z). \]
\item There is a symmetric monoidal functor $\delta: \Cor^\fr(S) \to \Cor^\fr_L(S)$ which sends $X$ to $(X, 0)$ and induces the evident maps on mapping spaces.
\end{itemize}
It follows that the tensor product in $\Cor^\fr_L(S)$ is given by the product of schemes, and the functor $\delta$ is faithful (induces monomorphisms on mapping spaces).

The following will be helpful.
\begin{lemma} \label{lemm:recognize-finite}
A span \[ X \xleftarrow{f} Z \to Y \in \Map_{\Cor^\fr_L(S)}((X, 0), (Y, 0)) \] is in the image of $\delta$ if and only if $f$ is finite.
\end{lemma}
\begin{proof}
The only concern is that $Z$ might be a derived scheme instead of a classical one; by \cite[Lemma 2.2.1]{EHKSY3} this cannot happen.
\end{proof}

We mainly introduce the category $\Cor_L^\fr(S)$ for the following technically convenient reason: all of its objects are strongly dualizable.
\begin{proposition} \label{prop:duals}
Let $X \in \Sm_S$.
The spans \[ * \leftarrow X \xrightarrow{\Delta} X \times X \] and \[ X \times X \xleftarrow{\Delta} X \to * \] admit evident framings, and exhibit $(X, L_X)$ as the dual of $(X,0)$ in $\Cor^\fr_L(S)$.
\end{proposition}
\begin{proof}
This kind of duality happens in all span categories; we just need to verify that the spans are frameable and that the induced framings of the compositions are trivial.
All of this is easy to verify.
For example $X \times X$ really means $(X, 0) \otimes (X, L_X) = (X \times X, p_2^* L_X)$ and hence to frame the first span we need to exhibit a path \[ 0 + L_X \wequi \Delta^* p_2^* L_X, \] but this holds on the nose since $\Delta^* p_2^* \wequi \id$; to frame the second span we need to exhibit a path \[ \Delta^* p_2^* L_X + L_\Delta \wequi 0 \] which is possible in $K$-theory since the composite $X \xrightarrow{\Delta} X \times X \xrightarrow{p_1} X$ is the identity, so $0 = L_{\id} \wequi L_\Delta + \Delta^* L_{p_1}$ and finally $L_{p_1} \wequi p_2^* L_X$ by base change.
\end{proof}

The following will be helpful later to exhibit spans.
\begin{construction} \label{construction:spans}
Suppose given $X, G \in \Sm_S$, a map $f: X \times G \to \A^1$ and a path $L_G \wequi 1 \in K(G)$.
Then there is a span \[ D(f): (X \xleftarrow{p_{1}} Z(f) \xrightarrow{p_2} G) \in \Map_{\Cor^\fr_L(S)}((X, 0), (G, 0)); \] the framing is given by \[ L_{p_1} \wequi L_{Z(f)/X \times G} + L_{X \times G/X} \wequi -1 + L_G \wequi 0 \in K(Z(f)), \] where we have used that $L_{Z(f)/X \times G} \wequi -1$ via $f$ and $L_G \wequi 1$ by assumption.
\end{construction}
We will always apply this construction with $G = \A^1 \setminus 0$, so that there is a canonical trivialization of $L_G$.

\subsection{A general construction}
Given $X, Y \in \Sm_S$, for notational convenience we will write $f: X \rightsquigarrow Y$ for $f \in \Map_{\Cor^\fr_L(S)}((X, 0), (Y, 0))$.
\begin{construction}
Let $A, G \in \Sm_S$ and $\alpha: A \times G \rightsquigarrow G$.
We obtain a $\Cor^\fr_L(S)$-module transformation \[ \rho_\alpha: \Omega_G \Sigma_G \to \Omega_A \in \End(\PSh_\Sigma(\Cor^\fr_L(S))) \] as follows: via strong dualizability (Proposition \ref{prop:duals}), we can rewrite the source and target and consider the transformation \[ G^\vee \otimes G \otimes \ph \xrightarrow{\alpha^\vee \otimes \id_{\ph}} A^\vee \otimes \ph \] where $\alpha^\vee: G^\vee \otimes G \to A^\vee$ is obtained from $\alpha$ in the evident manner.
\end{construction}
We will eventually apply this with $G = \A^1 \setminus 0$ and $A = \A^1$ or $A = *$.

\begin{remark} \label{rmk:explicit-construction}
Let $X, Y \in \Sm_S$.
Given a span \[ G \times Y \leftarrow Z \to G \times X, \] the transformation $\rho_\alpha$ produces a span \[ A \times Y \leftarrow \rho_\alpha(Z) \to X. \]
Write $\alpha$ as \[ A \times G \leftarrow C \to G. \]
Tracing through the definitions, one finds that \[ \rho_\alpha(Z) = Z \times_{G \times G} C, \] with an evident induced framing.
\end{remark}

\begin{lemma} \label{lemm:cancel-technical}
The transformation $\rho_\alpha$ satisfies the following properties.
\begin{enumerate}
\item Given $Z': X \rightsquigarrow X'$ and $Z: G \times Y \rightsquigarrow G \times X$ we have \[ \rho_\alpha((\id_G \otimes Z') \circ Z) \wequi (\id_A \otimes Z') \circ \rho_\alpha(Z). \]
\item Given $Z': Y \rightsquigarrow Y$ and $Z: G \times Y \rightsquigarrow G \times X$ we have \[ \rho_\alpha(Z \circ (\id_G \otimes Z')) \wequi \rho_\alpha(Z) \circ (\id_G \otimes Z'). \]
\item Given $i: A' \rightsquigarrow A$, we have \[ \rho_{i^* \alpha} \wequi i^* \rho_\alpha. \]
\end{enumerate}
\end{lemma}
\begin{proof}
Evident from the naturality of the construction.
\end{proof}

Now define \[ M_\alpha(Y, X) \subset \Map_{\Cor^\fr(S)}(G \times Y, G \times X) \] to consist of the disjoint union of those path components corresponding to spans $G \times Y \leftarrow Z \to G \times X$ such that $\rho_\alpha(Z)$ is finite.
Then (1) and (2) of Lemma \ref{lemm:cancel-technical} translate (using Lemma \ref{lemm:recognize-finite}) into
\begin{enumerate}
\item $(\id_G \otimes Z') \circ M_\alpha(Y, X) \subset M_\alpha(Y, X')$, and
\item $M_\alpha(Y, X) \circ (\id_G \otimes Z') \subset M_\alpha(Y', X)$.
\end{enumerate}

\begin{construction} \label{construction:transformation}
Define a subfunctor \[ F_\alpha \Omega_G \Sigma_G \hookrightarrow \Omega_G \Sigma_G \in \End(\PSh_\Sigma(\Cor^\fr(S))) \] via \[ (F_\alpha \Omega_G \Sigma_G X)(Y) = M_\alpha(X, Y). \]
The lax monoidal natural transformation \[ \Omega_G \delta_*\delta^* \Sigma_G \wequi \delta_* \Omega_G \Sigma_G \delta^* \xrightarrow{\rho_\alpha} \delta_* \Omega_A \delta^* \wequi \Omega_A \delta_* \delta^* \] restricts by construction to a natural transformation \[ \rho_\alpha: F_\alpha \Omega_G \Sigma_G \to \Omega_A, \] which we will think of as \[ \rho_\alpha: A \otimes F_\alpha \Omega_G \Sigma_G \to \id. \]
\end{construction}
Take $A = \A^1$, $G = \A^1 \setminus 0$ and suppose that $\rho_\alpha(\id_G)$ is finite.
Then the unit transformation \[ \id \to \Omega_G \Sigma_G \] factors through $F_\alpha \Omega_G \Sigma_G$.
Moreover we obtain two $\A^1$-homotopic transformations \[ \rho_{i_0^* \alpha}, \rho_{i_1^* \alpha}: F_\alpha \Omega_G \Sigma_G \to \id \in \End(\PSh_\Sigma(\Cor^\fr(S))). \]

\subsection{$\Gm$-cancellation}
Let $G = \A^1 \setminus 0$.
\begin{definition}
We define maps $G \times G \to \A^1$ via \[ g_n^+(t_1, t_2) = t_1^n + 1 \text{ and } g_n^-(t_1,t_2) = t_1^n + t_2. \]
We further define maps $\A^1 \times G \times G \to \A^1$ via \[ h_n^\pm(t, t_1, t_2) = t g_n^\pm(t_1, t_2) + (1-t) g_m^\pm(t_1, t_2). \]
\end{definition}

Recall the associated spans from Construction \ref{construction:spans}.
Put \[ F_i = \bigcap_{m,n\ge i} [F_{D(h_{m,n}^+)} \cap F_{D(h_{m,n}^-)}] \subset \Omega_G \Sigma_G. \]

\begin{lemma}
We have \[ \colim_{i} F_i \wequi \Map_{\Cor^\fr(S)}(X, Y). \]
\end{lemma}
\begin{proof}
We follow \cite[Lemma 4.1 and Remark 4.2]{voevodsky2002cancellation}.
Suppose given $Y \leftarrow Z \to X \in \Map_{\Cor^\fr(S)}(Y, X)$.
We shall exhibit an integer $N$ such that for all $m, n > N$ the projection $Z' = \rho_{D(h_{m,n}^\pm)}(Z) \to Y \times \A^1$ is finite; this will prove what we want.
Write $f_1, f_2: Z \to G$ for the two projections.
Using Zariski's main theorem, we can form a commutative diagram
\begin{equation*}
\begin{CD}
Z @>>> \bar C \\
@V{f_1 \times p_Y}VV  @V{\bar f_1 \times p_Y}VV \\
G \times Y @>>> \P^1 \times Y,
\end{CD}
\end{equation*}
where $\bar f_1 \times p_Y$ is finite.
There exists $N$ such that the rational function $\bar f_1^N/f_2$ is regular in a neighbourhood $U_0$ of $\bar f_1^{-1}(0)$ and $f_2/\bar f_1^N$ is regular in a neighbourhood $U_\infty$ of $\bar f_1^{-1}(\infty)$.\todo{details?}
We have the function $h = h_{m,n}^\pm(t, f_1, f_2)$ on $Z \times \A^1$, and Remark \ref{rmk:explicit-construction} implies that $Z' = Z(h) \subset Z \times \A^1$.
The composite $\bar C \times \A^1 \to \P^1 \times Y \times \A^1 \to Y \times \A^1$ is projective, and $Z(h) \to Y \times \A^1$ is affine.
We will finish the proof by showing that $i^\pm: Z(h) \to \bar C \times \A^1$ is a closed immersion for $n,m > N$; indeed then $Z(h) \to Y \times \A^1$ will be both proper and affine, and hence finite as desired.

Note that $h^+$ extends to the regular map $t \bar f_1^m + (1-t) \bar f_1^n + 1: \bar C \to \P^1$, which does not vanish if $\bar f_1 \in \{0, \infty\}$.
Thus $i^+$ is always a closed immersion.

Now we deal with $i^-$.
Let $U_1 = \bar f_1^{-1}(G)$.
A morphism being a closed immersion is local on the target \cite[Tag 01QO]{stacks-project}, so it is enough to show that $i$ is a closed immersion over $U_0, U_\infty$ and $U_1$.
This is clear for $U_1$.
Consider the function $h_0 = t\bar f_1^n/f_2 + (1-t)\bar f_1^m/f_2 + 1$.
By construction, this is regular on $h_0$, so $Z(h_0) \subset U_0$ is closed.
Also by construction, $h_0 = 1$ if $\bar f_1 = 0$, and $h^- = f_2 h_0$ on $U_0 \setminus 0$, where $f_2$ is a unit.
It follows that $Z(h_0) = U_0 \cap Z(h)$.
A similar argument works for $U_\infty$.
\end{proof}

Using Construction \ref{construction:transformation}, we thus obtain a sequence of lax module transformations
\begin{equation*}
\begin{tikzcd}
\id \ar[r, dashed] & F_0 \ar[r, hookrightarrow] \ar[l, "\rho_0^\pm", bend left=20] & F_1 \ar[r, hookrightarrow] \ar[ll, "\rho_1^\pm" swap, bend right=20] & \dots \ar[r, hookrightarrow] \ar[lll, "\rho_2^\pm" swap, bend right=40] & \Omega_G \Sigma_G,
\end{tikzcd}
\end{equation*}
where the arrows to the right form a colimit diagram.
The dashed arrow might not exist, but the lemma above implies that its composite sufficiently far to the right does, and this is all we need.\footnote{One may verify that the arrow actually does exist.}
For $m \ge n$, the $h_{m,n}^\pm$ induce $\A^1$-homotopies making the following diagram commute
\begin{equation*}
\begin{tikzcd}
F_n \ar[r, hookrightarrow] \ar[d, "\rho_n^\pm"] & F_m \ar[dl, "\rho_m^\pm"] \\
\id
\end{tikzcd}
\end{equation*}
Applying $L_{\A^1}$, there are thus induced transformations on the colimit
\begin{equation*}
\begin{tikzcd}
L_{\A^1} \ar[r, "u"] & L_{\A^1} \Omega_G \Sigma_G \ar[l, "\rho^\pm", bend left=20].
\end{tikzcd}
\end{equation*}
After group completion, we may take the difference, and hence obtain \[ \rho = \rho^+ - \rho^-: L_{\A^1}^\gc \Omega_G \Sigma_G \to L_{\A^1}^\gc. \]

We are now ready to prove our main result.
\begin{theorem}
Let $k$ be an infinite perfect field.
Then the unit transformation \[ u: \id \to \Omega_{\Gm} \Sigma_\Gm \in \End(\Spc^\fr(k)^\gc) \] is an equivalence.
\end{theorem}
\begin{proof}
We seek to apply the abstract cancellation Theorem \ref{thm:abstract-cancellation} (in the guise of Remark \ref{rmk:infinity-cancellation}).
Note that $\Gm$ is symmetric in $\Spc^\fr(k)^\gc$: $T \wequi S^1 \wedge \Gm$ is symmetric by the usual argument, and $S^1$ is (symmetric and) semi-invertible (by $S^1$-cancellation, i.e. Proposition \ref{prop:S1-cancellation}).
We have already constructed a lax module transformation \[ \rho: L_\mot^\gc \Omega_G \Sigma_G \to L_\mot^\gc. \]
Corollary \ref{cor:Lmot-OmegaGm} in \S\ref{sec:strict-A1} shows that $L_\mot^\gc \Omega_G \Sigma_G \wequi \Omega_G \Sigma_G L_\mot^\gc$ and hence we obtain a lax module transformation \[ \rho: \Omega_G \Sigma_G \to \id \in \End(\Spc^\fr(k)^\gc). \]
In $\Spc^\fr(k)^\gc$ there is a splitting $G \wequi \1 \oplus \Gm$, and hence a retraction $\Gm \to G \to \Gm$.
This induces a retraction of lax module functors \[ \Omega_\Gm \Sigma_\Gm \to \Omega_G \Sigma_G \to \Omega_\Gm \Sigma_\Gm, \] which in particular allows us to build the lax module transformation \[ \rho': \Omega_\Gm \Sigma_\Gm \to \Omega_G \Sigma_G \to \id. \]
In order to apply the abstract cancellation theorem, it remains to verify that $\rho' u \wequi \id$.
Via Example \ref{ex:id-endo}, for this it suffices to compute the effect of $\rho' u$ on $\id_\1$.
Now $u(\id_\1) = \id_\Gm$, which corresponds to $\id_G - p \in \Hom(G, G)$, where $p: G \to * \to G$, and so $\rho' u(\id_\1) = \rho(\id_G) - \rho(p)$.
The result thus follows from Lemma \ref{lemm:cancel-final} below.
\end{proof}

\begin{lemma} \label{lemm:cancel-final}
For each $n > 0$ we have
\begin{enumerate}
\item $\rho_n^+(p) = \rho_n^-(p)$, and
\item $\rho_n^+(\id_G) \stackrel{\A^1}{\wequi} \rho_n^-(\id_G) + \id_\1$.
\end{enumerate}
\end{lemma}
\begin{proof}
This is essentially \cite[Lemma 4.3]{voevodsky2002cancellation}.

Note that $p$ is represented by the correspondence $G \xleftarrow{\wequi} G \xrightarrow{1} G$, so that by Remark \ref{rmk:explicit-construction}, $\rho_n^\pm(p)$ is represented by $Z(g_n^\pm(t,1)) \to *$.
But $g_n^+(t,1) = g_n^-(t,1)$, whence (1).

Similarly $\rho_n^\pm(\id_G)$ is represented by $Z_\pm := Z(g_n^\pm(t,t))$, so $Z_+ = Z(t^n+1)$ and $Z_- = Z(t^n+t)$, where both $t^n+1, t^n+t$ are viewed as functions on $\A^1 \setminus 0$.
Consider $H = D(t^n + ts + 1-s): \A^1 \rightsquigarrow *$, where we view $h$ as a function $\A^1 \times \A^1 \to \A^1$.
Then $H$ provides an $\A^1$-homotopy between $D(t^n+1)$ and $D(t^n+t)$, where this time we view $t^n+1, t^n+t$ as functions on $\A^1$.
Now \[ Z(t^n+1|\A^1) = Z(t^n+1|\A^1 \setminus 0) = Z^+, \] whereas \[ Z(t^n+t|\A^1) = Z(t^n+t|\A^1 \setminus 0) \amalg \{0\} = Z^- \amalg \{0\}. \]
Since $0 \subset \A^1 \to *$ defines the identity correspondence, $H$ provides the desired homotopy.

This concludes the proof.
\end{proof}

\section{Strict $\A^1$-invariance}
\label{sec:strict-A1}
Primary sources: \cite{hitr,DruzhininPanin}.

\newcommand{\cok}{\mathrm{cok}}

\subsection{Introduction}
The title of this section derives from the following.
Write $\PSh_\Sigma(\Sm_k, \Ab)$ for the category of additive presheaves of abelian groups on $\Sm_k$.
\begin{definition}
Let $F \in \PSh_\Sigma(\Sm_k, \Ab)$.
Then $F$ is called \emph{$\A^1$-invariant} (or sometimes \emph{($\A^1$-)homotopy invariant}) if for all $X \in \Sm_k$, the canonical map $F(X) \to F(X \times \A^1)$ is an isomorphism.

Moreover $F$ is called \emph{strictly $\A^1$-invariant} if for all $n \ge 0$ and all $X \in \Sm_k$ the canonical map $H^n_\Nis(X, F) \to H^n_\Nis(X \times \A^1, F)$ is an isomorphism.
\end{definition}

\begin{remark} Observe that if $F$ is an abelian presheaf, then $F$ is $\A^1$-invariant if and only if the map $F(X \times \A^1) \rightarrow F(X)$ induced by the zero section $X \rightarrow X \times \A^1$ is injective. We will use this without further comment throughout the sequel.
\end{remark}

There are two important observations regarding this:
\begin{enumerate}
\item If $F$ is an $\A^1$-invariant presheaf, it need not be the case that $a_\Nis F$ is $\A^1$-invariant (let alone strictly $\A^1$-invariant).
\item If $F$ is an $\A^1$-invariant sheaf, it need not be strictly $\A^1$-invariant.
\end{enumerate}
However, it turns out that in the presence of transfers, neither of these problems occurs.
The first general results in this direction were probably obtained by Voevodsky in \cite{voevodsky2000cohomological}.
Here is a version for framed presheaves.

\begin{theorem} \label{thm:strictly-invariant}
Let $k$ be a field, and $F \in \PSh_\Sigma(\Cor^\fr(k), \Ab)$.
Suppose that $F$ is $\A^1$-invariant.
\begin{enumerate}
\item For $U \subset \A^1$ open we have $H^0_\Nis(U, F) \wequi F(U)$.
\item The sheafification $a_\Nis F$ is $\A^1$-invariant.
\item If $k$ is perfect, then $a_\Nis F$ is strictly $\A^1$-invariant.
\end{enumerate}
\end{theorem}

We can escalate the above result as follows.
\begin{corollary} \label{cor:strictly-invariant}
Let $k$ be a perfect field, and $F \in \PSh_\Sigma(\Cor^\fr(k))^\gc$ be $\A^1$-invariant.
Then $L_\Nis F$ is $\A^1$-invariant, and hence motivically local.
\end{corollary}
\begin{proof}
By an induction on the Postnikov tower, or equivalently using the (strongly convergent) descent spectral sequence, this is immediate from Theorem \ref{thm:strictly-invariant}.
\end{proof}

We can also deduce the following fact, which is very important for the cancellation theorem.
\begin{corollary} \label{cor:Lmot-OmegaGm}
Let $k$ be a perfect field.
On the category $\PSh_\Sigma(\Cor^\fr(k))^\gc$, the canonical transformation $L_\mot \Omega_\Gm \to \Omega_\Gm L_\mot$ is an equivalence.
\end{corollary}
\begin{proof}
Using \cite[Lemma 6.1.3]{morel2005stable}, it suffices to prove that the map induces an equivalence on sections over fields.
Thus let $K/k$ be a field extension.
By Corollary \ref{cor:strictly-invariant}, $L_\mot = L_\Nis L_{\A^1}$.
Note that $\Omega_\Gm$ commutes with $L_{\A^1}$ (see e.g. \cite[Lemma 4]{bachmann-grassmann}) and fields are stalks for the Nisnevich topology; hence it is enough to show that \[ (\Omega_\Gm L_{\A^1} \scr X)(K) \to (\Omega_\Gm L_\Nis L_\A^1 \scr X)(K) \] is an equivalence.
By another induction on the Postnikv tower / descent spectral sequence, we reduce to showing that for $F \in \PSh_\Sigma(\Cor^\fr(k), \Ab)$ which is $\A^1$-invariant, one has \[ H^n_\Nis(\Gm_K, F) = \begin{cases} F(\Gm_K) & n = 0 \\ 0 & \text{else} \end{cases}. \]
The first case is immediate from Theorem \ref{thm:strictly-invariant}(1).
Now assume that $n \geq 1$. Theorem \ref{thm:strictly-invariant}(3) asserts that $a_\Nis F$ is strictly $\A^1$-invariant. Since $\P^1 \wequi \Sigma \Gm \in \Spc(k)$ we find that $H^n_\Nis(\Gm_K, F) = H^{n+1}_\Nis(\P^1_K, F)$.
The result thus follows from the fact that $\P^1$ has Nisnevich cohomological dimension one \cite[Proposition 3.1.8]{A1-homotopy-theory} and thus its Nisnevich cohomology vanishes whenever $n+1 \geq 2$.
\end{proof}

The remainder of this section is devoted to proving Theorem \ref{thm:strictly-invariant}.
\subsubsection*{Notation and conventions}
From now on, all cohomology will be Nisnevich cohomology, i.e. $H^* := H^*_\Nis$.

Given a scheme $X$ and a point $x \in X$, we write $X_x$ for the local scheme $\Spec(\scr O_{X,x})$ and $X_x^h$ for the henselian local scheme $\Spec(\scr O_{X,x}^h)$.
If $S \subset X$ is a finite set of points, we write $X_S$ for the semilocalization of $X$ in $S$ (see \S\ref{subsec:semilocal} for more about semilocalizations).

Recall that a scheme is called \emph{essentially smooth affine over $k$} if it can be written as a cofiltered limit of smooth affine $k$-schemes, with étale transition maps.
Observe that essentially smooth affine schemes are affine of finite dimension and have local rings which are integral domains, but need not be noetherian (though we will not use non-noetherian schemes in any relevant way).

\subsection{A formalism for strict $\A^1$-invariance} \label{subsec:formalism}
We shall prove Theorem \ref{thm:strictly-invariant} following the strategy explained in \cite{druzhinin2018cohomological,druzhinin2017strictly}.

Throughout we fix a field $k$.
Consider an abelian presheaf $F \in \PSh_\Sigma(\SmAff_k, \Ab)$.
As usual, we extend $F$ to essentially smooth affine schemes by continuity: if $X = \lim_i X_i$ where each $X_i$ is smooth affine (and the transition maps are étale, so that $X$ is essentially smooth), then $F(X) := \colim_i F(X_i)$.
We isolate the following four properties which $F$ may satisfy.

\begin{definition}[IA]
We say that $F$ satisfies \emph{injectivity on the affine line} (short \emph{IA}) if the following holds.
For any finitely generated, separable field extension $K/k$ (automatically essentially smooth) and open subschemes $\emptyset \ne V_1 \subset V_2 \subset \A^1_K$ (automatically affine), the restriction $F(V_2) \to F(V_1)$ is injective.
\end{definition}

\begin{definition}[EA]
We say that $F$ satisfies \emph{excision on the relative affine line} (short \emph{EA}) if the following holds.
For any essentially smooth affine scheme $U$ and affine open subscheme $V \subset \A^1_U$ containing $0_U$, restriction induces an isomorphism \[ F(\A^1_U \setminus 0_U)/F(\A^1_U) \wequi F(V \setminus 0_U)/F(V). \]
(Note that $\A^1_U \setminus 0_U$ and $V \setminus 0_U$ are indeed affine.)

Furthermore we require that if $K/k$ is a finitely generated, separable field extension, $z \in \A^1_K$ a closed point, $V \subset \A^1_K$ an open neighborhood of $z$, then \[ F(\A^1_K \setminus z)/F(\A^1_K) \wequi F(V \setminus z)/F(V). \]
\end{definition}

\begin{definition}[IL]
We say that $F$ satisfies \emph{injectivity for henselian local schemes} (short \emph{IL}) if the following holds.
For any essentially smooth, henselian local scheme $U$ with generic point $\eta$, the restriction $F(U) \to F(\eta)$ is injective.
\end{definition}

\begin{definition}[EE]\NB{awkward formulation?}
We say that $F$ satisfies \emph{étale excision} (short \emph{EE}) if the following holds.
Let $\pi: X' \to X$ local morphism of local schemes which can be obtained as a cofiltered limit of étale morphisms of smooth $k$-schemes (with étale transition maps).
Let $Z \subset X$ be a principal closed subscheme such that $\pi^{-1}(Z) \to Z$ is an isomorphism.
Then the canonical map \[ F(X \setminus Z)/F(X) \to F(X' \setminus \pi^{-1}(Z))/F(X') \] is an isomorphism.
\end{definition}

\begin{remark}
Observe that if $X$ is affine and $Z \subset X$ is a principal closed subscheme, then $X \setminus Z$ is (a principal open) affine.
The above axioms are often stated in a more general form without affineness or principality assumptions.
As we will see in this section, our weak form of the axioms is enough to deduce strict $\A^1$-invariance.
\end{remark}

We shall also use the notion of contraction.
\begin{definition}
Let $F$ be a presheaf.
We denote by $F_{-1}$ the presheaf $X \mapsto F(X \times \Gm)/F(X)$, and by $F_{-n}$ the $n$-fold iterate of this construction.
We also write $F_{-1}'$ for the presheaf $X \mapsto F(X \times \Gm)/F(X \times \A^1)$.
\end{definition}
\begin{remark}\label{rmk:contraction}
$F_{-1}'$ is the definition of contraction in \cite[\S23]{lecture-notes-mot-cohom}.
This yields the most natural result in Lemma \ref{lemm:principal-imm}.
For $\A^1$-invariant presheaves, the two notions coincide.
For non-$\A^1$-invariant presheaves, the definition of $F_{-1}$ we gave seems more standard.
\end{remark}

The main result of this section is as follows.
\begin{theorem} \label{thm:strict-A1-formal}
Let $k$ be a perfect field.
Let $\scr C$ be a collection of abelian presheaves on $\SmAff_k$ which is closed under $F \mapsto F_{-1}$ and $F \mapsto H^i(\ph, F)$.
Assume that whenever $F \in \scr C$ is an $\A^1$-invariant presheaf, then it satisfies IA, EA, IL and EE.

Let $F \in \scr C$ be $\A^1$-invariant.
Then for every essentially smooth (not necessarily affine) $k$-scheme $X$ we have \[ H^i(X \times \A^1, F) \wequi H^i(X, F). \]
\end{theorem}
Note that since the category of Nisnevich sheaves on $\Sm_k$ is the same as the category of Nisnevich sheaves on $\SmAff_k$, $H^i(X, F)$ makes sense even if $X$ is not affine. The next lemma states that Nisnevich sheafification of a presheaf with IA, EA, IL and EE does not change its values on opens of (relative) $\A^1$'s.

\begin{lemma} \label{lemm:strict-formal-field}
Let $K/k$ be a finitely generated, separable field extension and $F$ a presheaf satisfying IA, EA, IL and EE.
Let $U \subset \A^1_K$ be open.
Then $F(U) \wequi (a_\Nis F)(U)$ and $H^i(U, F) = 0$ for $i > 0$.
\end{lemma}
\begin{proof}
Let $X = \A^1_K$.

We first establish the following claim: $(*)$ if $U \subset \A^1_K$ is open, $z_1, \dots, z_n \in U$ are distinct closed points, then \[ F(U \setminus \{z_1, \dots, z_n\})/F(U) \wequi \bigoplus_{i=1}^n F(U_{z_i}^h \setminus z_i)/F(U_{z_i}^h). \]
If $n=1$, this follows by combining EA and EE.
Now let $n>1$, and assume the claim proved for $n-1$.
Combining IA and the case $n=1$, we have a short exact sequence \[ 0 \to F(U \setminus \{z_1, \dots, z_{n-1}\})/F(U) \to F(U \setminus \{z_1, \dots, z_n\})/F(U) \to F(U_{z_n}^h \setminus z_n)/F(U_{z_n}^h) \to 0. \]
By induction, the first term in the sequence isomorphic to $\bigoplus_{i=1}^{n-1} F(U_{z_i}^h \setminus z_n)/F(U_{z_i}^h)$, and we can thus split the sequence.
This proves the claim.

Consider the following sequence of sheaves on $X_\Nis$ \[ 0 \to a_\Nis F \to \bigoplus_{\eta \in U^{(0)}} F(\eta) \to \bigoplus_{z \in U^{(1)}} F(U_z^h \setminus z)/F(U_z^h) \to 0, \] where $U \to X$ is an arbitrary affine étale scheme.
Observe that the second and third terms are skyscraper sheaves, and so are acyclic (see e.g. \cite[proof of Lemma 5.42]{A1-alg-top}).
We argue that this sequence is exact.
For this we need only consider the case where $U = \eta$ (a generic point of some étale $X$-scheme), and the case where $U$ is henselian local of dimension $1$, so in particular has only two points.
Both sequences are exact; the only non-trivial point is injectivity of $F(U) \to F(\eta)$ which is IL.

It follows that we may compute $H^i(U, F)$ using the above resolution; in particular $H^i=0$ for $i>1$.
Let $U \subset X$.
We first compute $H^0(U,F)$: it consists of those elements $a \in F(\eta)$ (where $\eta$ is the generic point of $U$) such that for every closed point $z \in U$, $a$ is in the image of $F(X_z^h) \to F(X_z^h \setminus z)$.
Let $a$ be such an element.
Then there exists $\emptyset \ne V \subset U$ and $a' \in F(V)$ such that $a = a'|_\eta$.
Let $z \in U \setminus V$ and put $V' = V \cup \{z\}$.
Note that $V' \subset U$ is open since its complement consists of finitely many (closed) points.
By $(*)$ with $n=1$ we have $F(V)/F(V') \wequi F((V'_z)^h \setminus z)/F((V'_z)^h)$.
The image of $a'$ in the right hand group vanishes by assumption, hence it vanishes in the left hand group.
In other words there exists $a'' \in F(V')$ extending $a'$.
Repeating this argument finitely many times we conclude that $F(U) \to H^0(U, F)$ is surjective.
The map is injective by IA, and hence an isomorphism.

It remains to prove that $H^1(U,F) = 0$.
In other words, given distinct closed points $z_1, \dots, z_n \in U$ we must prove that $F(\eta) \to \bigoplus_{i=1}^n F(U_{z_i}^h \setminus z_i)/F(U_{z_i}^h)$ is surjective.
This follows from $(*)$, since it identifies the right hand side with a quotient of $F(U \setminus \{z_1, \dots, z_n\})$.
\end{proof}

The following is essentially \cite[Theorem 23.12]{lecture-notes-mot-cohom}.
\begin{lemma} \label{lemm:principal-imm}
Let $X$ be essentially smooth and affine, $i: Z \hookrightarrow X$ a principal, essentially smooth closed subscheme, and $F$ satisfy EA, EE.
Write $j: U = X \setminus Z \to X$ for the complementary open immersion.
Suppose that we are given étale neighborhoods $(X, Z) \leftarrow (\Omega, Z) \to (\A^1_Z, Z)$.

There is a short exact sequence of Nisnevich sheaves on $X$ \[ a_\Nis F \to a_\Nis j_* j^* F \to a_\Nis i_* F_{-1}' \to 0. \]If $F$ satisfies IL, then $a_\Nis F \to a_\Nis j_* j^* F$ is injective.
\end{lemma}
\begin{proof}
We shall use without further comment the fact that $Z_\Nis$ has a conservative family of stalk functors of the form $F \mapsto F(U_x^h \times_X Z)$, where $U \to X$ is étale and $x \in U$.\NB{ref?}

Denote by $F_{(X,Z)}$ the sheaf $a_\Nis i^*(j_*j^*F/F)$ on $Z_\Nis$.
By adjunction we obtain a map $a_\Nis(j_*j^*F/F) \to i_* F_{(X,Z)}$.
Checking on stalks, we see that this is an equivalence.
Since $i_*$ commutes with $a_\Nis$\NB{ref?}, our task is to prove that $F_{(X,Z)} \wequi a_\Nis F'_{-1}$.
If $(X', Z) \to (X, Z)$ is an étale neighborhood, there is an induced map $F_{(X,Z)} \to F_{(X',Z')}$; again checking on stalks we see that this is an equivalence.
Using the étale neighborhoods provided by the hypothesis, we may thus assume that $X=\A^1_Z$.
For $U \to Z$ étale, $\A^1_U \to \A^1_Z$ is étale, and $F'_{-1}(U) = (j_*j^*F/F)(\A^1_U)$; this induces a map $F_{-1}' \to i^*j_*j^*F/F$.
We shall prove that this is an equivalence.

We check this on stalks.
Let $Z' \to Z$ be étale and $x \in Z'$.
We shall consider the stalk at $x$; to simplify notation replace $Z$ by $Z'$.
Consider the following commutative diagram
\begin{equation*}
\begin{CD}
(\A^1_Z)_x^h @>>> (\A^1_{Z_x^h})_x @>>> \A^1_{Z_x^h} @>>> \A^1_Z \\
@AAA                 @AAA                   @AAA            @AAA \\
B        @>>>   A                  @>>>  Z_x^h       @>>> Z.
\end{CD}
\end{equation*}
The right-most vertical map is the canonical inclusion, and all squares are defined to be cartesian.
Note that $A$ is local and a localization of $Z_x^h$ containing $x$; thus $A \wequi Z_x^h$.
Similarly $B$ is henselian local, pro-étale over $A$ and contains $x$, thus $B \wequi Z_x^h$ also.
Now we get \begin{gather*} [i^*(j_*j^* F/F)](Z_x^h) \wequi (j_*j^* F/F)((\A^1_Z)_x^h) \wequi F((\A^1_Z)_x^h \setminus Z_x^h)/F((\A^1_Z)_x^h) \stackrel{EE}{\wequi} F((\A^1_{Z_x^h})_x \setminus Z_x^h)/F((\A^1_{Z_x^h})_x) \\ \stackrel{EA}{\wequi} F(\A^1_{Z_x^h} \setminus Z_x^h)/F(\A^1_{Z_x^h}) \wequi F'_{-1}(Z_x^h). \end{gather*}

For the last part, it suffices to observe that if $X$ is henselian local with generic point $\eta$ and $U \subset X$ is non-empty, then \[ F(X) \wequi (a_\Nis F)(X) \to (a_\Nis F)(U) \to (a_\Nis F)(\eta) \wequi F(\eta) \] is injective by $IL$, and hence so is $(a_\Nis F)(X) \to (a_\Nis F)(U)$.
\end{proof}

\begin{remark} \label{rmk:principal-imm}
Let $X \in \Sm_k$, $Z \subset X$ a smooth, principal closed subscheme.
Then locally on $X$, étale neighborhoods as required in Lemma \ref{lemm:principal-imm} exist.
See e.g. \cite[\S5.9]{deglise-regular-base}. 
\end{remark}

\begin{lemma} \label{lemm:contr-inj}
Suppose that $F$ satisfies IA and $F_{-1}$ satisfies IL.
Then the canonical map $a_\Nis F_{-1} \to (a_\Nis F)_{-1}$ is an injection.
\end{lemma}
\begin{proof}
Using IL, it suffices to prove that $F_{-1}(K) \to (a_\Nis F)_{-1}(K)$ is injective.
Hence we need to prove that $F(\A^1_K \setminus 0) \to (a_\Nis F)(\A^1_K \setminus 0)$ is injective.
This follows from IA.
\end{proof}

\begin{proof}[Proof of Theorem \ref{thm:strict-A1-formal}.]
To begin with, note that if $F$ is $\A^1$-invariant then so is $F_{-n}$.
We shall use this freely in the sequel.

As a first step, we shall prove that if $F \in \scr C$ is $\A^1$-invariant then so is $a_\Nis F$.
Since $F$ is $\A^1$-invariant it satisfies IA, EA, IL and EE (by assumption on $\scr C$) and so Lemma \ref{lemm:strict-formal-field} applies.
Let $X \in \SmAff_k$.
We must prove that $H^0(X \times \A^1, F) \to H^0(X, F)$ is injective.
Consider the diagram
\begin{equation*}
\begin{CD}
H^0(X,F) @>>> \prod_{\eta \in X^{(0)}} H^0(\eta, F) \\
@AAA                      @AAA \\
H^0(\A^1_X, F) @>>> \prod_{\eta \in X^{(0)}} H^0(\A^1_\eta, F) \\
@VVV                               @VVV \\
\prod_{x \in U} H^0(U_x^h, F) @>>> \prod_{x \in U} H^0(\eta_x, F),
\end{CD}
\end{equation*}
where the product is over points of $(\A^1_X)_\Nis$ and $\eta_x$ is the generic point of $U_x^h$.
This lies over a point of $\A^1_\eta$, so the bottom right hand map is defined and the diagram commutes.
The bottom left hand map is injective (since $H^0(-,F)$ is Nisnevich-separated) and the bottom horizontal map is injective by IL; hence the middle horizontal map is injective.
Consequently the top left hand map is injective as soon as the top right hand map is.
This reduces the claim to the case $X=\eta$, which holds by Lemma \ref{lemm:strict-formal-field}.

Next we will prove by induction on $n$ that if $F \in \scr C$ is $\A^1$-invariant, then $H^n(\ph, F)$ is also $\A^1$-invariant.
The case $n=0$ has been dealt with.
In particular we may assume that $F$ is a sheaf.
Let $n>0$ and suppose that all smaller $n$ have been established.

\emph{Step (1).} If $j: U \to X$ is a principal open immersion of smooth $k$-schemes, with smooth closed complement $Z \hookrightarrow X$, then we claim 
\[
R^i j_* F = 0 \qquad \forall 0 < i < n.
\]
For this consider the presheaf $G=H^i(\ph, F)$, which we know is $\A^1$-invariant by induction.
The problem is local on $X$, so by Remark \ref{rmk:principal-imm} we may apply Lemma \ref{lemm:principal-imm} to $G$ and obtain an exact sequence (using that $G_{-1} \wequi G'_{-1}$ by Remark \ref{rmk:contraction}) \[ 0 \to a_\Nis G \to a_\Nis j_* j^* G \to i_* a_\Nis G_{-1} \to 0. \]
We have $a_\Nis G = 0$, $a_\Nis j_* j^* G = R^i j_* F$. By Lemma \ref{lemm:contr-inj} we have \[ a_\Nis G_{-1} \hookrightarrow (a_\Nis G)_{-1} = 0, \] which proves the claim.

\emph{Step (2).} If $X$ is an essentially smooth scheme, $U \subset X$ a principal open subscheme with essentially smooth closed complement, then we claim that $H^n(\A^1_X, j_*j^* F) \to H^n(\A^1_U, F)$ is injective.
To prove this, we may assume $X$ smooth.\todo{details?}
Consider the cofiber sequence $j_* j^* F \to Rj_* j^* F \to C$.
By step (1), $C$ has cohomology sheaves concentrated in degree $\ge n$.
Hence in the long exact sequence \[ H^{n-1}(\A^1_X, C) \to H^{n}(\A^1_X, j_* j^* F) \to H^n(\A^1_X, Rj_* j^* F) \] the first term vanishes, so the second map is injective.
The result follows since the last term identifies with $H^n(\A^1_U, F)$ by the previous step.

\emph{Step (3).} If $X$ is an essentially smooth scheme, $U \subset X$ a principal open subscheme with essentially smooth closed complement $Z$, $z \in Z$, then $H^n(\A^1_{X_z^h}, F) \to H^n(\A^1_{X_z^h}, j_*j^* F)$ is injective.
Again we may assume that $X$ is smooth.
The problem being local around $z$, via Remark \ref{rmk:principal-imm} we may assume given étale neighborhoods $(X, Z) \leftarrow (\Omega, Z) \to (\A^1_Z, Z)$.
Taking the product with $\A^1$, we may thus apply Lemma \ref{lemm:principal-imm} to $\A^1_Z \hookrightarrow \A^1_X$ and get an exact sequence $0 \to F \to j_*j^*F \to i_* F_{-1} \to 0$ on $(\A^1_X)_{\Nis}$.
Taking $H^1(\A^1_{X_z^h}, \ph)$ yields a long exact sequence, part of which reads \[ H^{n-1}(\A^1_{X_z^h}, j_*j^* F) \xrightarrow{a} H^{n-1}(\A^1_{Z_z^h}, F_{-1}) \to H^n(\A^1_{X_z^h}, F) \to H^n(\A^1_{X_z^h}, j_*j^* F). \]
It is thus enough to prove that $a$ is surjective.
If $n>1$ then $H^{n-1}(\A^1_{Z_z^h}, F_{-1}) \wequi H^{n-1}(Z_z^h, F_{-1}) = 0$, by induction on $n$.
It remains to prove that $H^0(\A^1_{X_z^h}, j_*j^* F) \to H^0(\A^1_{Z_z^h}, F_{-1})$ is surjective, i.e. that the map $F(\A^1_{X_z^h \setminus Z_z^h}) \to F_{-1}(\A^1_{Z_z^h})$ is surjective.
Since $F$ is $\A^1$-invariant, this is just $F(X_z^h \setminus Z_z^h) \to F_{-1}(Z_z^h)$, which is the evaluation of the surjective map of sheaves $j_*j^* F \to F_{-1}$ on $X_z^h$, and hence surjective.

\emph{Conclusion.}
Let $X$ be an essentially smooth scheme.
Write $f: X \to \A^1_X$ for the inclusion at $0$.
We seek to prove that $F \to Rf_*F$ induces an equivalence on $H^n$, and we already know this for $H^i$, $i<n$.
We shall prove this by induction on $d=\dim X$.
Let $C$ be the cofiber of $F \to Rf_* F$.
Then $Rf_* F \wequi F \oplus C$ (via $p: \A^1_X \to X$) and so we must prove that $C$ has cohomology concentrated in degrees $>n$; this may be checked stalkwise.
In other words we must prove that if $X$ is an essentially smooth, henselian local scheme of dimension $d$, then $H^n(\A^1_X, F) \to H^n(X, F)$ is an isomorphism; equivalently we may prove that it is injective.
If $d=0$, $X$ is the spectrum of a field, and we are reduced to Lemma \ref{lemm:strict-formal-field}.
Thus $d > 0$ and we can find a principal open $U \subset X$ with essentially smooth closed complement $Z$
(here we use that $k$ is perfect).
Consider the commutative diagram
\begin{equation*}
\begin{CD}
H^n(\A^1_X, F) @>{(2)}>> H^n(\A^1_X, j_*j^* F) @>{(1)}>> H^n(\A^1_U, F) \\
@VVV                  @VVV                @VVV    \\
H^n(X, F) @>>> H^n(X, j_*j^* F) @>>> H^n(U, F).
\end{CD}
\end{equation*}
The maps in the top composite are injective, by the steps indicated above them.
The right hand vertical map is injective by induction on $d$.
It follows that the left hand vertical map is injective, as desired.
\end{proof}

\begin{remark} \label{rmk:H0-imperf}
The proof shows that the perfectness assumption on $k$ is only needed to ensure $\A^1$-invariance of $H^i(\ph, F)$ for $i>0$.
\end{remark}

\subsection{Framed pretheories}
\newcommand{\tw}{\mathrm{tw}}
\def\div{\mathrm{div}}
As in the last section, we consider an abelian presheaf $F \in \PSh_\Sigma(\SmAff_k, \Ab)$, extended by continuity to essentially smooth affine schemes. The next definition is the framed analog of the notion of pretheories introduced by Voevodsky, see  \cite{voevodsky2000cohomological}.

\begin{definition} \label{def:framed-preth}
By a structure of \emph{framed pretheory} on $F$ we mean the following data: for every $X \in \SmAff_k$, $C \to X$ a smooth relative curve, $\mu: \Omega^1_{C/X} \wequi \scr O_C$, $f \in \scr O(C)$ and decomposition $Z(f) = Z \amalg Z'$ with $Z$ finite over $X$, we are given \[ \tr(f)^\mu_Z: F(C) \to F(X). \]
The transfers must satisfy the following properties:
\begin{enumerate}
\item If $Z = Z_1 \amalg Z_2$, then $\tr(f)^\mu_Z = \tr(f)^\mu_{Z_1} + \tr(f)^\mu_{Z_2}$.
\item If $p: (C', Z') \to (C, Z)$ is an étale neighborhood, then $\tr(f)^\mu_Z = \tr(f \circ p)^{p^* \mu}_{Z'} \circ p^*$.
\item Given $\alpha: X' \to X$ let $\alpha': X' \times_X C \to C$ be the induced map.
  Then 
  \[\alpha^* \circ \tr(f)^\mu_Z = \tr(f \circ \alpha')^{\alpha'^* \mu}_{\alpha'^{-1}(Z)} \circ \alpha'^*.
  \]
\item Suppose that $Z \to X$ is an isomorphism with inverse $i$.
  Then \[ \tw(f)^\mu_Z := \tr(f)^\mu_Z \circ (C \to X)^*: F(X) \to F(X) \] is an isomorphism and $\tr(f)^\mu_Z = \tw(f)^\mu_Z \circ i^*$.
\item Fix a section $i: X \to C$ with image $Z$ and a trivialization $\mu$.
  Assume that $X$ is semilocal.
  Then, there exists $\lambda \in H^0(Z, C_{Z/C})$ such that for any $f \in I_Z(C)$ with $df = \lambda$\todo{notation?} we have $\tw(f)^\mu_Z = \id$.
\end{enumerate}
\end{definition}
Note that condition (3) implies that the transfers on $F$ extend to essentially smooth schemes, so (5) makes sense.

\begin{example} \label{ex:framed}
Let $F \in \PSh_\Sigma(\Cor^\fr_k, \Ab)$.
Then $F$ admits a structure of framed pretheory as follows.
Given $(C\to X,\mu,f,Z)$, then $Z \to X$ is syntomic by \cite[Proposition 2.1.16]{EHKSY1}. Furthermore we define the following $K$-theoretic trivialization of the cotangent complex $L_{Z/X}$ \[ \tau: L_{Z/X} \wequi L_{Z/C} + L_{C/X}|_Z \stackrel{f,\mu}{\wequi} -\scr O + \scr O \wequi 0 \in K(Z)\footnote{Note that $L_{Z/C} \simeq I/I^2[1]$ where $I$ is the ideal defining $Z$. By hypothesis, $I/I^2$ is an invertible $\scr O_Z$-module and the class of $f$ in $I/I^2$ provides a trivialization of this line bundle}. \]
Consequently we obtain a framed correspondence $X \xleftarrow{\tau} Z \to C$, pullback along which defines $\tr(f)^\mu_Z$.
All axioms are easily verified.
(For the last axiom, one may argue as follows.
Since $Z$ is semilocal, $C_{Z/C}$ admits a non-vanishing section $\lambda'$.
Together $\lambda', \mu$ determine a trivialization of $L_{Z/X} \wequi L_{X/X} = 0$, whence a class in $K_1(Z)$.
Since $Z$ is semilocal, $K_1(Z) \wequi \scr O^\times(Z)$ \cite[Lemma III.1.4]{weibel-k-book}, and hence replacing $\lambda'$ by $\lambda := u\lambda'$ for well-chosen $u$ ensures that the trivialization is the canonical one and thus $\tw(f)$ is the identity morphism.)
\end{example}

Suppose given $(C \to X, f, \mu, Z)$ as in Definition \ref{def:framed-preth}, a map $g: C \to Y \in \SmAff_k$, $U \subset X$ and $U' \subset Y$ open.
Assume that $g^{-1}(Y \setminus U') \cap Z$ lies over $X \setminus U$.
Write $C_U \subset C$, $Z_U \subset Z$ for the canonical open subschemes.
Then $C_U \cap g^{-1}(U') \to C_U$ is an étale (in fact open) neighborhood of $Z_U$.
We may form the following diagram
\begin{equation*}
\begin{tikzcd}
F(Y) \ar[rr,"g^*"] \ar[d] && F(C) \ar[d] \ar[r, "\tr(f)^\mu_Z"] & F(X) \ar[d] \\
F(U') \ar[r, "g^*"] & F(C_U \cap g^{-1}(U')) \ar[rr, "\tr", bend right=20] & F(C_U) \ar[l] \ar[r, "\tr"] & F(U).
\end{tikzcd}
\end{equation*}
The maps labelled $\tr$ are the evident transfers, and the unlabelled maps are pullbacks along evident inclusions.
The diagram commutes by properties (2) and (3).
\begin{construction} \label{cons:pair-pullback}
Taking vertical cokernels in the outer rectangle of the above diagram, we obtain a map \[ F(U')/F(Y) \to F(U)/F(X). \]
\end{construction}

\begin{definition}
Let $X, Y$ be essentially smooth over $k$, $X \leftarrow C \xrightarrow{g} Y$ a span.
We call data \[ \Phi = (X \leftarrow C \rightarrow Y,f,\mu,Z): X \rightsquigarrow Y \] such that $(C \to X,f,\mu,Z)$ satisfies the assumptions of Definition \ref{def:framed-preth}, a \emph{curve correspondence} from $X$ to $Y$ and put \[ \Phi^* = \tr(f)^\mu_Z \circ g^*: F(Y) \to F(X). \]
If we are further given $U \subset X, U' \subset Y$ such that $g^{-1}(Y \setminus U') \cap Z$ lies over $X \setminus U$, we call the data a \emph{curve correspondences of pairs}, denote it by \[ \Phi = (X \leftarrow C \rightarrow Y,f,\mu,Z): (X,U) \rightsquigarrow (Y,U') \] and write \[ \Phi^*: F(U')/F(Y) \to F(U)/F(X) \] for the map of Construction \ref{cons:pair-pullback}.
\end{definition}

\begin{lemma} \label{lemm:almost-null}
Let $\Phi = (X \leftarrow C \xrightarrow{g} Y,f,\mu,Z): (X,U) \rightsquigarrow (Y,U')$ be a curve correspondence of pairs.
Suppose that $Z \subset g^{-1}(U')$.
Then $\Phi^* = 0$.
\end{lemma}
\begin{proof}
By axiom (2) we can replace $C$ by $g^{-1}(U')$.
Since now $\Phi^*$ factors through $F(U')/F(U') = 0$, the result follows.
\end{proof}

\subsection{Injectivity on the relative affine line}
In this section we establish $IA$ for $\A^1$-invariant framed pretheories.

\begin{lemma} \label{lemm:IA}
Let $U \in \SmAff_k$, $V_1 \subset V_2 \subset \A^1_U$ affine and open.
Assume that $\A^1_U \setminus V_2$ and $V_2 \setminus V_1$ are finite over $U$.
Then there exist curve correspondences $\Phi, \Phi^-: V_2 \rightsquigarrow V_1$ and $\Theta: V_2 \times \A^1 \rightsquigarrow V_2$ such that in any framed pretheory, $i_0^*\Theta^* = \Phi^* i^*$ and $i_1^* \Theta^* - (\Phi^-)^*i^*$ is invertible.
Here $i$ denotes the inclusion $V_1 \to V_2$ and $i_s: V_2 \to V_2 \times \A^1$ is the inclusion at $s$.
\end{lemma}
\begin{proof}
We begin by constructing certain functions $f, g \in k[\A^1_U \times_U V_2]$.
We shall denote the first coordinate by $y$ and the second by $x$.
We will arrange that $f,g$ are, respectively, monic in $y$ of degrees $n$ and $n-1$ (for some $n$ sufficiently large).
Moreover, we shall ensure that
\begin{gather*}
f|_{(\A^1_U \setminus V_1) \times_{U} V_2} = 1 \\
g|_{(\A^1_U \setminus V_2) \times_U V_2} = (y-x)^{-1} \\
g|_{(V_2 \setminus V_1) \times_U V_2} = 1 \quad g|_{Z(y-x)} = 1. \\
\end{gather*}
To do this, note that each of the subschemes we are restricting to is finite over $V_2$, and apply Lemma \ref{lemm:monic-res} below.

Let $h \in k[\A^1 \times V_2 \times \A^1]$ be given by \[ h = (1-t)f + t(y-x)g; \] here $t$ denotes the third coordinate.
Note that $h$ is monic in $y$.
Define
\begin{gather*}
  \Phi = (V_2 \leftarrow V_1 \times_U V_2 \to V_1, f, dy, Z(f)): V_2 \rightsquigarrow V_1 \\
  \Theta = (V_2 \times \A^1 \xleftarrow{pr_2} V_2 \times_U V_2 \times \A^1 \xrightarrow{pr_1} V_2, h, dy, Z(h)): V_2 \times \A^1 \rightsquigarrow V_2.
\end{gather*}
Note that here by $f$ we implicitly denote its restriction to $V_1 \times_U V_2$, and similarly for $h$.
Since $h$ is monic in $y$, $Z(h) \subset \A^1 \times V_2 \times \A^1$ is finite over $V_2 \times \A^1$.
By construction, $h$ is constantly equal to $1$ on $(\A^1_U \setminus V_2) \times_U V_2 \times \A^1$.
Thus $Z(h)$ is completely contained in $V_2 \times_U V_2 \times \A^1$, and so $\Theta$ is well-defined.
A similar argument applies to $\Phi$.

Since $h|_{t = 0} = f$ we find (using Definition \ref{def:framed-preth}(3)) that \[ i_0^* \Theta^* = \Phi^* i^*. \]
Since $h|_{t = 1} = (y-x)g$ has vanishing locus splitting into two disjoint pieces, we find (using Definition \ref{def:framed-preth}(1)) that \[ i_1^* \Theta^* = (V_2 \xleftarrow{pr_2} V_2 \times_U V_2 \xrightarrow{pr_1} V_2, (y-x)g, dy, Z(y-x))^* + (V_2 \times_U V_2, (y-x)g, dy, Z(g))^*. \]
The first term is invertible by Definition \ref{def:framed-preth}(4).
Note that $Z(g) \subset V_1 \times_U V_2$.
Thus we can define \[ \Phi^- = (V_2 \leftarrow V_1 \times_U V_2 \to V_1, (y-x)g, dy, Z(g)): V_2 \rightsquigarrow V_1, \] concluding the proof.
\NB{$Z((y-x)g|_{V_1 \times_U V_2}$ would not be finite!}
\end{proof}

\begin{lemma} \label{lemm:monic-res}
Let $U$ be an affine scheme and $Z \subset \A^1_U$ a closed subscheme which is finite over $U$.
Let $\bar f \in \scr O(Z)$.
Then for $n$ sufficiently large there exists a monic $f \in \scr O(\A^1_U)$ of degree $n$ with $f|_Z = \bar f$.
\end{lemma}
\begin{proof}
Let $U = \Spec(A)$, $Z = \Spec(A[T]/I)$.
Since $Z$ is finite, there exist $g_1, \dots, g_r \in A[T]$ whose images generate $A[T]/I$ as an $A$-module.
Let $n$ be larger than the maximum of the degrees of the $g_i$.
We claim that $f$ as desired can be found for such $n$.
Indeed note that \emph{any} $\bar h \in A[T]/I$ admits a lift $h \in A[T]$ of degree $<n$; in fact we can choose the lift to be an $A$-linear combination of the $g_i$.
Now let $f_1$ be an arbitrary lift of $\bar f - T^n$ of degree $<n$, and put $f= T^n + f_1$.
\end{proof}

\begin{theorem} \label{thm:IA}
Let $U$ be essentially smooth over $k$, $V_1 \subset V_2 \subset \A^1_U$ affine and open.
Assume that $\A^1_U \setminus V_2$ and $V_2 \setminus V_1$ are finite over $U$.
Let $F$ be an $\A^1$-invariant framed pretheory.

Then $F(V_2) \to F(V_1)$ is injective.

In particular $F$ satisfies IA.
\end{theorem}
\begin{proof}
All our open immersions are affine, hence quasi-compact \cite[Tag 01K4]{stacks-project} and so of finite presentation \cite[Tag 01TU]{stacks-project}.
It follows that when writing $U = \lim_i U_i$ as a cofiltered limit of smooth affine schemes, we may assume given $V_1' \subset V_2' \subset \A^1_{U_0}$ affine open with base change $V_i$, such that $\A^1_{U_0} \setminus V_2'$ and $V_2' \setminus V_1'$ are finite over $U_0$ \cite[Tags 01ZM, 01ZO and 0EUU]{stacks-project}.
By continuity of $F$, we may thus assume $U \in \SmAff_k$.
Now let $x \in F(V_2)$ with $i^*(x) = 0$.
Using $\A^1$-invariance we find that, in the notation of Lemma \ref{lemm:IA} \[ 0 = \Phi^*i^*(x) - (\Phi^-)^*i^*(x) = i_0^*\Theta^*(x) - (\Phi^-)^*i^*(x) = i_1^*\Theta^*(x) - (\Phi^-)^*i^*(x). \]
But $i_1^*\Theta^* - (\Phi^-)^* i^*$ is invertible, so $x=0$.
\end{proof}

\subsection{Geometric preliminaries}
The proofs of the other axioms are similar to the one for IA, but significantly more elaborate.
We collect here some results from algebraic geometry that we shall use.

\subsubsection{Serre's theorem} \label{subsec:serre}
Let $A$ be a Noetherian ring and $X \to \Spec(A)$ a projective $A$-scheme with ample line bundle $\scr O(1)$.
Then for any coherent sheaf $\scr F$ on $A$, $i > 0$ and $n$ sufficiently large, $H^i(X, \scr F(n)) = 0$ \cite[Tag 0B5T(4)]{stacks-project}.
We shall often use the following immediate consequence: if $\scr F \to \scr G$ is a surjection of coherent sheaves, then for $n$ sufficiently large, $H^0(X, \scr F(n)) \to H^0(X, \scr G(n))$ is surjective. (Indeed this follows from vanishing of $H^1(X, \ker(\scr F \to \scr G)(n))$.)
We mainly use this as follows: if $Z \subset X$ is closed, then $H^0(X, \scr F(n)) \to H^0(Z, \scr F(n))$ is surjective.
This is deduced by taking $\scr G = i_*i^* \scr F$, where $i: Z \to X$ is the closed immersion.

\subsubsection{Semilocal schemes} \label{subsec:semilocal}
We call a scheme \emph{semilocal} if it has only finitely many closed points, \emph{and is affine}.
Note that if $X$ is an affine scheme (or more generally AF-scheme\footnote{This means that every finite set of points is contained in an open affine; for example a scheme quasi-projective over an affine base.}) and $x_1, \dots, x_n \in X$, then \[ X_{x_1, \dots, x_n} := \lim_{U \supset \{x_1, \dots, x_n\}} U \] is a semilocal scheme, where the limit is over all open neighborhoods of the $x_i$.
Indeed every such neighborhood is quasi-affine, and hence contains a smaller \emph{affine} neighborhood of the finitely many points \cite[Corollaire 4.5.4]{EGAII}.
(But note that for example if $X$ is the affine line with the origin doubled and $x_1, x_2 \in X$ are the origins, then $X_{x_1,x_2}$ is not separated and hence not semilocal.\NB{does finitely many closed points + qcs imply semilocal?})

We shall frequently use the following properties of semilocal schemes.
\begin{enumerate}
\item If $X$ is semilocal and $X' \to X$ is finite then $X'$ is semilocal.
\item If $X$ is semilocal and $L$ is a line bundle (or more generally vector bundle of constant rank) on $X$, then $L$ is trivial \cite[Lemma 1.4.4]{bruns1998cohen}.\footnote{Here is a proof. Let $X_0 \subset X$ denote the closed subscheme which is the disjoint union of the closed points of $X$.
  Then $L|_{X_0}$ admits a non-vanishing section, which can ($X$ being affine) be lifted to a section of $L$ on $X$.
  Its vanishing locus avoids $X_0$ and is closed, hence empty.}
\item If $X$ is semilocal and $Y \to X$ is a closed immersion, then $\scr O^\times(X) \to \scr O^\times(Y)$ is surjective.\footnote{Here is a proof. Let $X_0' \subset X$ denote the closed subscheme which is the disjoint union of the closed points of $X \setminus Y$.
  Then $Y \amalg X_0' \to X$ is a closed immersion.
  Now given $a \in \scr O^\times(Y)$, as before we can lift the non-vanishing section $(a, 1) \in \scr O(Y \amalg X_0')$ to a section $\tilde a \in \scr O(X)$, which is non-vanishing.}
\end{enumerate}

\subsubsection{Some general position arguments}
The following is essentially \cite[Lemma 4.1]{druzhinin2017strictly}.
\begin{lemma} \label{lemm:find-Z0}
Let $U$ be a local Noetherian scheme with infinite residue field.
Let $\overline{C'} \to \overline{C}$ be a finite morphism of projective curves over $U$, $\scr Z', D' \subset \overline{C'}$ closed subschemes finite over $U$ with $\scr Z' \cap D' = \emptyset$, $\Delta_Z' \subset \scr Z'$ a principal closed subscheme, $\overline{C'} \setminus D'$ smooth affine over $U$.

Assume that the composite $\Delta_Z' \to \overline{C'} \to \overline{C}$ is a closed immersion.
Then for $n$ sufficiently large there exists a section $\xi \in H^0(\overline{C'}, \scr O(n))$ such that $Z(\xi) \cap \scr Z' = \Delta_Z'$, $Z(\xi) \cap D' = \emptyset$, and $Z(\xi) \to \overline{C'} \to \overline{C}$ is a closed immersion.
\end{lemma}
\begin{proof}
Let us begin with the following preparatory remarks.
Let $X \subset \overline{C'}$ be a closed subscheme which is finite over $\overline{C}$.
Let $T \subset \overline{C}$ be the set of points $t$ such that the geometric fiber $X_{\bar t} \to \bar t$ is not a closed immersion.
Then $T$ is the support of the coherent sheaf $\cok(\scr O_{\overline C} \to \pi_* \scr O_X)$, and hence closed in $\overline{C}$.
Since a proper morphism is a closed immersion if and only if it is unramified and radicial \cite[Tags 01S2 and 04XV]{stacks-project}, we see that $X \to \overline{C}$ is a closed immersion if and only if $T=\emptyset$.
In particular $X \to \overline{C}$ is a closed immersion if and only if its restriction to the closed fiber over $U$ is.

By assumption $\Delta_Z' \subset \scr Z'$ is principal, say cut out by a section $t \in H^0(\scr Z, \scr O)$.
Since $\scr Z'$ is semilocal, $\scr O(1)|_{\scr Z}$ admits a non-vanishing section $d$.
Let $\xi \in H^0(\overline{C'}, \scr O(n))$ be a section such that $\xi|_{\scr Z'} = td^n$.
Let $x \in U$ be the closed point.
For any scheme $S \to U$ denote by $S_x$ the fiber over $x$.
Assume that $Z(\xi) \cap D'_x = \emptyset$ and $Z(\xi)_x \to \overline{C}_x$ is a closed immersion.
Then $Z(\xi) \cap D' = \emptyset$ (being proper over $U$ with empty closed fiber), so $Z(\xi) \to U$ is finite (being proper and affine \cite[Tag 01WN]{stacks-project}).
Hence by the preparatory remarks, $Z(\xi) \to \overline{C}$ is a closed immersion.
That is, such a $\xi$ satisfies the required properties.
Let $M = \scr O_{\overline{C'}_x} \times_{\scr O_{\scr Z'_x}} \scr O_{\scr Z'}$.
Then $\scr O_{\overline{C'}} \to M$ is surjective by \cite[Tag 0C4J]{stacks-project} and so is surjective on $H^0$ after twisting up sufficiently.
Thus it suffices to construct $\xi$ on the closed fiber (satisfying the additional condition that $\xi|_{\scr Z'_x} = td^n$, so that $Z(\xi) \cap \scr Z' = \Delta_Z'$).

We may thus assume that $U$ is the spectrum of an infinite field $k$.
For each point $x \in D'$, pick a trivialization $s_x$ of $\scr O(1)|_x$.
Let $n>0$ and $\Gamma \subset H^0(\overline{C'}, \scr O(n))$ consist of those sections $s$ such that $s|_x = s_x^n$ for all $x \in D'$, and $s|_{\scr Z'} = td^n$.
We must show that there exists (for $n$ sufficiently large) $s \in \Gamma$ such that $Z(s) \to \overline{C}$ is a closed immersion.
Let $T \subset \overline{C} \times \Gamma$ denote the subset of pairs $(s, c)$ such that $Z(s) \to \overline{C}$ is not a closed immersion over the geometric point $\bar c$.
We claim that $\dim T < \dim \Gamma$.
This implies that the complement of the closure of the image of $T \to \Gamma$ is non-empty, and hence has a rational point ($k$ being infinite and $\Gamma$ an affine space).
The preparatory remarks show that any such rational point corresponds to a closed immersion, as desired.

To prove the claim, we may base change to an algebraic closure of $k$, and hence assume $k$ algebraically closed.
Note that for $n$ sufficiently large, for any $x_1, x_2 \in \overline{C'}$ the map \[ H^0(\overline{C'}, \scr O(n)) \to H^0(Z_{x_1,x_2} \cup D' \cup \scr Z', \scr O(n)) \] is surjective, where $Z_{x_1,x_2} = Z(I(x_1)I(x_2))$.
(Indeed there is a closed subscheme of $\P(H^0(\overline{C'}, \scr O(n))) \times \overline{C'} \times \overline{C'}$ witnessing the failure of this condition, so the set of points $(x_1,x_2)$ satisfying the condition is open, but for every $(x_1,x_2)$ and $n$ sufficiently large the condition holds, so we conclude by quasi-compactness of $\overline{C'} \times \overline{C'}$.)
Now let $x_1, x_2 \in \overline{C'} \setminus (D' \cup \scr Z')$.
Then $H^0(Z_{x_1,x_2} \cup D' \cup \scr Z', \scr O(n)) \twoheadrightarrow H^0(Z_{x_1,x_2}, \scr O(n)) \wequi k^2$.
Now let $c \in \overline{C}$ and $s \in \Gamma$.
Then $Z(s) \to \overline{C}$ can only fail to be a closed immersion over $c$ if either there exist $x_1 \ne x_2 \in \overline{C'_c}$ with $s(x_1) = 0 = s(x_2)$, or there exists $x \in \overline{C'_c}$ such that $s$ vanishes to order $\ge 2$ at $x$.\NB{Here we use the smoothness assumption. Vanishing to order $2$ does not exclude closed immersion either...}
By the above remark, either condition is of codimension $2$ on $\Gamma$, provided $\scr Z'_c = \emptyset$.
For the finitely many other points $c$, the only condition is that $Z(s)_c$ contains other points, which is of codimension $1$ on $\Gamma$ by similar arguments.
It follows that all but finitely many fibers of $T \to \overline{C}$ have dimension $\le \dim \Gamma - 2$, and the remaining ones have dimension $\le \dim \Gamma - 1$.
Since $\dim \overline{C} = 1$, this concludes the proof.\NB{ref for dimension arguments?}
\end{proof}

In the rest of this section we will establish a moving lemma.
The core argument uses the method of general projections, which we encapsulate in the following.
\begin{theorem} \label{thm:projn}
Let $k$ be a field, $X \subset \A^N_k$ a closed subscheme of dimension $d$, $Z \subset \A^N$ of dimension $\le d-1$, $S \subset \A^N$ a finite set of closed point (i.e. a subscheme of dimension $0$).
Then for a general linear projection $\pi: \A^N \to \A^d$, the following hold:
\begin{enumerate}
\item $\pi|_X: X \to \A^d$ is finite.
\item If $X$ is smooth, then $\pi|_X$ is étale at all points of $S \cap X$.
\item $\pi^{-1}(\pi(S)) \cap Z \subset S$.
\end{enumerate}
\end{theorem}
\begin{proof}
This is proved for example in \cite[\S3.2]{kai2015moving}.
Specifically (1) is proved just before the beginning of \S3.2.1.
(3) is proved in the case $S = \{s\}$, $s \not\in Z$ in \S3.2.1.
The same proof works for $s \in Z$ (our statement is slightly different than the one in the reference, to allow this situation).
The case of general $S$ follows.
(2) is proved in \S3.2.2.
\end{proof}
\begin{remark}
Note that a dense open subset of affine space over an \emph{infinite} field contains a rational point.\NB{ref -- also same claim used previously}
It follows that in the case of an infinite field, there is an actual linear projection $\pi: \A^N_k \to \A^d_k$ satisfying all the properties.
\end{remark}

The following is our moving lemma.
It is essentially the same as \cite[Lemma 3.7]{druzhinin2017strictly}.
\begin{proposition} \label{prop:M-redux}
Let $k$ be an infinite field, $X \in \SmAff_k$, $Z \subset X$ a nowhere dense closed subscheme and $\pi: X' \to X \in \SmAff_k$ an étale neighborhood of $Z$.
Write $Z' \subset X'$ for the lift of $Z$.
Let $T' \subset Z'$ be a finite set of closed points and put $T = \pi(T') \subset Z$.
Let $U$ (resp. $U'$) be the semilocalization of $X$ in $T$ (resp. $X'$ in $T'$).
There exist commutative diagrams of essentially smooth, affine $k$-schemes
\begin{equation*}
\begin{CD}
U' @>s'>> C' @>v'>> X' \\
@V{\pi}VV @V{\varpi}VV @V{\pi}VV \\
U @>s>> C @>v>> X,
\end{CD}
\end{equation*}
and
\begin{equation*}
\begin{CD}
C' @>{j'}>> \overline{C'} \\
@V{\varpi}VV @V{\overline{\varpi}}VV \\
C  @>{j}>> \overline C @>>> S,
\end{CD}
\end{equation*}
such that the following hold:
\begin{enumerate}
\item The composites $U \to C \to X$ and $U' \to C' \to X'$ are the canonical inclusions.
\item $j, j'$ are open immersions, $\overline{C'}, \overline{C}$ are projective curves over $S$, $C, C'$ are smooth affine over $S$.
\item $\overline \varpi: \overline{C'} \to \overline{C}$ is finite and $\varpi: C' \to C$ is étale.
\item $\scr Z := v^{-1}(Z)$ and $\scr Z' := v'^{-1}(Z')$ are finite over $S$, and in fact $\scr Z' \xrightarrow{\wequi} \scr Z$.
\item $D' := \overline{C'} \setminus \overline{C}$ and $D := \overline{C'} \setminus \overline{C}$ are finite over $S$.
  Moreover $\overline{\varpi}(D') \supset D$.
\item We have $D = Z(d)$ for some section $d$ of an ample line bundle $\scr O(1)$ on $\overline C$.
\item $\Omega^1_{C/S}$ is trivial (and hence so is $\Omega^1_{C'/S}$).
\end{enumerate}
\end{proposition}

\begin{remark} \label{rmk:M-1}
We can base change $\overline{C}$, $\overline{C'}$ and so on along $U \to C \to S$.
Utilizing also the diagonal maps $U \to U \times_S U$ and $U' \to U' \times_S U$, we find that in Proposition \ref{prop:M-redux} we may can set $S=U$, which is the case of interest. In this case, we have the following extra properties:
\begin{enumerate}
\item The map $U \xrightarrow{s} C$ is a section of the separated morphism $C \to U$, whence its image is a closed subscheme (in fact an effective Cartier divisor) $\Delta \subset C$, which isomorphic to $U$.
\item The map $U \wequi \Delta \subset C \xrightarrow{v} X$ is the canonical inclusion.
\item We have $\Delta \cap \scr Z \wequi Z_T$ via the projection to $U$.
\item The composite $Z'_{T'} \hookrightarrow U' \xrightarrow{s'} C' \to U$ is a closed immersion, whence ($C' \to U$ being separated) we obtain a closed subscheme $\Delta_Z' \subset C'$ mapping isomorphically to $Z'_{T'} \wequi Z_T \subset U$ under the projection.
  Note that $\Delta_Z' \subset \scr Z'$.
\end{enumerate}
\end{remark}

\begin{remark} \label{rmk:M-2}
In the situation of Remark \ref{rmk:M-1}, let $C'' = C' \times_U U'$ and $\overline{C''} = \overline{C'} \times_U U'$.
We obtain the following commutative diagram
\begin{equation*}
\begin{CD}
C'' @>>> C' @>>> X' \\
@VVV  @VVV          \\
\overline{C''} @>>> \overline{C'} \\
@VVV       @VVV \\
U' @>>> U.
\end{CD}
\end{equation*}
Denote the composite $C'' \to C' \to X'$ by $v''$ and put $\scr Z'' = v''^{-1}(Z')$.
Let $D''$ be the preimage of $D'$ in $\overline{C''}$.
Note the following:
\begin{enumerate}
\item $\scr Z''$, $D''$ are finite over $U'$.
\item The pullback of $\scr O(1)$ to $\overline{C''}$ exhibits $C''$ as a projective curve over $U'$.
\item The map $s:U' \to C''$ induces a closed immersion $U' \to C''$; denote its image by $\Delta''$.
  Then $\Delta'' \cap \scr Z''$ maps isomorphically to $Z'_{T'}$ (and to $\Delta_Z'$).
\item The composite $U' \wequi \Delta'' \to X'$ is the canonical inclusion.
\end{enumerate}
\end{remark}

\begin{proof}[Proof of Proposition \ref{prop:M-redux}]
Shrinking $X, X'$ if necessary, and arguing on connected components, we may assume that $X', X$ are pure of dimension $d$ and $\pi^{-1}(Z) = Z'$.
Using Zariski's main theorem \cite[Tag 05K0]{stacks-project}, we obtain a dense open immersion $X' \hookrightarrow \overline{X'}$ over $X$ with $\overline{X'} \to X$ finite.
Choose an embedding $X \hookrightarrow{\A^N}$.
Using general projections (Theorem \ref{thm:projn}) we find a linear map $p_1: \A^N \to \A^d$ such that $X \to \A^d$ is finite, $X \to \A^d$ is étale at $T$, and $p_1^{-1}(p_1(T)) \cap Z \subset T$.
Let $X_0 \subset X$ be an affine open neighborhood of $T$ such that $X_0 \to \A^d$ is étale\footnote{Recall that any open neighborhood of $T$ contains an affine open neighborhood of $T$, as explained in \S\ref{subsec:semilocal}.} and put $X'_0 = \pi^{-1}(X_0)$, $Z_0 = Z \cap X_0$.
Using general projections again, we find a linear map $p_2: \A^d \to \A^{d-1}$ such that $p_1(Z) \to \A^{d-1}$ is finite, $p_1(X \setminus X_0) \to \A^d$ is finite, and $p_2^{-1}(p_2(p_1(T))) \cap p_1(Z \setminus Z_0) = \emptyset$ (the latter is possible since $p_1(Z \setminus Z_0)$ does not contain $p_1(T)$, by construction).
Let $S$ be the semilocalization of $\A^{d-1}$ at the closed subscheme $p_2(p_1(T))$. At this point we have the following diagram
\begin{equation*}
\begin{tikzcd}
X_0' \ar{d} \ar[r, hookrightarrow] & X' \ar{r} \ar{d}{\pi}& \overline{X}' \ar{dl} \\
X_0 \ar[r, hookrightarrow] \ar{d} & X \ar{dl} &\\
\A^d \ar[d], & & \\
\A^{d-1}
\end{tikzcd}
\end{equation*}
Base changing the above diagram along $S \to \A^{d-1}$, we obtain the schemes $C_0', C_1', C_2', C_0, C_1, \A^1_S$ of the following commutative diagram
\begin{equation*}
\begin{tikzcd}
C_0' \ar[r, hookrightarrow] \ar[d] & C_1' \ar[r, hookrightarrow] \ar[d] & C_2' \ar[r, hookrightarrow] \ar[dl] & \overline{C'} \ar[dl] \\
C_0 \ar[r, hookrightarrow] \ar[d] & C_1 \ar[r, hookrightarrow] \ar[dl] & \overline{C} \ar[dl] \\
\A^1_S \ar[d] \ar[r, hookrightarrow] & \P^1_S \ar[dl] \\
S
\end{tikzcd}
\end{equation*}
By construction $C_1 \to \A^1_S$ and $C_2' \to C_1$ are finite.
We obtain $\overline{C}$ by compactifying $C_1 \to \A^1_S \to \P^1_S$ and $\overline{C'}$ by compactifying $C_2' \to C_1 \to \overline{C}$ (using Zariski's main Theorem again).
Write $v_1: C_1 \to X$ for the canonical map.
By construction $\scr Z := v_1^{-1}(Z)$ is finite over $S$.
Moreover $v_1^{-1}(Z \setminus Z_0)$ is finite over $S$ but its image misses the closed points; hence it must be empty.
In other words $v_1^{-1}(Z) \subset C_0$.
Let $v_1': C_1' \to X'$ be the canonical map.
Then $\scr Z' := (v_1')^{-1}(Z') \to \scr Z$ is an isomorphism as needed.
Since the square part of the diagram is cartesian, we find that also $\scr Z' \subset C_0'$.
We shall find at the end a section $d \in H^0(\overline{C}, \scr O(n))$ such that $D := Z(d)$ is finite over $U$, $\overline{C} \setminus C_0 \subset D$ and $D \cap \scr Z = \emptyset$.
We put $C = \overline{C} \setminus D$ and let $C'$ be the preimage of $C$ in $\overline{C'}$.
Since $C_0$ is affine and $C \subset C_0$ is a principal open subset (note $\scr O(1)$ is trivial over $\A^1$ and hence over $C_0$), $C_0$ is affine.
The same argument applies to $C'$.
Since $C' \to C \to \A^1_S$ are étale, the canonical modules vanish as needed.
Since $\overline{X'} \setminus X'_0 \to \A^{d-1}$ is finite so is $C_2' \setminus C_0' \to S$; from this one deduces that $D'$ is finite over $S$.\footnote{Note that if $A \to B$ is finite, $B \subset \overline{B}$ is a dense open immersion, and $A \to \overline{A} \to \overline{B}$ is a compactification, then $A \to \overline{A}|_B$ is both closed ($A \to B$ being finite) and also a dense open immersion, whence an isomorphism. That is, the $\overline{A} \setminus A$ lies completely over $\overline{B} \setminus B$.}
The natural maps $U \to C_0$ and $U' \to C_0'$ factor through $C$ and $C'$ respectively, since the images of the closed points $T$ do.
It follows that the theorem is proved, up to constructing $d$.

First note that $D_0 := \overline{C} \setminus C_0$ is finite over $S$.
Indeed it is proper, so we need only establish quasi-finiteness; but $C_1 \setminus C_0 \to S$ is finite by construction and so is $\overline{C} \setminus C_1 \to (\P^1_S \setminus \A^1_S) \wequi S$, as needed.\footnote{See the previous footnote.}
For a closed point $s \in S$, let $R$ be a connected component of dimension $1$ of $\overline{C}_s$.
Since $(D_0)_s \to s$ is finite, it cannot contain all of $R$; let $x_R \in R \setminus D$.
Now pick $d$ such that $d|_{D_0} = 0$, $d|_{\scr Z} \ne 0$ and $d(x_R) \ne 0$ for all such $(R,s)$ (of which there are only finitely many).
This satisfies the required properties (the only non-trivial claim is that $Z(d) \to S$ is finite. But using properness and semicontinuity of fiber dimension \cite[Tag 0D4I]{stacks-project}, it suffices to prove quasi-finiteness over the closed points, which we have ensured).
\end{proof}

\subsection{Injectivity for semilocal schemes} We now verify that any framed pretheory satisfy IL (in fact we prove a somewhat stronger property).
\begin{lemma} \label{lemm:inj-H}
Let $k$ be an infinite field, $X \in \SmAff_k$, $Z \subset X$ closed and nowhere dense, $x_1, \dots, x_n \in Z$, $U$ the semilocalization of $X$ in the $x_i$.
Then there are curve correspondences $\Phi, \Phi^-: U \rightsquigarrow X \setminus Z$, $\Theta: U \times \A^1 \rightsquigarrow X$ such that for any framed pretheory, \[ i_0^*\Theta^* = \Phi^* \circ (X \setminus Z \to X)^* \] and \[ i_1^*\Theta^* = \tw \circ (U \to X)^* + (\Phi^-)^* \circ (X \setminus Z \to X)^*, \] where $\tw$ is some automorphism of $F(U)$.
\end{lemma}
\begin{proof}
We first show that we may assume that $x_1, \dots, x_n \in X$ are closed.
Indeed if not, pick closed specializations $y_1, \dots, y_n$.
Note that $y_i \in Z$ ($Z$ being closed) and $x_i \in X_{y_1, \dots, y_n}$ ($X_{y_1, \dots, y_n}$ being an intersection of open subsets containing $y_i$).
Applying the claim with the $y_i$ in place of the $x_i$ yields curve correspondences $\Phi', \Theta'$ over $X_{y_1, \dots, y_n}$.
Pulling them back along $X_{x_1, \dots, x_n} \to X_{y_1, \dots, y_n}$ yields the desired result.

Hence from now on we assume that the $x_i$ are closed.
Apply Proposition \ref{prop:M-redux} and Remark \ref{rmk:M-1} to the identity map $(X, Z) \to (X, Z)$, with $T=\{x_1, \dots, x_n\}$.
We hence obtain a diagram \[ X \xleftarrow{v} C \xrightarrow{j} \overline{C} \xrightarrow{p} U \] where $\overline C \to U$ is a projective curve with ample line bundle $\scr O(1)$, $D = \overline{C} \setminus Z$ is given by $Z(d)$ for some $d \in \scr O(1)$ and is finite over $U$, $\Omega^1_{C/U}$ is trivial, $Z' := v^{-1}(Z)$ is finite over $U$, $j$ is an open immersion and $C$ is smooth over $U$.

Note that $\Delta: U \to C$ is a regular immersion of codimension $1$, and hence $\Delta \subset \overline C$ is a divisor.
In particular $\scr O(-\Delta)$ (the ideal sheaf defining $\Delta$) is a line bundle on $\overline C$, with inverse $\scr O(\Delta)$.
We shall show at the end that for $n$ sufficiently large, we can find sections $s \in H^0(\overline C, \scr O(n))$, $s'\in H^0(\overline C, \scr O(n) \otimes \scr O(-\Delta))$  satisfying the following:
\begin{itemize}
\item $s|_{Z'}$ and $s'|_{Z' \cup D \cup \Delta}$ are non-vanishing
\item $s|_D = s' \otimes \delta$, where $\delta \in H^0(\overline{C}, \scr O(\Delta))$ defines $\Delta$
\end{itemize}
Set $\tilde s = (1-t)s + t s' \otimes \delta$.
Since $\tilde s$ is constantly non-zero on $D \times \A^1$, $Z(\tilde s) \subset C \times \A^1$ and so this is affine and proper, whence finite, over $U \times \A^1$.
Similarly $Z(s), Z(s')$ are finite over $U$.
Note that $Z(s), Z(s') \subset C \setminus Z'$, and $Z(s' \otimes \delta) = Z(s') \amalg Z(\delta)$ (since $Z(\delta) = \Delta$ and so $Z(s') \cap Z(\delta) = \emptyset$).
Pick an isomorphism $\mu: \Omega^1_{C/U} \wequi \scr O_C$.
Put \[ \Theta = (U \times \A^1 \leftarrow C \times \A^1 \xrightarrow{v} X, \tilde s/d^n, \mu, Z(\tilde s)), \] \[ \Phi = (U \leftarrow C \setminus Z' \xrightarrow{v} X \setminus Z, s/d^n, \mu, Z(s)), \] \[ \Phi^- = (U \leftarrow C \setminus Z' \xrightarrow{v} X \setminus Z, s' \otimes \delta/d^n, \mu, Z(s')). \]
Then by construction $i_0^* \Theta^* = \Phi^*\circ (X \setminus Z \to X)^*$ and \begin{gather*} i_1^*\Theta^* = (U \leftarrow C \xrightarrow{v} X, s' \otimes \delta/d^n, \mu, Z(s' \otimes \delta))^* \\ = (U \leftarrow C \xrightarrow{v} X, s' \otimes \delta/d^n, \mu, Z(\delta))^* + (\Phi^-)^* \circ (X \setminus Z \to X)^*. \end{gather*}
Since $Z(\delta) \to U$ is an isomorphism, the first term is $\tw(s' \otimes \delta/d^n)^\mu_{Z(\delta)} \circ (U \wequi \Delta \to X)^*$ by Definition \ref{def:framed-preth}(4).
We conclude since $U \wequi \Delta \to X$ is the canonical map, by construction.

It remains to construct $s, s'$.
For $n$ large enough, both \[ H^0(\overline C, \scr O(n) \otimes \scr O(-\Delta)) \to H^0(Z' \cup D \cup \Delta, \scr O(n) \otimes \scr O(-\Delta)) \] and \[ H^0(\overline C, \scr O(n)) \to H^0(Z' \amalg D, \scr O(n)) \] are surjective (see \S\ref{subsec:serre}).
Since $Z' \cup D \cup \Delta$ is semilocal (being proper and quasi-finite, hence finite, over $U$), $\scr O(n) \otimes \scr O(-\Delta)$ admits a non-vanishing section section on it (see \S\ref{subsec:semilocal}); let $s'$ be any lift thereof.
Note that $H^0(Z' \amalg D, \scr O(n)) \wequi H^0(Z', \scr O(n)) \times H^0(D, \scr O(n))$; let $s$ be any lift of $(s' \otimes \delta|_{Z'}, 1)$, where $1 \in H^0(Z', \scr O(n))$ is a non-vanishing section.
The required properties hold by construction.
\end{proof}

\begin{theorem} \label{thm:inj}
Let $U$ be a semilocal scheme, essentially smooth over an infinite field $k$.
Let $Z \subset U$ be a closed subscheme not containing any connected component of $U$.
Then for any $\A^1$-invariant framed pretheory $F$, the restriction $F(U) \to F(U \setminus Z)$ is injective.
\end{theorem}
\begin{proof}
Since $U$ is semilocal, it has only finitely many connected components.\NB{each connected component contains a closed point}
Since $F(A \amalg B) \wequi F(A) \times F(B)$, we may argue separately for each connected component of $U$; hence we may assume that $U$ is connected.
If $U$ has only one point, the result is trivial.
We may thus assume that the subset $Z_0 \subset U$ of closed points is a proper closed subscheme.
Replacing $Z$ by $Z \cup Z_0$, we may assume that $Z$ contains all closed points.
Replacing $Z$ by a larger proper closed subscheme, we may also assume that $Z$ is finitely presented (e.g. principal).
Write $U = \lim_i V_i$, where $V_i \in \SmAff_k$.
Replacing $V_i$ by its single connected component containing the image of $U$, we may assume each $V_i$ is connected.
Since $Z$ is a finitely presented closed subscheme, without loss of generality we may assume that $Z = Z_0 \times_{V_0} U$, where $Z_0 \subset V_0$ is closed.
If $Z_i := Z_0 \times_{V_0} V_i$ contains all of $V_i$, then $Z$ contains all of $X$, which is not the case.
It follows that $Z_i$ is nowhere dense in $V_i$, for each $i$.
Let $x_1^{(i)}, \dots, x_n^{(i)} \in Z_i$ denote the images of the closed points of $X$.
Let $U_i = (V_i)_{x_1^{(i)}, \dots, x_n^{(i)}}$ be the semilocalization (see \S\ref{subsec:semilocal}).
By continuity, it will suffice to prove that $F(U_i) \to F(U_i \setminus Z_i)$ is injective for each $i$.

In other words we may assume that $U = V_{x_1, \dots, x_n}$ where $V$ is a smooth affine scheme, $Z \subset V$ nowhere dense, $x_i \in Z$.
Let $X \subset V$ be an open affine neighborhood of the $x_i$.
Lemma \ref{lemm:inj-H} shows that \[ \ker(F(X) \to F(X \setminus Z) ) \subset \ker(F(X) \to F(U)). \]
Indeed if $x \in F(X)$ with $x|_{X \setminus Z} = 0$, then \[ 0 = \Phi^*(x|_{X \setminus Z}) - (\Phi^-)^*(x|_{X \setminus Z}) = \Phi^*(x|_{X \setminus Z}) - i_0^*\Theta^*(x) = \Phi^*(x|_{X \setminus Z}) - i_1^*\Theta^*(x) = -  \tw(x|_{U}) \] and so $x|_U = 0$, $\tw$ being invertible.
Now taking the (filtered, whence exact) colimit over all such $X$ we obtain the desired result.
\end{proof}

\begin{corollary} \label{cor:IL}
Let $U$ be a semilocal connected scheme, essentially smooth over an infinite field $k$.
Write $\eta \in U$ for the generic point.
Then for any $\A^1$-invariant framed pretheory $F$, the map $F(X) \to F(\eta)$ is injective.
In particular $F$ satisfies IL.
\end{corollary}
\begin{proof}
By Theorem \ref{thm:inj}, $F(U) \to F(V)$ is injective for every non-empty open affine subscheme $V \subset U$.
The result follows by taking the filtered colimit over all such $V$.
\end{proof}

\subsection{Excision on the relative affine line} We now proceed with EA. 
\begin{lemma} \label{lemm:EA-smooth}
Let $U \in \SmAff_k$, $V \subset \A^1_U$ open, $0_U \subset V$.
Write $i: (V, V \setminus 0) \to (\A^1_U, \A^1_U \setminus 0)$ for the open immersion of pairs.
There exist curve correspondences of pairs \[ \Phi, \Psi: (\A^1_U, \A^1_U \setminus 0) \rightsquigarrow (V, V \setminus 0), \] \[ \Theta_1: (\A^1_U, \A^1_U \setminus 0) \times \A^1 \rightsquigarrow (\A^1_U, \A^1_U \setminus 0), \] \[ \Theta_2: (V, V \setminus 0) \times \A^1 \rightsquigarrow (V, V \setminus 0) \] such that for any framed pretheory
\begin{enumerate}
\item $i_0^* \Theta_1^* = \Phi^*i^*$, $i_1^* \Theta_1^*$ is invertible,
\item $i_0^* \Theta_2^* = i^* \Psi^*$, $i_1^* \Theta_2^*$ is invertible.
\end{enumerate}
\end{lemma}
Note that here we are using the pullback along a curve correspondence of pairs from Construction \ref{cons:pair-pullback}; thus for example $\Phi^*$ is a map \[ \Phi^*: F(V \setminus 0)/F(V) \to F(\A^1_U \setminus 0)/F(\A^1_U). \]
\begin{proof}
(1) We shall construct sections $s \in H^0(\P^1 \times \A^1_U, \scr O(n))$ and $s' \in H^0(\P^1 \times \A^1_U, \scr O(n-1))$, for some $n > 0$, satisfying the following properties.
Denote the coordinate on $\A^1$ by $x$ and on $\P^1$ by $y=(Y_0:Y_1)$.
Let $\delta = Y_1 - xY_0 \in H^0(\P^1 \times \A^1_U, \scr O(1))$; observe that $Z(\delta)$ defines the diagonal $\{x=y\} \hookrightarrow \P^1 \times \A^1_U$.
Let $D = \P^1_U \setminus V$.
We shall ensure that:
\begin{itemize}
\item $s|_{D \times \A^1}$, $s'|_{0 \times \A^1_U}$, $s'|_{Z(\delta)}$ and $s'|_{\infty \times \A^1_U}$ are all non-vanishing,
\item $s|_{0 \times \A^1_U} = \delta s'$, and
\item $s|_{\infty \times \A^1_U} = \delta s'$.
\end{itemize}
Put $\tilde s = (1-t)s + t\delta s' \in H^0(\P^1 \times \A^1_U \times \A^1)$.
Now consider the function $f = s/Y_0^n$ on $V \times \A^1 \subset \P^1 \times \A^1_U$ and the function $\tilde f = \tilde s/Y_0^n$ on $\A^3_U \subset \P^1 \times \A^1_U \times \A^1$.
We claim that the following are curve correspondences of pairs \[ \Phi = (\A^1_U \xleftarrow{pr_x} V \times \A^1 \xrightarrow{pr_y} V, f, dy, Z(f)), \] \[ \Theta_1 = (\A^1_U \times \A^1 \xleftarrow{pr_x} \A^3_U \xrightarrow{pr_y} \A^1_U, \tilde f, dy, Z(\tilde f)), \] satisfying the required properties.

To begin with, since $\tilde s$ is constantly non-zero over $\infty$, we find that $Z(\tilde s) = Z(\tilde f)$.
In particular this is both proper and affine, whence finite, over $\A^1_U \times \A^1$.
Similarly $Z(f) = Z(s)$ is finite over $\A^1_U$.
Further $\tilde s|_{y=0}$ is constantly equal to $\delta s'$, which vanishes there only if $y=x$, i.e. $x=0$.
It follows that $\Theta_1$ and $\Phi$ are well-defined curve correspondences of pairs as displayed in the proposition.
By construction we have $i_0^* \Theta_1^* = \Phi^* i^*$.
On the other hand, $Z(\delta) \cap Z(s') = \emptyset$ by assumption and thus $Z(\delta s') = Z(\delta) \amalg Z(s')$, so we get \begin{gather*} i_1^* \Theta_1^* = (\A^1_U \xleftarrow{pr_x} \A^2_U \xrightarrow{pr_y} \A^1_U, \delta s'/Y_0^n, dy, Z(\delta))^* \\ + (\A^1_U \xleftarrow{pr_x} \A^2_U \xrightarrow{pr_y} \A^1_U, \delta s'/Y_0^n, dy, Z(s'))^*. \end{gather*}
The first term is invertible by Definition \ref{def:framed-preth}(4), since $Z(\delta)$ maps isomorphically to $\A^1_U$ via both $x$ and $y$.
The second term vanishes by Lemma \ref{lemm:almost-null}, since $Z(s') \subset (\A^1_U \setminus 0) \times \A^1$ by construction. We have thus proved (1) up to constructing $s$ and $s'$.

We now construct $s$ and $s'$.
By Serre's theorem (see \S\ref{subsec:serre}) we may ensure that \[s|_{D \times \A^1} = Y_1^n, \quad s|_{0 \times \A^1_U} = Y_0^{n-1}\delta, \quad s'|_{\infty \times \A^1_U} = Y_1^{n-1}, \quad s'|_{0 \times \A^1_U} = Y_0^{n-1}, \quad s'|_{Z(\delta)} = Y_0^{n-1}. \]
Since $Y_1$ only vanishes at $0 \not\in D$, $s|_{D \times \A^1}$ is non-vanishing.
The other non-vanishing conditions hold for similar reasons.
Since $\delta = Y_1$ at $\infty$ (i.e. $Y_0=0$), $s = s' \delta$ there.
The agreement at $0$ holds by construction.

(2) The argument is very similar.
One constructs sections $s \in H^0(\P^1 \times \A^1_U, \scr O(n))$ and $s' \in H^0(\P^1 \times V, \scr O(n-1))$ such that:
\begin{itemize}
\item $s|_{D \times V}$, $s'|_{0 \times V}$, $s'|_{Z(\delta)}$ and $s'|_{D \times V}$ are is non-vanishing,
\item $s|_{0 \times V} = \delta s'$, and
\item $s|_{D \times V} = \delta s'$.
\end{itemize}
This is done by using Serre's theorem to ensure that \[s|_{D \times V} = Y_1^n, \quad s|_{0 \times V} = Y_0^{n-1}\delta, \quad s'|_{D \times V} = Y_1^n \delta^{-1}, \quad s'|_{0 \times V} = Y_0^{n-1}, \quad s'|_{Z(\delta)} = Y_0^{n-1}, \] and arguing as before.
Put $\tilde s = (1-t)s + t\delta s' \in H^0(\P^1 \times V \times \A^1)$.
Arguing as before that this is well-defined, we obtain curve correspondences of pairs \[ \Psi = (\A^1_U \xleftarrow{pr_x} V \times \A^1 \xrightarrow{pr_y} V, s/Y_0^n, dy, Z(s)) \] and \[ \Theta_2 = (V \times \A^1 \xleftarrow{pr_x} V \times_U V \times \A^1 \xrightarrow{pr_y} V, \tilde s/Y_0^n, dy, Z(\tilde s)). \]
One checks as before that these satisfy the required properties.
\end{proof}

We also have the following variant.
\begin{lemma} \label{lemm:EA-field}
Let $K$ be a field, $z \in \A^1_K$ closed and $V \subset \A^1_K$ an open neighbourhood of $z$.
Write $i: (V, V \setminus z) \to (\A^1_K, \A^1_K \setminus z)$ for the open immersion of pairs.
There exist curve correspondences of pairs \[ \Phi, \Psi: (\A^1_K, \A^1_K \setminus z) \rightsquigarrow (V, V \setminus z), \] \[ \Theta_1: (\A^1_K, \A^1_K \setminus z) \times \A^1 \rightsquigarrow (\A^1_K, \A^1_K \setminus z), \] \[ \Theta_2: (V, V \setminus z) \times \A^1 \rightsquigarrow (V, V \setminus z) \] such that for any framed pretheory
\begin{enumerate}
\item $i_0^* \Theta_1^* = \Phi^*i^*$, $i_1^* \Theta_1^*$ is invertible,
\item $i_0^* \Theta_2^* = i^* \Psi^*$, $i_1^* \Theta_2^*$ is invertible.
\end{enumerate}
\end{lemma}
\begin{proof}
The proofs is almost the same as for Lemma \ref{lemm:EA-smooth}.
Let $d$ be the degree of $z$; then there exists a section $\nu \in H^0(\P^1 \times \A^1_K, \scr O(d))$ such that $Z(\nu) = z \times \A^1$.
Now replace $\scr O(1)$ by $\scr O(d)$, $t_1$ by $\nu$ and $t_0$ by $t_0^d$ in the previous argument.\todo{details?}
\end{proof}

We can use this to prove EA.
\begin{theorem} \label{thm:EA}
Let $U$ be essentially smooth, affine over a field $k$, $V \subset \A^1_U$ an open subscheme containing $0_U$.
Let $F$ be an $\A^1$-invariant framed pretheory.
Then restriction induces \[ F(\A^1_U \setminus 0_U)/F(\A^1_U) \wequi F(V \setminus 0_U)/F(V). \]

Similarly if $K$ is a field, $z \in \A^1_K$ is closed and $V$ is an open neighbourhood, then  \[ F(\A^1_K \setminus z)/F(\A^1_K) \wequi F(V \setminus z)/F(V). \]
\end{theorem}
\begin{proof}
Since $V \to \A^1_U$ is affine it is quasi-compact \cite[Tag 01K4]{stacks-project} and hence of finite presentation \cite[Tag 01TU]{stacks-project}.
It follows that when writing $U = \lim_i U_i$ as a cofiltered limit of smooth affine schemes, we may assume given $V_0 \subset \A^1_{U_0}$ affine open with base change $V$ \cite[Tags 01ZM and 0EUU]{stacks-project}.
By continuity of $F$, we may thus assume that $U$ is smooth over $k$.
The first statement now follows from Lemma \ref{lemm:EA-smooth}.
(Recall that if $i: A \to B, \Phi, \Psi: B \to A$ are maps of sets such that $\Phi \circ i: A \to A$ and $i \circ \Psi: B \to B$ are invertible, then $A \xrightarrow{i} B \xrightarrow{\Phi} A$ is injective so $i$ is injective, and $B \xrightarrow{\Psi} A \xrightarrow{i} B$ is surjective so $i$ is surjective, i.e., $i$ is invertible.)

The second statement is immediate from Lemma \ref{lemm:EA-field}.
\end{proof}

\subsection{Étale excision} Finally we treat EE.
\begin{lemma} \label{lemm:EE-smooth}
Let $X \in \SmAff_k$, $Z \subset X$ a closed subscheme and $\pi: (X',Z') \to (X,Z) \in \SmAff_k$ an étale neighborhood.
Let $z' \in Z'$ and put $z = \pi(z') \in Z$.
Let $U = X_z$, $U' = X'_{z'}$.

Write $i: (U, U \setminus Z) \to (X, X \setminus Z)$ and $i': (U', U' \setminus Z') \to (X', X' \setminus Z')$ for the open immersions of pairs.
There exist curve correspondences of pairs \[ \Phi, \Psi: (U, U \setminus Z) \rightsquigarrow (X', X' \setminus Z'), \] \[\Theta_1: (U, U \setminus Z) \times \A^1 \to (X, X \setminus Z) ,\] \[\Theta_2: (U', U' \setminus Z') \times \A^1 \to (X', X' \setminus Z') \] such that for any framed pretheory
\begin{enumerate}
\item $i_0^* \Theta_1^* = \Phi^* \pi^*$, $i_1^* \Theta_1^* = \tw \circ i^*$, where $\tw$ is some automorphism of $F(U)$
\item $i_0^* \Theta_2^* = \pi^* \Psi^*$, $i_1^* \Theta_2^* = i'^*$.
\end{enumerate}
\end{lemma}
\begin{proof}
Shrinking $X'$ if necessary, we may assume that $Z' = \pi^{-1}(Z)$.

(1) We apply Proposition \ref{prop:M-redux} and Remark \ref{rmk:M-1} and hence obtain a diagram in notation as stated there (with $S=U$).
We shall (at the end) find sections $s \in H^0(\overline C, \scr O(n))$ and $s' \in H^0(\overline C, \scr O(n) \otimes \scr O(-\Delta))$ satisfying the following:
\begin{itemize}
\item $s'|_{D \cup \scr Z \cup \Delta}$ is non-vanishing
\item $s|_{D \cup \scr Z} = s' \otimes \delta$, where $\delta \in H^0(\overline C, \scr O(\Delta))$ defines $\Delta$
\item $Z(s) = Z_0 \amalg Z_1$ where $\varpi$ is an étale neighborhood of $Z_0$ and $Z_1 \cap \scr Z = \emptyset$.
\end{itemize}
Put \[ \tilde s = (1-t)s + t\delta \otimes s' \in H^0(\overline{C} \times \A^1, \scr O(n)). \]
Then $\tilde s$ is constantly non-vanishing on $D \times \A^1$.
In particular $Z(\tilde s) \subset C \times \A^1$ is finite (since it is proper and affine) over $U \times \A^1$.
Moreover $\tilde s$ is constantly equal to $\delta \otimes s'$ on $\scr Z \times \A^1$.
It follows that $Z(\tilde s) \cap \scr Z = \Delta \cap \scr Z$ lies over $Z_z \subset U$.
Choose a trivialization $\mu: \Omega^1_{C/U} \wequi \scr O_C$.
Consider \[ \Theta_1 = (U \times \A^1 \leftarrow C \times \A^1 \xrightarrow{v} X, \tilde s/d^n, \mu, Z(\tilde s)): (U, U \setminus Z) \times  \A^1 \rightsquigarrow (X, X \setminus Z). \]
What we have said so far shows that this is a well-defined curve correspondence of pairs.
We get \[ i_1^* \Theta_1^* = (U \leftarrow C \to X, \delta \otimes s'/d^n, \mu, Z(\delta))^* + (U \leftarrow C \to X, \delta \otimes s'/d^n, \mu, Z(s'))^*. \]
Since $Z(s') \cap \scr Z = \emptyset$, the second term vanishes by Lemma \ref{lemm:almost-null}.
Since $U \wequi Z(\delta) \to X$ is the canonical map, the first term is $\tw(\delta \otimes s'/d^n, \mu, Z(s')) \circ i^*,$ as needed.
Similarly \[ i_0^*  \Theta_1^* = (U \leftarrow C \to X, s/d^n, \mu, Z_0)^* + (U \leftarrow C \to X, s/d^n, \mu, Z_1)^*. \]
Since $Z_1 \cap \scr Z = \emptyset$, the second term vanishes.
On the other hand by construction $C' \to C$ is an étale neighborhood of $Z_0$.
Let \[ \Phi = (U \leftarrow C' \xrightarrow{v'} X', s/d^n, \mu, Z_0). \]
This has the required property, by Definition \ref{def:framed-preth}(2).

It remains to construct $s, s'$.
Since $\Delta \subset C$ is an effective Cartier divisor, and $\scr Z$ is semilocal, $\Delta \cap \scr Z \to \scr Z$ principal, and hence so is (the isomorphic map) $\Delta_Z' \to \scr Z'$.
Applying Lemma \ref{lemm:find-Z0}, we obtain an effective divisor $Z(\xi) =: Z_0 \subset C \subset \overline C$ (finite over $U$) such that $\varpi$ is an étale neighborhood and $Z_0 \cap \scr Z = \Delta \cap \scr Z$ (as schemes).
Let $\zeta \in \scr O(-Z_0)$ be the section defining $Z_0$.
Pick $n$ large enough such that both of the maps \[ H^0(\overline C, \scr O(n) \otimes \scr O(-Z_0)) \to H^0(\scr Z \cup D, \scr O(n) \otimes \scr O(-Z_0)) \] and  \[ H^0(\overline C, \scr O(n) \otimes \scr O(-\Delta)) \to H^0(\scr Z \cup D \cup \Delta, \scr O(n) \otimes \scr O(-\Delta)) \] are surjective.
Since $\scr Z \cup D$ is semilocal, $\scr O(n) \otimes \scr O(-Z_0)$ admits a non-vanishing section on it.
Let $\zeta'$ be a lift of such a non-vanishing section to $\overline C$ and put $s = \zeta \otimes \zeta'$.
By construction $s|_{\scr Z \amalg D} = s_0 \otimes \delta$, for some non-vanishing section $s_0 \in H^0(\scr Z \cup D, \scr O(n) \otimes \scr O(-\Delta))$.
We may find a non-vanishing section $s_1 \in H^0(\scr Z \cup D \cup \Delta, \scr O(n) \otimes \scr O(-\Delta))$ extending $s_0$ (see \S\ref{subsec:semilocal}); finally let $s' \in H^0(\overline C, \scr O(n) \otimes \scr O(-\Delta))$ be any lift of $s_1$.
The required properties hold by construction.

(2) We apply Proposition \ref{prop:M-redux} and Remark \ref{rmk:M-2} and hence obtain a diagram in notation as stated there.
Let $\lambda \in H^0(\Delta'', \scr O(-\Delta''))$ be a generator as in Definition \ref{def:framed-preth}(5) (with $Z=\Delta'' \subset C''$).
We shall at the end find sections $s \in H^0(\overline{C'}, \scr O(n))$ and $s' \in H^0(\overline{C''}, \scr O(n) \otimes \scr O(-\Delta''))$ such that:
\begin{itemize}
\item $s'|_{\scr Z''}$ and $s|_{D'}$ are non-vanishing
\item $s'|_{\Delta''} = d^n \lambda$
\item $s'|_{D''} = \delta^{-1} \overline{\varpi''}^*(s)$, where $\delta \in H^0(\overline{C''}, \scr O(\Delta''))$ defines $\Delta''$
\item $\overline{\varpi''}^*(s)|_{\scr Z''} = s' \otimes \delta$.
\end{itemize}
Put $\tilde s = (1-t)\overline{\varpi''}^*(s) + t \delta \otimes s'$.
As before $\tilde s$ is constantly non-vanishing over $D'' \times \A^1$ and hence $Z(\tilde s) \subset C''$ is finite over $U' \times \A^1$.
Also $\tilde s$ is constantly equal to $s' \otimes \delta$ on $\scr Z'' \times \A^1$, and so $Z(\tilde s) \cap \scr Z''$ lies over $Z'_{z'} \subset U'$.
We thus obtain a curve correspondence of pairs \[ \Theta_2 = (U' \times \A^1 \leftarrow C'' \times \A^1 \xrightarrow{v''} X', \tilde s/d^n, \mu, Z(\tilde s)): (U', U' \setminus Z') \times \A^1 \rightsquigarrow (X', X' \setminus Z'). \]
Similarly we obtain \[ \Psi = (U \leftarrow C \xrightarrow{v'} X', s/d^n, \mu, Z(s)): (U, U \setminus Z) \to (X', X' \setminus Z'). \]
Arguing as before we see that $i_0^* \Theta_2^* = \pi^*\Psi^*$ and \[ i_1^* \Theta_2^* = \tw(\delta \otimes s'/d^n)^\mu_{Z(\delta)} \circ i^*. \]
Since $Z(\delta) = \Delta''$ and $d(\delta \otimes s'/d^n) = \lambda$\NB{details?}, by construction (see Definition \ref{def:framed-preth}(5)) we have $\tw(\delta \otimes s'/d^n)^\mu_{Z(\delta)} = \id$, as needed.

It thus remains to construct $s,s'$.
Write $Z_1 = \scr Z'' \cap \Delta''$ and $Z_2$ for its image in $C'$.
Then $Z_1 \wequi Z_2 \wequi Z_z$, and $Z_1 \xrightarrow{\wequi} Z_2 \times_U U'$.
The closed immersion $Z_2 \to \scr Z'$ is isomorphic to $\Delta \cap \scr Z \to \scr Z$, whence locally principal, and so principal since $\scr Z'$ is semilocal.
Let $\rho \in \scr O(\scr Z')$ cut out $Z_2$, so that $\overline{\varpi}^*(\rho)$ cuts out $Z_1$.
We may thus write $\overline{\varpi}^*(\rho) = \rho' \otimes \delta$, with $\rho' \in H^0(\scr Z'', \scr O(-\Delta))$ a generator.
Now $\lambda|_{Z_1}$ and $\rho'|_{Z_1}$ both generate $\scr O(-\Delta)|_{Z_1}$ and hence differ by a unit.
Since $\scr O^\times(\scr Z') \to \scr O^\times(Z_2) \wequi \scr O^\times(Z_1)$ is surjective ($\scr Z'$ being semilocal), we may multiply $\rho$ by a unit and so assume that $\lambda|_{Z_1} = \rho'|_{Z_1}$.
Since $\scr Z'' \cup \Delta''$ is the pushout in schemes of $\scr Z'' \leftarrow \scr Z'' \cap \Delta \to \Delta''$ \cite[Tag 0C4J]{stacks-project}, there exists $\tilde \lambda \in H^0(\scr Z'' \cup \Delta'', \scr O(-\Delta''))$ such that $\tilde \lambda|_{\scr Z''} = \rho'$ and $\tilde \lambda|_{\Delta''} = \lambda$.
Choose $n$ large enough such that \[ H^0(\overline{C'}, \scr O(n)) \to H^0(D' \amalg \scr Z', \scr O(n)) \] and \[ H^0(\overline{C''}, \scr O(n) \otimes \scr O(-\Delta'')) \to H^0(D'' \amalg (\Delta'' \cup \scr Z''), \scr O(n) \otimes \scr O(-\Delta'')) \] are surjective.
Let $s'$ be a lift of $(1,\rho d^n)$, where $1 \in H^0(D', \scr O(n))$ is a non-vanishing section.
Let $s$ be a lift of $(\delta^{-1} \cdot \overline{\varpi}^*(1),\tilde \lambda d^n)$.
The required properties hold by construction.
\end{proof}

\begin{theorem} \label{thm:EE}
Let $k$ be an infinite field and $\pi: U' \to U$ a cofiltered limit of étale morphisms of smooth $k$-schemes.
Assume that $U', U$ are local schemes and $\pi$ is a local morphism.
Let $Z' \subset U', Z \subset U$ be finitely presented closed subschemes such that $\pi$ induces an isomorphism of $Z'$ onto $Z$.
Let $F$ be an $\A^1$-invariant framed pretheory.
Then $\pi^*$ induces \[ F(U \setminus Z)/F(U) \wequi F(U' \setminus Z')/F(U'). \]
\end{theorem}
\begin{proof}
Since $Z, Z'$ are finitely presented we may without loss of generality assume that $\pi = \lim_\alpha \pi_\alpha$, where $\pi_\alpha: (X_\alpha', Z_\alpha') \to (X_\alpha, Z_\alpha)$ is an étale neighborhood of smooth affine $k$-schemes, and $Z = \lim_\alpha Z_\alpha, Z' = \lim_\alpha Z_\alpha'$.
Write $z_\alpha, z'_\alpha$ for the images in $X_\alpha, X_\alpha'$ of the closed point.
Set $U_\alpha = (X_\alpha)_{z_\alpha}$, $U_\alpha' = (X_\alpha')_{z_\alpha'}$ and write $\overline{\pi}_\alpha: U_\alpha' \to U_\alpha$ for the restriction of $\pi_\alpha$.
Consider the commutative diagram
\begin{equation*}
\begin{CD}
F(U_\alpha \setminus Z_\alpha)/F(U_\alpha) @>{\overline{\pi}^*_\alpha}>> F(U_\alpha' \setminus Z_\alpha')/F(U_\alpha') \\
@A{i^*_\alpha}AA                         @A{i'^*_\alpha}AA \\
F(X_\alpha \setminus Z_\alpha)/F(X_\alpha)       @>{\pi^*_\alpha}>> F(X_\alpha' \setminus Z_\alpha')/F(X_\alpha').
\end{CD}
\end{equation*}
Lemma \ref{lemm:EE-smooth} yields equations \[ \Phi_\alpha^* \pi^*_\alpha = \tw \circ i^*_\alpha \quad\text{and}\quad \overline{\pi}^*_\alpha \Psi^*_\alpha = i'^*_\alpha. \]
These show that \begin{gather*} \ker(\pi^*_\alpha) \subset \ker(i^*_\alpha) \quad\text{and}\quad \cok(i'^*_\alpha) \twoheadrightarrow \cok(\overline{\pi}^*_\alpha). \end{gather*}
Taking the colimit over all $\alpha$ concludes the proof (noting that $\lim_\alpha i_\alpha: \lim_\alpha U_\alpha \to \lim_\alpha X_\alpha$ is an isomorphism, and similarly for $i'$).
\end{proof}

\subsection{Conclusion}
We have proved the following.
\begin{theorem} \label{thm:axioms}
Any $\A^1$-invariant framed pretheory (see Definition \ref{def:framed-preth}) over an infinite field satisfies the axioms IA, EA, IL, and EE of \S\ref{subsec:formalism}.
\end{theorem}
\begin{proof}
Combine Theorems \ref{thm:IA}, \ref{thm:EA} and \ref{thm:EE}, and Corollary \ref{cor:IL}.
\end{proof}

We can now prove the main theorem.
\begin{proof}[Proof of Theorem \ref{thm:strictly-invariant}.]
We first consider the case where $k$ is infinite.
We know that the forgetful functor $\PSh_\Sigma(\Cor^\fr(k)) \to \PSh_\Sigma(\Sm_k)$ commutes with $\Omega_\Gm$ (for trivial reasons) and $L_\Nis$ (by \cite[Proposition 3.2.14]{EHKSY1}).
Thus if $F \in \PSh_\Sigma(\Cor^\fr(k), \Ab)$, then so are $F_{-1}$ and $H^i(-;F)$.
Suppose that $F$ is $\A^1$-invariant.
Then $H^i(\ph, F)$ is $\A^1$-invariant by Theorem \ref{thm:strict-A1-formal} (which applies because of Theorem \ref{thm:axioms} and Example \ref{ex:framed}), provided $k$ is perfect.
For $H^0$ we do not need perfectness; see Remark \ref{rmk:H0-imperf}.
The fact that $F$ coincides with its sheafification on open subsets of $\A^1$ is Lemma \ref{lemm:strict-formal-field}.

Now let $k$ be finite and $F \in \PSh_\Sigma(\Cor^\fr(k), \SH)$ be $\A^1$-invariant.
It suffices to prove that $F \to L_\mot F$ is a Nisnevich local equivalence (indeed then $L_\Nis F \wequi L_\mot F$ is $\A^1$-invariant).
That is, we must prove that $F(X) \wequi (L_\mot F)(X)$ for any $X$ which is essentially smooth, henselian local over $k$.
Arguing as in \cite[Corollary B.2.5]{EHKSY1}, for this it suffices to prove that if $k'/k$ is an infinite perfect extension of $k$, then $F(X_{k'}) \wequi (L_\mot F)(X_{k'})$.
Since $X_{k'}$ is a finite disjoint union of henselian local schemes, this follows (using that $L_\mot$ commutes with essentially smooth base change \cite[Lemma A.4]{hoyois-algebraic-cobordism}) from $L_\Nis(F|_{\Sm_{k'}}) \wequi L_\mot (F|_{\Sm_{k'}})$, which we have already established.
\end{proof}

\bibliographystyle{alpha}
\bibliography{references}

\end{document}